\documentclass[12pt,dvipdfmx,a4paper]{amsart}

\selectfont

\usepackage{epic,eepic,ecltree}

\usepackage{graphicx}

\usepackage{amssymb,color, euscript, enumerate}
\usepackage{amsthm}
\usepackage{amsmath}
\usepackage{braket}
\usepackage{amscd}
\usepackage{txfonts}
\usepackage{comment}
\usepackage{bm}
\usepackage{amscd}

\numberwithin{equation}{section}

\catcode`,\active

\catcode`\,12

\newcommand{\SU}{\mathrm{SU}(2)}
\newcommand{\OS}{O_{\mathrm{SU}(2)}}
\newcommand{\OO}{O_{\mathrm{SO}(3)}}

\newcommand{\PB}{\mathbb{PB}}

\newcommand{\lla}{\langle\langle}
\newcommand{\rra}{\rangle\rangle}

\newcommand{\lP}{P}
\newcommand{\lQ}{Q}
\newcommand{\lD}{{\lP/\lQ}}

\newcommand{\La}{\Lambda}

\newcommand{\dD}{\Delta}
\newcommand{\dS}{S}

\newcommand{\Hom}{\mathrm{Hom}}

\newcommand{\cB}{{\mathcal B}}

\newcommand{\modu}{{\text{-mod}}}
\newcommand{\wt}{{\text{wt}}}
\newcommand{\ab}{{\text{ab}}}

\newcommand{\N}{\mathbb{N}}
\newcommand{\Z}{\mathbb{Z}}
\newcommand{\R}{\mathbb{R}}
\newcommand{\C}{\mathbb{C}}
\newcommand{\Q}{\mathbb{Q}}

\newcommand{\mO}{\mathrm{O}}

\newcommand{\so}{{\mathrm{so}_{2n+1}}}

\newcommand{\vep}{{\varepsilon}}

\newcommand{\va}{\bm{1}}
\newcommand{\1}{\bm{1}}
\newcommand{\0}{{\bar{0}}}
\newcommand{\id}{{\mathrm{id}}}

\newcommand{\z}{{\bar{z}}}

\newcommand{\h}{{\bar{h}}}
\newcommand{\p}{{\bar{p}}}

\newcommand{\uz}{\underline{z}}

\newcommand{\s}{{\bar{s}}}
\newcommand{\n}{{\bar{n}}}

\newcommand{\Vect}{{\bf{Vec}}}

\newcommand{\al}{\alpha}

\newcommand{\be}{\beta}
\newcommand{\ga}{\gamma}

\newcommand{\om}{\omega}
\newcommand{\la}{\lambda}

\newcommand{\g}{{\mathfrak{g}}}
\newcommand{\fg}{{\mathfrak{g}}}
\newcommand{\fh}{{\mathfrak{h}}}

\newcommand{\CB}{{\mathcal{B}}}
\newcommand{\CE}{{\mathcal{E}}}

\newcommand{\Ld}{{\overline{L}}}

\newcommand{\tw}{{{I\hspace{-.1em}I}}}

\newcommand{\End}{\mathrm{End}}
\newcommand{\fora}{\text{ for any }}

\newcommand{\cop}{{\mathrm{cop}}}
\newcommand{\op}{{\mathrm{op}}}
\newcommand{\rev}{{\mathrm{rev}}}

\newtheorem{thm}{Theorem}[section]
\newtheorem{dfn}[thm]{Definition}
\newtheorem{lem}[thm]{Lemma}
\newtheorem{prop}[thm]{Proposition}
\newtheorem{cor}[thm]{Corollary}
\newtheorem{rem}[thm]{Remark}

\newtheorem{conj}{Conjecture}
\newtheorem{conji}{Conjecture}

\newtheorem{mainthm}{Main Theorem}

\begin{document}
\title[Quantum coordinate ring]{}

\begin{center}
{\LARGE \bf Quantum coordinate ring in WZW model and
affine vertex algebra extensions
} \par \bigskip

\renewcommand*{\thefootnote}{\fnsymbol{footnote}}
{\normalsize
Yuto Moriwaki \footnote{email: \texttt{moriwaki.yuto (at) gmail.com}}
}
\par \bigskip
{\footnotesize Research Institute for Mathematical Sciences, Kyoto University\\
Kyoto, Japan}

\par \bigskip
\end{center}

\noindent
\textbf{Abstract.}
In this paper, we construct various simple vertex superalgebras which are extensions of 
affine vertex algebras, by using abelian cocycle twists of representation categories of quantum groups.
This solves the Creutzig and Gaiotto conjectures \cite[Conjecture 1.1 and 1.4]{CG}
in the case of type ABC.
If the twist is trivial, the resulting algebras correspond to chiral differential operators in the chiral case, 
and to WZW models in the non-chiral case.

\vspace*{8mm}

\begin{center}
{\large \bf Introduction
}
\end{center}
 \par \bigskip

Let $\g$ be a finite-dimensional simple Lie algebra,
$h^\vee$ the dual Coxeter number
and $r^\vee$ the lacing number, that is,
$r^\vee=1$ (resp. $r^\vee=2$ and $r^\vee=3$) if the simple Lie algebra $\g$ is simply-raced
(resp. of type BCF and of type G).
Let $k,k' \in \C \setminus \Q$ and $N \in \Z$
satisfy 
\begin{align}
\frac{1}{r^\vee(k+h^\vee)}+\frac{1}{r^\vee(k'+h^\vee)}=N. \label{eq_level}
\end{align}
%and $L$ be a subgroup of the weight lattice $\lP$
%such that $L$ contains the root lattice, $\lQ \subset L$.
Let $\lP$ be the weight lattice and $\lQ$ the root lattice and set
\begin{align*}
V_{\g,k,k'}^N(\lP) = \bigoplus_{\la \in \lP^+}
L_{\g,k}(\la)\otimes L_{\g,k'}(\la^*),
\end{align*}
where $\lP^+$ is the dominant integer weights.
Here, $L_{\g,k}(\la)$ is {\it the Weyl module} induced from $L(\la)$, the irreducible finite dimensional representation of $\g$ with highest weight $\la$. The module induced from the dual representation  $L(\la)^*$ of level $k'$ is denoted by $L_{\g,k'}(\la^*)$.

Creutzig and Gaiotto conjectured that $V_{\g,k,k'}^N(\lP)$ inherits a {vertex superalgebra} structure based on gauge theory \cite[Conjecture 1.1]{CG}.
The construction of this algebra is considered to be important not only in gauge theory but also in various areas of mathematics, such as the quantum geometric Langlands program (see \cite{FG}).

For $N=0$, the condition \eqref{eq_level} can be written as $k+k'=-2h^\vee$.
Such a $\bar{k}=k'$ is called a {\it dual level}.
The $V_{\g,k,\bar{k}}^0(\lP)$ is called a {\it chiral differential operator}, 
and its geometric constructions are known \cite{AG,FS,GMS1,GMS2,Zh}.
In \cite[Corollary 1.4]{CKM2} and \cite[Proposition 5.3]{Fe}, it was shown that $V_{\g,k,k'}^N(\lP)$ is a vertex algebra when $N \in 2n_\g\Z$. Here $n_\g$ is the smallest positive integer such that
$n_\g \lP$ is an integral lattice with respect to the bilinear form $\lla-,-\rra:\lP\times \lP\rightarrow \Q$ 
which is normalized as $(\al,\al)=2$ for short roots $\al$.

In this paper, we will show that the construction for a general integer $N$ follows from a certain conjecture about the representation category of quantum groups (see Conjecture A below).
More precisely, if Conjecture A is true for $(\g,N)$,
then we show that $V_{\g,k,k'}^N(\lP)$ is an abelian intertwining algebra
with abelian cocycle $\mathrm{EM}^{-1}(\lQ_{\g}^N) \in H_\ab^3(\lD,\C^\times)$.

Conjecture A is proved for any $N \in 2\Z$ and any simple Lie algebra $\g$ (Proposition \ref{even_twist})
%which implies that $V_{\g,k,k'}^N(\lP)$ is an abelian intertwining algebra
%with abelian cocycle $\mathrm{EM}^{-1}(\lQ_{\g}^N) \in H_\ab^3(\lD,\C^\times)$ (see Remark \ref{rem_AIA}).
%What is particularly interesting in Creutzig-Gaiotto conjecture is the case where $N$ is odd.
%One key result of this paper is that 
and for any $N \in \Z$ and simple Lie algebras $\g$ of type ABC (see Main Theorem A below).
This in particular solves the Creutzig-Gaiotto conjecture for type ABC.
In fact, we can construct a more general family of vertex superalgebras including $V_{\g,k,k'}^N(\lP)$.

More specifically, we construct families of simple vertex superalgebras which are extensions of 
$V_M\otimes \bigotimes_{i=1}^r L_{\g_i,k_i}(0)\otimes L_{\g_i,k'_i}(0)$, where $V_M$ is the lattice vertex (super)algebra associated with an integral lattice $M$  (for the precise statement, see Theorem B below).
%Furthermore, using Main Theorem B below, we can construct a more general vertex superalgebra which is an extension of 
%$V_M\otimes \bigotimes_{i=1}^r L_{\g_i,k_i}\otimes L_{\g_i,k'_i}$, where $V_M$ is the lattice vertex algebra associated with an integral lattice $M$. 
This allows us to construct, for example, the following vertex superalgebra for any $n\geq 2$ (Proposition \ref{example_GL} and \ref{example_r}):
\begin{align*}
&\bigoplus_{\la \in \lP_{\mathrm{sl}_n}} L_{\mathrm{sl}_n,k}(\la)\otimes L_{\mathrm{sl}_n,k'}(\la^*)\otimes 
V_{\frac{i(\la)\sqrt{n}}{n}+\sqrt{n}\Z},\;\;\;\;\;\;
\text{ for }\frac{1}{k+n}+\frac{1}{k'+n}=1
\\
&\bigoplus_{\la \in \lP_{\mathrm{sl}_n}} \bigotimes_{i=1}^n
L_{\mathrm{sl}_n,k_i}(\la)\otimes L_{\mathrm{sl}_n,k'_i}(\la^*),\;\;\;\;\;\;\;\;\;\;\;\;\;\;\;\;
\text{ for }\frac{1}{k_i+n}+\frac{1}{k'_i+n}=1
\text{ and }i=1,\dots,n,
\end{align*}
where $V_{\sqrt{n}\Z}$ is a rank one lattice vertex (super)algebra and
$V_{\frac{i(\la)\sqrt{n}}{n}+\sqrt{n}\Z}$ is its module, $i$ is a map given by $i:\lP_{\mathrm{sl}_n}\rightarrow \lP_{\mathrm{sl}_n}/\lQ_{\mathrm{sl}_n} \cong \Z/n\Z$ (The first algebra is the algebra conjectured in \cite[Conjecture 1.4]{CG}).

In the following, we will explain that {\it quantum groups} and {\it quantum coordinate rings} appear naturally when considering extensions of 
affine vertex algebras, and then we will discuss the conjecture about the representation categories of quantum groups and the main results derived from it.

\noindent
\begin{center}
{0.1. \bf 
Quantum coordinate rings and WZW models
}
\end{center}

%We will explain how to show the main theorems.
%There is a one-to-one correspondence between extensions of a good vertex algebra $V$ and commutative algebra objects in its representation category $V\modu$ \cite{FRS,HKL}.

There is a natural correspondence between {\it commutative algebra objects} in the representation category of a vertex (operator) algebra and extensions of the vertex (operator) algebra \cite{FRS,HKL}.
The representation category of the affine VOA $L_{\g,k}(0)$
at level $k\in\C \setminus \Q$ is called the Drinfeld category $D(\g,k)$
\cite{TK,FZ,EFK}
and is equivalent to the representation category of the quantum group $(U_q(\g),R(\rho))\modu$ as braided tensor categories \cite{Dr1,Dr2,KL,Lu2}.
%From the Kazhdan-Lusztig correspondence, the representation category of $L_{\g,k}(0)$, the affine vertex algebra of level $k \in \C\setminus \Q$, and the representation category of the quantum group $(U_q(\g),R(\rho))\modu$ are equivalent as braided tensor categories \cite{Dr1,Dr2,KL,Lu2}.
Here, $q=\exp(\pi i \rho), \rho=\frac{1}{r^\vee(k+h^\vee)}$, and $R(\rho)$ is the R-matrix of $U_q(\g)$ which gives the braiding.

The dual Hopf algebra of $U_q(\g)$ is called a {\it quantum coordinate ring}
and denoted by $\mO_q(G)$.
By the natural $U_q(\g)$-bimodule structure on $\mO_q(G)$,
it defines a commutative algebra object in $(U_q(\g),R(\rho))\otimes (U_q(\g)^\cop,R(\rho)^{-1})\modu$ (Proposition \ref{dual_commutative} and Proposition \ref{commutative_algebra}).

%The quantum coordinate ring $\mO_q(G)$, which is the dual Hopf algebra of $U_q(\g)$, is naturally a $U_q(\g)$-bimodule.
%Although $\mO_q(G)$ is a non-commutative algebra, since $U_q(\g)$ is a quasi-triangular Hopf algebra, $\mO_q(G)$ is a commutative algebra object in $(U_q(\g),R(\rho))\otimes (U_q(\g)^\cop,R(\rho)^{-1})\modu$ (Proposition \ref{dual_commutative} and Proposition \ref{commutative_algebra}).

Using an isomorphism of quasi-triangular Hopf algebras 
$(U_q(\g),R(\rho)_{21}^{-1}) \cong (U_q(\g)^\cop,R(\rho)^{-1})$
(Proposition \ref{op_isomorphism})
and the Kazhdan-Lusztig correspondence, we can think of $\mO_q(G)$ as a commutative algebra object of $D(\g,k)\otimes D(\g,k)^\rev$,
where $D(\g,k)^\rev$ is the braided tensor category with the reverse braiding.
%$(U_q(\g),R(\rho))\otimes (U_q(\g),R(\rho)_{21}^{-1})\modu$.
Hence, $\mO_q(G)$ defines an extension of
the tensor product of the ``holomorphic'' affine vertex algebra $L_{\g,k}(0)$ and an ``anti-holomorphic'' affine vertex algebra  $\overline{L_{\g,k}(0)}$,
\begin{align*}
F_{G,k}=\bigoplus_{\la \in \lP^+} L_{\g,k}(\la)\otimes \overline{L_{\g,k}(\la^*)},\;\;
Y:F_{G,k} \rightarrow \End F_{G,k}[[z,\z,|z|^\C]].
%\text{with a vertex operator } Y(a,z,\z)=\sum_{r,s \in \C} a_{r,s}z^{-r-1}\z^{-s-1}.
\end{align*}
This algebra satisfies the axiom of a
\renewcommand*{\thefootnote}{\fnsymbol{footnote}}
{\normalsize
{\it full vertex algebra}
 \footnote{
$F_{G,k}$ also satisfies the axiom of a {\it full field algebra}
introduced by Huang and Kong in \cite{HK}.
}}
 (one of the formulations of the non-chiral CFT introduced in \cite{M1}) and is nothing but the analytic continuation of {\it the WZW model} in physics (Proposition \ref{prop_WZW}).

Also, using an isomorphism $(U_q(\g)^\cop,R(\rho)^{-1}) \cong (U_{q^{-1}}(\g),R(-\rho))$ (Proposition \ref{psi_isomorphism}), 
$\mO_q(G)$
gives a commutative algebra object in 
$D(\g,k)\otimes D(\g,\bar{k})$,
% \cong (U_q(\g),R(\rho))\otimes (U_{q^{-1}}(\g),R(-\rho))\modu$,
which corresponds to the chiral differential operator 
$V_{\g,k,\bar{k}}^0(\lP)$.
These are the stories for the $N=0$ case.

%頂点代数の表現圏の可換代数対象と頂点代数の拡大には自然な対応がある。KL対応から、levelのアファイン頂点代数の表現圏と量子群の表現圏$(U_q(\g),R(\rho))\modu$は組紐テンソル圏同値になる。ただし$q=\exp(\pi i \rho), \rho=\frac{1}{r^\vee(k+h^\vee)}$であり、
%$R(\rho)$は braiding を与える$U_q(\g)$のR行列。
%$U_q(\g)$の双対代数、量子群$O_q(G)$は自然な$U_q(\g)$の両側作用が入る。$U_q(\g)$がquasi-triangular Hopf代数であることから、
%$O_q(G)$は$U_q(\g)\otimes U_q(\g)^\cop$の可換代数対象である
%(代数としては$O_q(G)$は非可換環であることに注意)。
%同型$U_q(\g)^\cop \cong U_q(\g)$を用いて
%$O_q(G)$を$...$の可換代数対象と思うと、KL対応から$O_q(G)$は正則なアファイン頂点代数と
%反正則なアファイン頂点代数の拡大、
%\begin{align*}
%F_{G,k}=
%\end{align*}
%を定める。この代数はfull頂点代数(non-chiral CFTの定式化の一つ)の構造を持ち、WZW模型の解析接続に他ならない。
%また同型$U_q(\g)^\cop \cong U_{q^{-1}}(\g)$を用いると、
%$O_q(G)$はCDO$V_{\g,k\bar{k}}^0$を定める。
%これらは$N=0$の場合のストーリーである。

\noindent
\begin{center}
{0.2. \bf 
Graded twisting
}
\end{center}

Next, we discuss the case where the level is shifted by $N \in \Z$.
The representation category of $U_q(\g)$ is natural graded by $\lP/\lQ$.
In general, for an abelian group $A$ and an $A$-graded braided tensor category, 
we can ``twist'' the category by an abelian cocycle $(\al,c) \in H_\ab^3(A,\C^\times)$ (see \cite{KW,NY}).
By Eilenberg and Mac Lane, $H_\ab^3(A,\C^\times)$ and the space of all quadratic forms on $A$ are isomorphic as a group.

Now, for $N\in \Z$, the quadratic form $Q_\g^N:\lD\rightarrow \C^\times$ on $\lP/\lQ$ is defined by
$$
Q_\g^N(\la) = \exp(N\pi i \lla \la,\la \rra) \text{ for }\la \in \lD.
$$
Denote by $(U_q(\g),R(\rho))\modu^{Q_\g^N}$
the twist of the braided tensor category $(U_q(\g),R(\rho))\modu$ by 
the abelian cocycle associated with the quadratic form $Q_\g^N$
 (for more precise definitions, see Section \ref{sec_twist_intro}
and Section \ref{sec_twist}).
Then, we expect that:
%
%次にレベルを$N$だけシフトする場合について述べる。
%$U_q(\g)$の表現圏には自然な$\lP/\lQ$次数付けが入る。
%Aをアーベル群とする。一般にA次数付きの組紐テンソル圏があると、abelian cocycle$H_\ab^3(A,\C^\times)$で圏をtwistすることができる。
%またEMにより$H_\ab^3(A,\C^\times)$とA上のquadratic formの空間は群として同型である。
%さて$N\in \Z$に対して、$\lP/\lQ$上の quadratic form $Q_\g^N:\lD\rightarrow \C^\times$を
%$$
%Q_\g^N(\la) = \exp(N\pi i \lla \la,\la \rra) \text{ for }\la \in \lD
%$$
%で定める。
%$EM^{-1}(Q_\g^N) \in H^3(\lD,\C^\times)$による圏$U_q(\g)$のtwistを
%$(U_q(\g),R(\rho))\modu^{Q_\g^N}$とおく。
%このとき我々は次が成り立つと予想している：
\begin{conji}
For any $N \in \Z$ and any simple Lie algebra $\g$,
$(U_q(\g),R(\rho))\modu^{Q_\g^N}$ and 
$(U_q(\g),R(\rho+N))$ are equivalent
as braided tensor categories.
\end{conji}
%More precisely, if Conjecture A is true for $(\g,N)$,
%then we show that $V_{\g,k,k'}^N(\lP)$ is an abelian intertwining algebra
%with abelian cocycle $\mathrm{EM}^{-1}(\lQ_{\g}^N) \in H_\ab^3(\lD,\C^\times)$ (see below and Remark \ref{rem_AIA}).
We will show this conjecture in the following cases (Proposition \ref{even_twist} and Theorem \ref{odd_twist}):
\begin{mainthm}
Conjecture A is true at least in the following cases:
\begin{enumerate}
\item
$N \in 2\Z$ for any simple Lie algebra $\g$;
\item
$N \in \Z$ for a simple Lie algebra $\g$ of type ABC.
\end{enumerate}
Furthermore, in the case of type D,
$(U_q(\mathrm{so}_{2n}),R(\rho),\Lambda_v)\modu^{Q_{\mathrm{so}_{2n}}}$ and $(U_q(\mathrm{so}_{2n}),R(\rho+N),\Lambda_v)\modu$
are equivalent as braided tensor categories,
where $(U_q(\mathrm{so}_{2n}),R(\rho),\Lambda_v)\modu$ is the full subcategory of $(U_q(\mathrm{so}_{2n}),R(\rho))\modu$ generated by the vector representation of $U_q(\mathrm{so}_{2n})$.
\end{mainthm}
We briefly explain the proof of Theorem A. In the case of type B, we use the Hopf algebra isomorphism $\phi:U_q(\so)\rightarrow U_{-q}(\so)$. In this case, the pullback by $\phi$ does not transfer the type 1 representation of $U_{-q}(\so)$ to the type 1 representation $U_{q}(\so)$. 
This is why ``the twisted category'' appears. This will be discussed in detail in the appendix.
In the case of type ACD, Theorem A can be proved by using a characterization of the braided tensor category of type A (resp. type BCD) by \cite{KW}(resp. \cite{TW}), instead of a direct proof using such Hopf algebra isomorphisms.
The result for the braided tensor category of type BD by \cite{TW} is a characterization of subcategories, and hence the claim for type D is weaker than Conjecture A.

%定理Aの証明であるが、B型($B_1=A_1$を含む)の場合は、
%Hopf代数としての同型$\phi:U_q(\so)\rightarrow U_{-q}(\so)$
%を用いる。この場合は$\phi$による引き戻しはtype 1表現をtype 1表現に移さないため、cocycleでtwistされた圏が現れる。
%ACD型の場合は、このようなHopf代数を用いた直接的な証明ではなく、
%\cite{KW}(resp. \cite{TW})によるA型(resp. BCD型)の組紐テンソル圏の特徴付けを用いると定理Aが証明できる。
%TWによるBD型の組紐テンソル圏の結果は部分圏に対する特徴付けであり、
%D型の主張は予想Aより弱くなっている。B型の場合は上記の特徴付けによらない直接の証明から、予想Aが従っている。

\noindent
\begin{center}
{0.3. \bf 
Lax monoidal functors and main theorem
}
\end{center}

We will now state the result and the proof of the construction of vertex superalgebras (Main Theorem B) using Conjecture A.
Let $A$ be an abelian group and let $\Vect_A$ be the 
category of $A$-graded vector spaces.
Then, $\Vect_A$ is a trivial braided tensor category.
Using the fact that $\mO_q(G)$ gives a commutative algebra object in
 $(U_q(\g),R(\rho))\otimes (U_{q^{-1}}(\g),R(-\rho))\modu$
and that $\mO_q(G)$ is a $\lD$-graded algebra,
we can construct {\it a lax monoidal functor}
$$
O_\g: \Vect_\lD \rightarrow (U_q(\g),R(\rho))\otimes (U_{-q}(\g),R(-\rho))\modu
$$
which preserves the braiding (we call it {\it a lax braided monoidal functor}, see Section \ref{sec_braided}).

Using Conjecture A with simultaneous grading twists of $\Vect_\lD$ and $(U_q(\g),R(\rho))\modu$, we obtain a lax braided monoidal functor (Lemma \ref{lem_lax0} and Proposition \ref{prop_lax})
$$
O_\g^N: \Vect_\lD^{Q_\g^N} \rightarrow (U_q(\g),R(\rho+N))\otimes (U_{-q}(\g),R(-\rho))\modu.
$$
Since a lax braided monoidal functor transfers a supercommutative algebra object to a supercommutative algebra object (Lemma \ref{lem_lax}), we only need to find a supercommutative algebra object in $\Vect_\lD^{Q_\g^N}$.

To find a supercommutative algebra object in a more flexible framework, we consider an even lattice $M$.
The dual lattice of $M$ is $M^\vee$, and a quadratic form on the abelian group $M^\vee/M$ is defined by $Q_M(\la)=\exp(\pi i (\la,\la))$.
Then, from \cite{DL}, the representation category of the lattice vertex algebra $V_M$ is equivalent to $\Vect_{M^\vee/M}^{Q_M}$
as a braided tensor category.
Thus, from supercommutative algebra objects in 
$\Vect_{\lD \oplus M^\vee/M}^{Q_\g^N\oplus Q_M}$, we can construct extensions of 
$L_{\g,k}(0)\otimes L_{\g,k'}(0)\otimes V_M$. Here, $k,k'$ 
satisfy 
\begin{align*}
\rho+N = \frac{1}{r^\vee(k+h^\vee)},\;\;\;\;\;\;\;\;
-\rho = \frac{1}{r^\vee(k'+h^\vee)},
\end{align*}
which is essentially the same with $\frac{1}{r^\vee(k+h^\vee)}+\frac{1}{r^\vee(k'+h^\vee)}=N$.
Also, supercommutative algebra objects in 
$\Vect_{\lD\oplus M^\vee/M}^{Q_\g^N\oplus Q_M}$
can be constructed from {\it a super isotropic subspace} of quadratic space $(\lD\oplus M^\vee/M,Q_\g^N\oplus Q_M)$.
(for the definition of a super isotropic subspace, see Section \ref{sec_twist_intro}). Our main result can then be stated as follows:
%予想Aを用いて頂点超代数の構成(主定理B)の結果と証明を述べる。
%$A$をアーベル群として、$\Vect_A$をA次数付きベクトル空間のなすテンソル圏とする。$\Vect_A$は自明な組紐テンソル圏である。
%さて$O_q(G)$が$U_q(\g)..$の可換代数対象を与えることと、
%$O_q(G)$が$\lD$次数付き代数であることを用いると、braidingを保つlax monoidal functor
%$$
%O_\g: \Vect_\D \rightarrow (U_q(\g),R(\rho))\otimes (U_{-q}(\g),R(-\rho))\modu
%$$
%が構成できる。
%$\Vect_\lD$と$(U_q(\g),R(\rho))\modu$を同時にgrading twistしてConjecture Aを用いると、lax braided monoidal functor
%$$
%O_\g^N: \Vect_\D^{Q_\g^N} \rightarrow (U_q(\g),R(\rho+N))\otimes (U_{-q}(\g),R(-\rho))\modu
%$$
%を得る。lax braided monoidal functorは超可換代数対象を超可換代数対象へと移すため、$\Vect_\D^{Q_\g^N}$の超可換代数対象を探せば良い。
%より柔軟に超可換代数対象を探すため、偶格子$M$を考える。
%$M$の双対格子を$M^\vee$とし、アーベル群$M^\vee/M$上のquadratic formを$Q_M(\la)=\exp(2\pi i (\la,\la))$で定める。
%すると\cite{DL}より、偶格子$M$に付随する格子頂点代数$V_M$の表現圏は、$\Vect_{M^\vee/M}^{Q_M}$と組紐テンソル圏同値である。
%よって$\Vect_{\D\oplus M^\vee/M}^{Q_\g^N\oplus Q_M}$の可換代数対象から、$L_{\g,k}(0)\otimes L_{\g,k'}(0)\otimes V_M$の拡大が構成できる。ただし$k,k'$は
%$\rho+N = \frac{1}{r^\vee(k+h^\vee)}$ and 
%$-\rho = \frac{1}{r^\vee(k'+h^\vee)}$を満たす(すなわち\eqref{})。
%また$\Vect_{\D\oplus M^\vee/M}^{Q_\g^N\oplus Q_M}$の超可換代数対象は、quadratic space $(\D\oplus M^\vee/M,Q_\g^N\oplus Q_M)$
%のsuper isotropic subspace から構成できる(for the definition of a super isotropic subspace, see Lemma \ref{})。
%このとき我々の主結果は以下のように述べられる：
\begin{mainthm}
Let $r \in \Z_{>}$ and $\g_i$ be simple Lie algebras
and $k_i,k_i' \in \C \setminus \Q$ and $N_i \in \Z$
satisfy 
$$
\frac{1}{r_i^\vee(k_i+h_i^\vee)}+\frac{1}{r_i^\vee(k_i'+h_i^\vee)}=N_i
$$
for $i=1,\dots,r$.
Let $M$ be an even lattice
and $(A,Q)$ a quadratic space defined by
\begin{align*}
A=(\bigoplus_{i=1}^r \lP_i/\lQ_i) \oplus M^\vee/M,\;\;\;\;\;\;\;
Q=(\bigoplus_{i=1}^r Q_{\g_i}^{N_i})\oplus Q_M.
%A&=(\bigoplus_{i=1}^r \lP_i/\lQ_i) \oplus M^\vee/M,\\
%Q&=(\bigoplus_{i=1}^r Q_{\g_i}^{N_i})\oplus Q_M.
\end{align*}
Let $(I,p)$ be a super isotropic subspace of the quadratic space $(A,Q)$.
Set 
\begin{align*}
V_{\vec{\g},\vec{k},\vec{k'},M}^{\vec{N}}(I)&=\bigoplus_{(\la_1,\dots,\la_r,\mu) \in I} \bigotimes_{i=1}^r L_{\g_i,k_i,k_i'}^{N_i}(\la_i+\lQ_i)\otimes V_{\mu+M},
\end{align*}
for $(\g,k,k',M,I)$ with $\vec{\g}=(g_1,\dots,g_r)$ and $\vec{k}=(k_1,\dots,k_r),\vec{k'}=(k'_1,\dots,k'_r),\vec{N}=(N_1,\dots,N_r)$,
where 
\begin{align*}
L_{\g_i,k_i,k_i'}^{N_i}(\la_i+\lQ_i)
=\bigoplus_{\la \in (\la_i+\lQ_i)\cap \lP_i^+}
L_{\g_i,k_i}(\la) \otimes L_{\g_i,k'_i}(\la).
\end{align*}
Assume that for each $a=1,\dots,r$
one of the following conditions is satisfied:
%Suppose that $N$ is even or one of the following conditions hold:
\begin{enumerate}
\item
$N_a$ is even;
\item
The simple Lie algebra $\g_a$ is of type ABC;
\item
The simple Lie algebra $\g_a$ is of type D
and $\mathrm{pr}_a(I) \subset \Lambda_v/\lQ_i$,
where $\mathrm{pr}_a:(\oplus_{i=1}^r \lP_i/\lQ_i) \oplus M^\vee/M
\rightarrow \lP_a/\lQ_a$ is the projection to the $a$-th component.
\end{enumerate}
Then, there is a simple vertex superalgebra structure on $V_{\vec{\g},\vec{k},\vec{k'},M}^{\vec{N}}(I)$
as an extension of $\left(\bigotimes_{i=1}^r L_{\g_i,k_i}(0)\otimes {L_{\g_i,k'_i}(0)}\right) \otimes V_M$.
Furthermore, the even part ($s=0$) and the odd part ($s=1$) are given by
$$
V_{\vec{\g},\vec{k},\vec{k'},M}^{\vec{N}}(I)_s=\bigoplus_{\substack{(\la_1,\dots,\la_r,\mu) \in I\\p(\la_1,\dots,\la_r,\mu)=s}} \bigotimes_{i=1}^r L_{\g_i,k_i,k_i'}^{N_i}(\la_i+\lQ)\otimes V_{\mu+M}.
$$
\end{mainthm}
We remark that Conditions (1), (2) and (3) in the theorem are due to the fact that Conjecture A has only been partially proved.
The list of all vertex superalgebras (except for type A) that can be constructed from the theorem for $M=0,r=1$ is summarized in Table \ref{table1_intro} and \ref{table_S_intro}.
In this case, super isotropic subspaces $I \subset \lD$ correspond to lattices $L$ such that $\lQ \subset L \subset \lP$.
\begin{table}[htbp]
  \begin{tabular}{cc}
    \begin{minipage}[c]{0.5\hsize}
      \centering
  \begin{tabular}{|c||c|c|c|} \hline
type & shift $N$& lattice $L$ & super \\ \hline \hline
$B_{2n}$ & $1+2\Z$ & $\lP$ & S \\
$B_{2n}$ & $2\Z$ & $\lP$ & \\
$B_{2n+1}$ & $\Z$ & $\lQ$ & \\
$B_{2n+1}$ & $2+4\Z$ & $\lP$ & S \\
$B_{2n+1}$ & $4\Z$ & $\lP$ &\\ \hline
$C_n$ & $1+2\Z$ & $\lP$ & S \\
$C_n$ & $2\Z$ & $\lP$ & \\ \hline
$D_{n}$ & $1+2\Z$ & $\La_v$ & S \\
$D_{n}$ & $2\Z$ & $\La_v$ & \\
$D_{4n+2}$ & $2+4\Z$ & $\lP$ & S\\
$D_{4n}$ & $2\Z$ & $\lP$ & \\
$D_{2n}$ & $4\Z$ & $\lP$ & \\
$D_{2n+1}$ & $4+8\Z$ & $\lP$ & S \\
$D_{2n+1}$ & $8\Z$ & $\lP$ & \\ \hline
 \hline
\end{tabular}
\caption{List of type BCD}
\label{table1_intro}
    \end{minipage} &
    \begin{minipage}[c]{0.5\hsize}
      \centering
  \begin{tabular}{|c||c|c|c|} \hline
type & shift $N$ & lattice $L$ & super\\ \hline \hline
$E_{6}$ & $6\Z$ & $\lP$ &  \\
$E_{6}$ & $2\Z$ & $\lQ$ &  \\
$E_{7}$ & $2+4\Z$ & $\lP$ & S \\
$E_{7}$ & $4\Z$ & $\lP$ & \\
$E_{8}$ & $2\Z$ & $\lP$ &  \\ \hline
$F_{4}$ & $2\Z$ & $\lP$ &  \\ \hline
$G_{2}$ & $2\Z$ & $\lP$ &  \\
 \hline
\end{tabular}
\caption{List of type EFG}
\label{table_S_intro}
    \end{minipage}
  \end {tabular}
\end{table}

In the case of $A_{n-1}$ type, there are various ways to choose shift $N$ and lattice $L$. In fact, for any $N \in \Z$,
we can show that $V_{\mathrm{sl}_{n},k,k'}^{N}(m\lP+\lQ)$
is a vertex superalgebra if $\exp(\frac{m^2N(n-1)}{n}\pi i)=-1$
and a non-super vertex algebra if $\exp(\frac{m^2N(n-1)}{n}\pi i)=1$ (see Proposition \ref{example_even} (4)).

%The list of simple vertex superalgebras $V_{\g,k,k'}^N(L)$ constructed in this paper (except for type A) is summarized in Table \ref{table1_intro} and \ref{table_S_intro}.

\vspace{4mm}

This paper is organized as follows:
We recall the definition of 
a braided tensor category, a lax monoidal functor and supercommutative algebra object in Section \ref{sec_braided}
and 
the fundamental results of abelian cocycles, graded twists
and super isotropic subspaces in Section \ref{sec_twist_intro}.
In Section \ref{sec_quasi}, we recall some elementary results of
quasi-triangular Hopf algebra
and in Section \ref{sec_affine1},
we briefly review the Drinfeld category
and the fact that supercommutative algebra objects in 
the Drinfeld category corresponds to vertex superalgebra extensions
of the affine vertex algebra. Non-chiral cases are briefly explained in Section \ref{sec_affine2}, which will be used for the construction of WZW models.

Section \ref{sec_quantum_def} and \ref{sec_bilinear}
is devoted to recalling the definition and some results of quantum group, its R-matrix and the Kazhdan-Lusztig correspondence.
Then, we show that isomorphisms among quantum groups
give equivalences of braided tensor categories in Section \ref{sec_isomorphism}.
In Section \ref{sec_twist}, we state Conjecture A and prove Theorem A.

In Section \ref{sec_quantum_coordinate} and \ref{sec_lax},
we recall the definition of quantum coordinate rings
and construct the lax monoidal functor by using it.
Then, Theorem B is proved.
Various vertex algebras as applications of Theorem B are constructed in Section \ref{sec_application}.
Finally, in Appendix, Theorem A is proved in the case of type B.

\noindent
\begin{center}
{\bf Acknowledgements}
\end{center}

I wish to express my gratitude to
Shigenori Nakatsuka for letting me know about the Creutzig and Gaiotto's conjecture and valuable discussions
and to Yuki Arano for discussions on quantum groups,
and to Makoto Yamashita
and Hironori Oya for giving me the references.
I would also like to thank Tomoyuki Arakawa and Thomas Creutzig
for valuable comments.
This work was supported by the Research Institute for Mathematical Sciences,
an International Joint Usage/Research Center located in Kyoto University.

%\tableofcontents

\section{Preliminary}
\label{sec_preliminary}
\subsection{Notations}
\label{sec_notation}

We fix the following notations:
\begin{itemize}
\item[] $\g$ is a finite-dimensional simple Lie algebra over $\C$,
\item[] $\fh$ its Cartan subalgebra,
\item[] $\Delta \subset \fh^*$ the root system,
\item[] $\Pi =\{\al_1,\dots,\al_r \} \subset \Delta$ the set of simple roots,
\item[] Let $(-,-)$ be an invariant bilinear form on $\g$ normalized by $(\al,\al)=2$ for long roots $\al$,
\item[] $h^\vee$ the dual Coxeter number,
%\item[] $A=(a_{ij})_{1\leq i,k\leq r}$ the Cartan matrix, $a_{ij}=\frac{2(\al_i,\al_j)}{(\al_i,\al_i)}$,
\item[] $\lP \subset \fh^*$ the weight lattice,
\item[] $\lP^+ \subset \lP$ the dominant integer weights,
\item[] $\lQ \subset \lP$ the root lattice,
\item[] $\rho \in \lP$ half the sum of all positive roots,
%\item[] $\lQ^\vee\subset  \fh$ the dual root lattice,
\item[] Let $\lla-,-\rra$ be another invariant bilinear form on $\g$ normalized by 
$\lla\al,\al \rra=2$
for short roots $\al$.
\end{itemize}

Let $z$ and $\z$ be independent formal variables.
We will use the notation $\underline{z}$ for the pair $(z,\z)$ and $|z|$ for $z\z$.
For a vector space $V$,
we denote by $V[[z^\C,\z^\C]]$
the set of formal sums 
$$\sum_{s,\s \in \C} a_{s,\s}
z^{s} \bar{z}^{\s}$$ such that $a_{s,\s} \in V$.
The space $V[[z^\C,\z^\C]]$ contains various useful subspaces:
\begin{align*}
V[[z^\C]]&=\{\sum_{s \in \C} a_{s}
z^{s}\;|\;a_s \in V \}, \\
V[[z^\pm]]&=\{\sum_{n \in \Z} a_{n}
z^{n}\;|\;a_n \in V \}, \\
V[[z,\z,|z|^\C]] &= \{\sum_{s,\s \in \C} a_{s,\s}
z^{s} \bar{z}^{\s}\;|\;a_{s,\s}=0 \text{ unless }s-\s \in \Z
\},\\
V[[z,\z]] &= \{\sum_{n,\n \in \N} a_{n,\n}
z^{n} \bar{z}^{\n}\;|\;a_{n,\n} \in V\}.
\end{align*}

We also denote by
$V((z,\z,|z|^\C))$ the subspace of $V[[z,\z,|z|^\C]]$
consisting of the series
$\sum_{s,\s \in \R} a_{s,\s}z^{s} \bar{z}^{\s} \in V[[z,\z,|z|^\C]]$  such that:
\begin{enumerate}
\item
For any $H \in \R$,
$\#\{(s,\s)\in \R^2\;|\; a_{s,\s}\neq 0 \text{ and }\mathrm{Re}\; (s+\s)\leq H \}
$ is finite.
\item
There exists $N \in \R$ such that
$a_{s,\s}=0$ unless $\mathrm{Re}\;s \geq N$ and $\mathrm{Re}\;\s \geq N$
\end{enumerate}
and $V((z))$ the subspace of $V[[z^\pm]]$ consisting of the series
$\sum_{n \in \Z} a_{n}z^{n} \in V[[z^\pm]]$  such that:
\begin{enumerate}
\item
There exists $N \in \R$ such that
$a_{n}=0$ unless $n \geq N$.
\end{enumerate}
The space $V((z))$ is called the space of formal Laurent series.
Thus, $V((z,\z,|z|^\C))$ is a generalization of the Laurent series to two-variables.

\subsection{Braided tensor category}
\label{sec_braided}
This section is mainly based on \cite{EGNO} (for a lax monoidal functor, see also \cite{AM}).
Let $\mathcal{B}$ be an essentially small, $\C$-linear, monoidal category.
We write 
$$\otimes\colon \mathcal{B}\times \mathcal{B}\rightarrow \mathcal{B},\quad (M,N)\mapsto M\otimes N$$
for the tensor product functor with unit object $1_\mathcal{B}$, by 
\begin{align*}
&l_{\bullet}\colon 1_\mathcal{B} \otimes \bullet\xrightarrow{\simeq} \bullet,\quad (l_M\colon 1_\mathcal{B}\otimes M\xrightarrow{\simeq} M),\\
&r_{\bullet}\colon \bullet \otimes 1_\mathcal{B}\xrightarrow{\simeq} \bullet,\quad (r_M\colon M\otimes  1_\mathcal{B}\xrightarrow{\simeq} M),\\
&{\al}_{\bullet,\bullet,\bullet}\colon (\bullet \otimes\bullet)\otimes \bullet\xrightarrow{\simeq} \bullet\otimes (\bullet\otimes \bullet),\quad \left({\al}_{M,N,L}\colon (M\otimes N)\otimes L\xrightarrow{\simeq} M\otimes (N\otimes L)\right),
\end{align*}
the structural natural isomorphisms of functors satisfying the pentagon and triangle axioms, see \cite{Ka,BK}.
%$$
%\begin{array}{ccc}
%&& (w \otimes x) \otimes (y \otimes z)
%\\
%& {}^{\mathllap{\alpha_{w \otimes x, y, z}}}\nearrow
%&&
%\searrow^{\mathrlap{\alpha_{w,x,y \otimes z}}}
%\\
%((w \otimes x ) \otimes y) \otimes z
%&& &&
%(w \otimes (x \otimes (y \otimes z)))
%\\
%{}^{\mathllap{\alpha_{w,x,y}} \otimes id_z }\downarrow 
%&& && 
%\uparrow^{\mathrlap{ id_w \otimes \alpha_{x,y,z} }}
%\\
%(w \otimes (x \otimes y)) \otimes z
%&& \underset{\alpha_{w,x \otimes y, z}}{\longrightarrow} &&
%w \otimes ( (x \otimes y) \otimes z )
%\end{array}
%$$
The last one is called the associativity isomorphism.

{\it A braided monoidal tensor category} is a monoidal tensor category $\mathcal{B}$ equipped with a natural isomorphism of functors, called the \emph{braiding},
$${B}_{\bullet,\bullet}\colon(\bullet\otimes\bullet)\xrightarrow{\simeq} ( \bullet\otimes \bullet)\circ \sigma,\quad ({B}_{M,N}\colon M\otimes N\xrightarrow{\simeq}N\otimes M),$$
where $\sigma$ is the functor
$$\sigma\colon \mathcal{C}\times \mathcal{C}\rightarrow \mathcal{C}\times \mathcal{C},\quad (M,N)\mapsto (N,M).$$
The natural isomorphism ${B}_{\bullet,\bullet}$ is required to satisfy the hexagon identity,
$$
\begin{array}{ccccc}
   (M \otimes N) \otimes L 
   &\stackrel{\al_{M,N,L}}{\rightarrow}&
   M \otimes (N \otimes L)
   &\stackrel{B_{M,N \otimes L}}{\rightarrow}&
   (N \otimes L) \otimes M
   \\
   \downarrow^{B_{M,N}\otimes Id}
   &&&&
   \downarrow^{\al_{N,L,M}}
   \\
   (N \otimes M) \otimes L
   &\stackrel{\al_{N,M,L}}{\rightarrow}&
   N \otimes (M \otimes L)
   &\stackrel{Id \otimes B_{M,L}}{\rightarrow}&
   N \otimes (L \otimes M)
\end{array}
$$
and
$$
\begin{array}{ccccc}
   M \otimes (N \otimes L) 
   &\stackrel{\al_{M,N,L}^{-1}}{\to}&
   (M \otimes N) \otimes z
   &\stackrel{B_{M \otimes N, L}}{\to}&
   L \otimes (M \otimes N)
   \\
   \downarrow^{Id \otimes B_{N,L}}
   &&&&
   \downarrow^{\al_{L,M,N}^{-1}}
   \\
   M \otimes (L \otimes N)
   &\stackrel{\al^{-1}_{M,L,N}}{\to}&
   (M \otimes L) \otimes N
   &\stackrel{B_{M,L} \otimes Id}{\to}&
   (L \otimes M) \otimes N
\end{array}.
$$
%The natural isomorphism
%$$\mathcal{M}_{\bullet,\bullet}:=\mathcal{R}_{\bullet,\bullet}^2\colon (\bullet\boxtimes \bullet) \xrightarrow{\simeq}  (\bullet\boxtimes \bullet),\quad (\mathcal{M}_{M,N}\colon M\boxtimes N\xrightarrow{\simeq} M\boxtimes N)$$
%is called the \emph{monodromy}.

A \emph{twist}, or a \emph{balance}, in a braided monoidal category $\cB$ is a natural transformation $\theta$ from the identity functor on $\cB$ to itself satisfying 
$$
B_{M,N}B_{N,M}=\theta_{M\otimes N}\circ (\theta_M^{-1} \otimes \theta_N^{-1}).
$$
A {\it balanced monoidal category} is a braided monoidal category equipped with such a balance.

Let $(\cB,\al,{B})$ be a braided tensor category
and $B_{\bullet,\bullet}^\rev$ a natural isomorphism defined by
$$
B_{M,N}^\rev = B_{N,M}^{-1}: M \otimes N \rightarrow N\otimes M.
$$
Then, it is clear that $(\cB,\al,B^\rev)$ is a braided tensor category,
which is denoted by $\cB^\rev$.

%Let $(\cB,\al,{B})$ be a braided tensor category.
%A {\it commutative algebra object} in $\mathcal{B}$ is a triple $(A,m,\eta)$, (or $A$ for simplicity), 
%consisting of $A \in Ob(\mathcal{B})$ and a morphism  $m \in \Hom_\mathcal{B}(A\otimes A,A)$ and a
%nonzero morphism $\eta \in \Hom_\mathcal{B}(1_\mathcal{B},A)$ such that:
%\begin{enumerate}
%\item[(CA1)] $m\circ (\eta\otimes \id)=\id =(\id \otimes \eta)$;
%\item[(CA2)]  $m \circ (m \otimes \id)\circ \al_{A,A,A} = m \circ (\id \otimes m)$ as a map $A\otimes (A\otimes A) \rightarrow A$;
%
%\end{enumerate}
%An associative algebra object $A \in Ob(\mathcal{B})$ is called a {\it commutative algebra object}
%if it further satisfies the following condition:
%\begin{enumerate}
%\item[(CA3)] $m \circ B_{A,A} = m$.
%\end{enumerate}

A {\it supercommutative algebra object} in $\mathcal{B}$ is a triple $(S=S_0\oplus S_1, \{m_{i,j}\}_{i,j =0,1}, \eta)$,
consisting of objects $S_0,S_1 \in Ob(\mathcal{B})$ and morphisms  $m_{i,j} \in \Hom_\mathcal{B}(S_i\otimes S_j,S_{i+j})$ and a nonzero morphism $\eta \in \Hom_\mathcal{B}(1_\mathcal{B},S_0)$ such that:
\begin{enumerate}
\item[(SCA1)] $m\circ (\eta\otimes \id)\circ l_S^{-1} = \id =m\circ (\id \otimes \eta)\circ r_S^{-1}$ as a map $S\rightarrow S$;
\item[(SCA2)]  $m \circ (m \otimes \id) = m \circ (\id \otimes m)\circ \al_{S,S,S}$ as a map $(S\otimes S)\otimes S \rightarrow S$,
\end{enumerate}
where $m:S\otimes S \rightarrow S$ is a map defined linearly by $\{m_{i,j}\}_{i,j =0,1}$,
and 
\begin{enumerate}
\item[(SCA3)] $m_{j,i} \circ B_{S_i,S_j} = (-1)^{ij} m_{i,j}$ as a map $S_i\otimes S_j \rightarrow S_{i+j}$ for any $i,j \in \Z_2$.
\end{enumerate}
If a supercommutative algebra object in $\mathcal{B}$ consists of only the even part, i.e., $S_1=0$,
then the triple $(S_0, m_{0,0}, \eta)$ is called a {\it commutative algebra object} in $\mathcal{B}$.

Let $\mathcal{C}$ and $\mathcal{D}$ be braided tensor categories.
A {\it lax monoidal} functor between them is a functor $F:\mathcal{C} \rightarrow \mathcal{D}$ and a morphism 
$$\epsilon:1_{{\mathcal{D}}} \rightarrow F(1_{\mathcal{C}})$$
and a natural transformation 
$$\mu_{M,N}:F(M)\otimes_\mathcal{D} F(N) \rightarrow F(M\otimes_{\mathcal{C}} N) \;\;\;\;\;\;\;\;(\text{for all } M,N \in \mathcal{C})$$
such that:
\begin{enumerate}
\item[LM1)]
For all objects $M,N,L \in \mathcal{C}$, the following diagram commutes
\begin{align*}
\begin{array}{ccc}
    (F(M) \otimes_{\mathcal{D}} F(N)) \otimes_{\mathcal{D}} F(L)
      &\stackrel{\al^{\mathcal{D}}_{F(M),F(N),F(L)}}{\longrightarrow}&
    F(M) \otimes_{\mathcal{D}}( F(N)\otimes_{\mathcal{D}} F(L) )
    \\
   \downarrow^{\mu_{M,N} \otimes id}
%    {}^{\mathllap{\mu_{M,N} \otimes id}}\downarrow 
      && 
    \downarrow^{{id\otimes \mu_{N,L}}}
    \\
    F(M \otimes_{\mathcal{C}} N) \otimes_{\mathcal{D}} F(L)
     &&
    F(M) \otimes_{\mathcal{D}} F(N \otimes_{\mathcal{C}} L)
    \\
   \downarrow^{\mu_{M \otimes_{\mathcal{C}} N , L}}
%    {}^{\mathllap{\mu_{M \otimes_{\mathcal{C}} N , L} } }\downarrow 
      && 
    \downarrow^{{\mu_{ M, N \otimes_{\mathcal{C}} L  }}}
    \\
    F((M \otimes_{\mathcal{C}} N) \otimes_{\mathcal{C}} L)
      &\underset{F(\al^{\mathcal{C}}_{M,N,L})}{\longrightarrow}&
    F(M \otimes_{\mathcal{C}} (N \otimes_{\mathcal{C}} L))
\end{array}
\end{align*}
\item[LM2)]
For all $M \in \mathcal{C}$ the following diagram commute
\begin{align*}
\begin{array}{ccc}
    1_{\mathcal{D}} \otimes_{\mathcal{D}} F(M)
      &\overset{\epsilon \otimes id}{\longrightarrow}&
    F(1_{\mathcal{C}}) \otimes_{\mathcal{D}} F(M)
    \\
    {}^{{l^{\mathcal{D}}_{F(M)}}}\downarrow 
      && 
    \downarrow^{{\mu_{1_{\mathcal{C}}, M }}}
    \\
    F(M) 
      &\overset{F(l^{\mathcal{C}}_M )}{\longleftarrow}&
    F(1 \otimes_{\mathcal{C}} M  )
\end{array}
\end{align*}
and
\begin{align*}
\begin{array}{ccc}
    F(M) \otimes_{\mathcal{D}}  1_{\mathcal{D}}
      &\overset{id \otimes \epsilon }{\longrightarrow}&
    F(M) \otimes_{\mathcal{D}}  F(1_{\mathcal{C}}) 
    \\
    {}^{{r^{\mathcal{D}}_{F(M)}}}\downarrow 
      && 
    \downarrow^{{\mu_{M, 1_{\mathcal{C}} }}}
    \\
    F(M) 
      &\overset{F(r^{\mathcal{C}}_M )}{\longleftarrow}&
    F(M \otimes_{\mathcal{C}} 1 ),
\end{array}
\end{align*}
where $l^\mathcal{C},r^\mathcal{C},l^\mathcal{D},r^\mathcal{D}$ denote the left and right unitors of
the two monoidal categories, respectively.
\end{enumerate}
Note that if $\epsilon$ and all $\mu_{M,N}$ are isomorphism, then $F$ is called a {\it monoidal functor} (sometimes called a strong monoidal functor).

A lax monoidal functor $F:\mathcal{C}\rightarrow \mathcal{D}$ is called a {\it lax braided monoidal functor} if
\begin{align*}
\begin{array}{ccc}
    F(M) \otimes_{\mathcal{D}}  F(N)
      &\overset{B_{F(M),F(N)}^\mathcal{D}}{\longrightarrow}&
    F(N) \otimes_{\mathcal{D}}  F(M)
    \\
    {}^{{\mu_{M,N}}}\downarrow 
      && 
    \downarrow^{{\mu_{N,M}}}
    \\
    F(M\otimes_{\mathcal{C}} N) 
      &\overset{F(B_{M,N}^\mathcal{C})}{\longrightarrow}&
    F(N\otimes_{\mathcal{C}} M)
\end{array}
\end{align*}
for all objects $M,N \in \mathcal{C}$.

The following lemma is clear from the definition.
\begin{lem}
\label{lem_lax}
Let $F:\mathcal{C} \rightarrow \mathcal{D}$ be a lax braided monoidal functor
and $(S=S_0\oplus S_1, \{m_{i,j}\}_{i,j =0,1}, \eta)$ be a supercommutative algebra object in $\mathcal{C}$.
Then, $F(S_0)\oplus F(S_1)$ is naturally equipped with the structure of a supercommutative algebra object in $\mathcal{D}$
by setting
$$
\eta^F:1_{\mathcal{D}} \overset{\epsilon}{\rightarrow} F(1_\mathcal{C})  \overset{ F(\eta)}{\rightarrow} F(S_0)
$$
and for $i,j =0,1$
$$
m_{i,j}^F: F(S_i)\otimes_{\mathcal{D}} F(S_j)  \overset{\mu_{S_i,S_j}}{\rightarrow} F(S_i\otimes_{C} S_j )  
\overset{ F(m_{i,j})}{\rightarrow} F(S_{i+j}).
$$
\end{lem}

\subsection{Graded twist of braided tensor categories}
\label{sec_twist_intro}
In this section, we will review how to construct a new braided tensor category using an abelian cocycle
and give examples of braided tensor categories and (super) commutative algebra objects, which play an important role in this paper.

Let $A$ be a finite abelian group, and let $\om:A\times A\times A\rightarrow \C^\times$
and $c:A\times A \rightarrow \C^\times$ be maps satisfying
\begin{align}
\om(\al_1+\al_2,\al_3,\al_4)\om(\al_1,\al_2,\al_3+\al_4)
&=\om(\al_1,\al_2,\al_3)\om(\al_1,\al_2+\al_3,\al_4)\om(\al_2,\al_3,\al_4)\nonumber \\
\om(\al_2,\al_3,\al_1)c(\al_1,\al_2+\al_3)\om(\al_1,\al_2,\al_3)
&=c(\al_1,\al_3)\om(\la_2,\al_1,\al_3)c(\al_1,\al_2)\label{eq_ab_co} \\
\om(\al_3,\al_1,\al_2)^{-1}c(\al_1+\al_2,\al_3)\om(\al_1,\al_2,\al_3)^{-1}
&=c(\al_1,\al_3)\om(\la_1,\al_3,\al_2)^{-1}c(\al_2,\al_3)\nonumber 
\end{align}
for any $\al_1,\al_2,\al_3,\al_4 \in A$.
Let $Z_{\ab}^3(A,\C^\times)$ be the set of all pairs of functions $(\om,c)$ satisfying \eqref{eq_ab_co},
which is an abelian group with respect to the pointwise multiplication,
and elements of this group are called {\it abelian cocycles}.
Any map $k:A\times A \rightarrow \C^\times$ defines an abelian cocycle $(d_2(k),c_k)$ by
\begin{align}
(d_2 k)(\al_1,\al_2,\al_3)&=
k(\al_2,\al_3)k(\al_1+\al_2,\al_3)^{-1}k(\al_1,\al_2+\al_3)k(\al_1,\al_2)^{-1},\label{eq_ab_bo} \\
c_k(\al_1,\al_2)&=k(\al_1,\al_2)k(\al_2,\al_1)^{-1},\nonumber
\end{align}
for $\al_1,\al_2,\al_3 \in A$.
Let $B_\ab^3(A,\C^\times)$ be the subgroup of {\it abelian coboundaries},
that is, of the abelian cocycles defined by \eqref{eq_ab_bo}.
The group $H_\ab^3(A,\C^\times)=Z_\ab^3(A,\C^\times) / B_\ab^3(A,\C^\times)$
is called the {\it abelian cohomology group} of $A$ with coefficients in $\C^\times$.

Abelian cocycles $Z_\ab^3(A,\C^\times)$ can be used to construct a new braided tensor category
from an $A$-graded braided tensor category.
{\it An $A$-graded category} $C$ is a $\C$-linear category with full subcategories
$C_a$ for $a \in A$, such that any object $M$ in $C$ admits a unique (up to isomorphism) decomposition
$M\cong \bigoplus_{a \in A}M_a$ with $M_a \in C_a$ 
and there are no nonzero morphisms between objects in $C_a$ and $C_b$ for any $a\neq b$.
We say that $C$ is {\it an $A$-graded braided tensor category} if in addition it is a braided tensor category,
such that $1_C \in C_0$ and the monoidal structure satisfies $M \otimes N \in C_{ab}$
for all homogeneous objects $M \in C_a$ and $N \in C_b$.
For a semisimple braided tensor category $C$, there is a universal grading (see \cite{BNY}).

Let $(\mathcal{B},\al,B)$ be an $A$-graded braided tensor category and $(\om,c) \in Z_{\ab}^3(A,\C^\times)$.
Let us define a new braided tensor category $\mathcal{B}^{\om,c}$ as follows.
As a category, it is the same as $\mathcal{B}$ and the bifunctor $\otimes$ and
the unit object $1_\mathcal{B}$ are also the same as those in $\mathcal{B}$.
The differences are in the new associativity isomorphism $\al^\om$
and the unit morphisms $l^\om$, $r^\om$ and
the braiding $B^c$.
They are defined by the formulas,
\begin{align*}
\al_{M_1,M_2,M_3}^\om &= \om(\al_1,\al_2,\al_3) \al_{M_1,M_2,M_3}
:(M_1\otimes M_2)\otimes M_3\rightarrow M_1\otimes (M_2\otimes M_3)\\
l_{M_1}^\om&=\om(0,0,\al_1)^{-1}l_{M_1}: 1_{\mathcal{B}}\otimes M_1 \rightarrow M_1\\
r_{M_1}^\om&=\om(\al_1,0,0)r_{M_1} :M_1 \otimes 1_{\mathcal{B}} \rightarrow M_1\\
B_{M_1,M_2}^c &= c(\al_1,\al_2) B_{M_1,M_2} : M_1\otimes M_2\rightarrow M_2\otimes M_1
\end{align*}
for any $\al_i \in A$ and $M_i \in \mathcal{B}_{\al_i}$ ($i=1,2,3$).
We note that by \cite[Remark 2.6.3]{EGNO}, we may assume that the abelian cocycle satisfies $\om(\al,0,0)=\om(0,0,\be)=1$ for any $\al,\be \in A$, which is called a {\it normalized cocycle}.

A category of $A$-graded vector spaces give an example of an $A$-graded braided tensor category.
Let $\Vect_A$ denote the category of $A$-graded vector spaces over $\C$, i.e., vector spaces $V$ with a
decomposition $V=\bigoplus_{\al \in A}V_\al$.
Morphisms in this category are linear maps which preserve the grading. 
Define the tensor product on this category by the formula 
\begin{align*}
(V\otimes W)_\al=\bigoplus_{\be \in A} V_{\be}\otimes W_{\al-\be}
\end{align*}
and the unit object ${\bf1}_{\Vect_A}=\C \delta_0$,
where for $\al \in A$ $\C \delta_\al$ is a one-dimensional vector space
defined by
\begin{align*}
(\C \delta_\al)_\be=\begin{cases}
\C & (\text{if }\al=\be)\\
0 & (\text{otherwise}).
\end{cases}
\end{align*}
Then $\Vect_A$ is a braided tensor category with the associativity isomorphism, unit morphisms, and the braiding being the identities.
Since $\Vect_A$ has an $A$-grading as a braided tensor category,
for any $(\om,c) \in Z_\ab^3(A,\C^\times)$, $\Vect_A^{\om,c}$ is a braided tensor category.
%The associative isomorphism and the braiding are defined by
%\begin{align*}
%\al_{\delta_{\al},\delta_{\be},\delta_{\ga}}^\om = 
%\om(\al,\be,\ga)\mathrm{id}_{\de_{\al+\be+\ga}}&: 
%(\delta_{\al} \otimes \delta_{\be})\otimes \delta_{\ga}
%\rightarrow 
%\delta_{\al} \otimes (\delta_{\be}\otimes \delta_{\ga}),\\
%B_{\delta_\al,\delta_\be}^c=c(\al,\be)\mathrm{id}_{\de_{\al+\be}}&:\delta_{\al}\otimes \delta_{\be} 
%\rightarrow \delta_{\be}\otimes \delta_{\al},
%\end{align*}
%where $\al,\be,\ga \in A$.
%Then, \eqref{eq_ab_co} is equivalent to the axiom of the braided tensor category,
%and thus $C_{A}^{\om,c}$ is a braided tensor category.
We note that the equivalence classes of $\Vect_{A}^{\om,c}$ as braided tensor categories 
only depends on the cohomology classes in $H_\ab^3(A,\C^\times)$.

Due to Eilenberg and Mac Lane, any abelian cocycle can be obtained from quadratic forms on $A$.
More precisely, a quadratic form on $A$ with values in $\C^\times$ is a map
$Q:A\rightarrow \C^\times$ such that $Q(\al)=Q(-\al)$ and the symmetric function
$$
b(\al,\be)=\frac{Q(\al+\be)}{Q(\al)Q(\be)}
$$
is a bicharacter, i.e., $b(\al_1+\al_2,\be)=b(\al_1,\be)b(\al_2,\be)$ for any $\al_1,\al_2,\be \in A$.
Set $\mathrm{Quad}(A)$ be the group of quadratic forms on $A$.

For any abelian cocycle $(\om,c) \in Z_\al^3(A,\C^\times)$,
let $Q_{\om,c}:A\rightarrow \C^\times$ be a function defined by
$Q_{\om,c}(\al)=c(\al,\al)$ for $\al \in A$.
Then, it is easy to check that $Q_{\om,c}$ is a quadratic form
and the map $\mathrm{EM}:H_\ab^3(A,\C^\times)\rightarrow \mathrm{Quad}(A),\;(\om,c)\mapsto Q_{\om,c}$
is well-defined. Then, Eilenberg and Mac Lane show the following result (see \cite{EGNO}):
\begin{thm}[Eilenberg and Mac Lane]
\label{thm_coho}
The above map $\mathrm{EM}:H_\ab^3(A,\C^\times)\rightarrow \mathrm{Quad}(A)$ is an isomorphism of abelian groups.
\end{thm}
We will denote by $\Vect_{A}^Q$ the braided tensor category associated with
$\mathrm{EM}^{-1}(Q) \in H_\ab^3(A,\C^\times)$ for $Q \in \mathrm{Quad}(A)$.
The following proposition in \cite[Section 8.4]{EGNO} is useful:
\begin{lem}
\label{lem_degenerate_quad}
Let $Q \in \mathrm{Quad}(A)$ and $(\om,c) \in Z_\ab^3(A,\C^\times)$ be a corresponding abelian cocycle.
The following conditions are equivalent:
\begin{enumerate}
\item
There exists a bicharacter $B:A\times A\rightarrow \C^\times$ such that
$Q(x)=B(x,x)$ for all $x \in A$.
\item
The 3-cocycle $\om$ is trivial, i.e., there exists a map $k:A\times A\rightarrow \C^\times$
such that $\om =d_2(k)$.
\end{enumerate}
\end{lem}

We will construct (super) commutative algebra objects in $\Vect_{A}^Q$.
Let $Q$ be a quadratic form on $A$.
A pair of a subgroup $I \subset A$ and a group homomorphism $p:I\rightarrow \Z_2$
 is called a {\it super isotropic subspace} of $(A,Q)$ if it satisfies
\begin{align}
Q(\al) = (-1)^{p(\al)}, \;\;\;\;\;\text{ for any }\al \in I. \label{eq_isotropic}
\end{align}
If a {super isotropic subspace} $(I,p)$ satisfies $p(\al)=0$ for any $\al \in I$,
i.e., the restriction of the quadratic form on $I$ is trivial,
then $I$ is called an {\it isotropic subspace} of $(A,Q)$.

Let $(I,p)$ be a {super isotropic subspace} of $A$
and set 
\begin{align*}
S(I)_{0}=\bigoplus_{\al \in A,p(\al)=0}\C\delta_{\al},\\
S(I)_{1}=\bigoplus_{\al \in A,p(\al)=1} \C\delta_{\al},
\end{align*}
objects in $\Vect_{A}^Q$.
Let $(\om_Q,c_Q) \in Z_\ab^3(A,\C^\times)$ be the normalized abelian cocycle associated with the quadratic form $Q$
by Theorem \ref{thm_coho}.

\begin{lem}
\label{lem_existence_coboundary}
There is a function $k:I \times I \rightarrow \C^\times$ such that:
\begin{align}
k(\al,0)&=1=k(0,\al),\nonumber \\
\om_Q(\al_1,\al_2,\al_3)&=
k(\al_2,\al_3)k(\al_1+\al_2,\al_3)^{-1}k(\al_1,\al_2+\al_3)k(\al_1,\al_2)^{-1},\label{eq_resolve} \\
c_Q(\al_1,\al_2)&=(-1)^{p(\al_1)p(\al_2)}k(\al_1,\al_2)k(\al_2,\al_1)^{-1}, \nonumber
\end{align}
for any $\al,\al_1,\al_2,\al_3 \in I$.
\end{lem}
\begin{proof}
Let $B_I:I\times I\rightarrow \C^\times$
be a bicharacter defined by 
$B_I(\al,\be)= (-1)^{p(\al)p(\be)}$ for $\al,\be \in I$.
Then, by \eqref{eq_isotropic}, 
$B_I(\al,\al)=Q(\al)$ for any $\al \in I$
and by Lemma \ref{lem_degenerate_quad},
$\mathrm{EM}(1_I,B_I)=Q|_I \in \mathrm{Quad}(I)$,
where $Q|_I$ is the restriction of the quadratic form $Q$ on $I \subset A$
and $1_I$ is the trivial 3-cocycle on $I$.
Since the restriction of $(\om_Q,c_Q) \in Z_\ab^3(A,\C^\times)$
on $I \subset A$
also satisfies 
$\mathrm{EM}(\om_Q|_I,c_Q|_I)=Q|_I$,
by Theorem \ref{thm_coho} there exists $k:I\times I\rightarrow \C^\times$ such that:
\begin{align*}
\om_Q(\al_1,\al_2,\al_3)&=
k(\al_2,\al_3)k(\al_1+\al_2,\al_3)^{-1}k(\al_1,\al_2+\al_3)k(\al_1,\al_2)^{-1},\\
c_Q(\al_1,\al_2)&=B_I(\al_1,\al_2)k(\al_1,\al_2)k(\al_2,\al_1)^{-1}
\end{align*}
for any $\al_1,\al_2,\al_3 \in I$.
Since $\om_Q$ is normalized,
$\om_Q(\al_1,0,\al_2)=k(0,\al_3)k(\al_1,0)^{-1}$
implies that $k(-,-)$ satisfies 
$k(\al,0)=1=k(0,\al)$ for any $\al \in I$.
\end{proof}

%By \cite[Remark 2.6.3]{EGNO}, we may assume that the abelian cocycle satisfies
%%$\om_Q(\al,0,\be)=1$ for any $\al,\be \in A$, a {\it normalized cocycle}.
%Since the full subcategory whose object is consisting of $\{\delta_\al\}_{\al \in I}$ is the supercategory by \eqref{eq_isotropic},
%it is easy to show that by \ref{thm_coho} there is a function $k:I \times I \rightarrow \C^\times$ such that:
%\begin{align}
%k(0,0)&=1,\nonumber \\
%\om_Q(\al_1,\al_2,\al_3)&=
%k(\al_2,\al_3)k(\al_1+\al_2,\al_3)^{-1}k(\al_1,\al_2+\al_3)k(\al_1,\al_2)^{-1},\label{eq_resolve} \\
%c_Q(\al_1,\al_2)&=(-1)^{p(\al_1)p(\al_2)}k(\al_1,\al_2)k(\al_2,\al_1)^{-1}, \nonumber
%\end{align}
%for any $\al_1,\al_2,\al_3 \in N$.
%We note that since $\om_Q$ is normalized, $k(\al,0)=1=k(0,\be)$ for any $\al,\be \in N$.

Define maps $m_{i,j}: S(N)_i\otimes S(N)_j \rightarrow S(N)_{i+j}$
by the linear extensions of
$k(\al,\be)\mathrm{id}_{\C\delta_{\al+\be}}:
\C\delta_{\al} \otimes \C\delta_{\be} \rightarrow \C\delta_{\al+\be}
$ for $\al \in p^{-1}(i)$ and $\be \in p^{-1}(j)$,
and 
$\eta:\C \delta_0 \rightarrow S(N)_0$ by the inclusion map.
Then, we have:
\begin{lem}
\label{lem_commutative}
Let $(I,p)$ be a {super isotropic subspace} of $(A,Q)$.
Then, the above triple $(S(I)=S(I)_0\oplus S(I)_1,\{m_{i,j}\}_{i,j =0,1},\eta)$
is a supercommutative algebra object in $\Vect_{A}^Q$.
\end{lem}
\begin{rem}
Let $k':I\times I\rightarrow \C^\times$ be another map satisfying \eqref{eq_resolve}.
Then, $f(\al,\be)=k(\al,\be)k'(\al,\be)^{-1}$ satisfies $d_2(f)(\al,\be,\ga)=1$ for any $\al,\be,\ga \in I$.
Thus, $f$ is a normalized 2-cocycle of $I$ (in the sense of the usual group cohomology).
Since $f(\al,\be)=f(\be,\al)$ by \eqref{eq_resolve}, $f$ is coboundary (see for example \cite{FLM}).
Thus, the non-trivial superalgebra structure on $S(I)$ is unique up to isomorphism.
\end{rem}

It is noteworthy that the braided tensor category $\Vect_{A}^Q$ naturally arises from
lattice vertex operator algebras. We end this section by recalling this fact.
{\it An even lattice} is a finite rank free abelian group $M$
equipped with a non-degenerate symmetric bilinear form $(-,-):M\times M \rightarrow \Z$ 
such that
$(\al,\al) \in 2\Z$ for any $\al \in M$. 
Let $M$ be an even lattice (not assumed to be positive-definite).
We can extend the bilinear form $(-,-)$ linearly on the vector space $M\otimes_\Z \R$.
Set $M^\vee = \{\la \in M\otimes_\Z \R \mid (\la,\al) \in \Z \text{ for any }\al \in M \}$,
the dual lattice.
Define a quadratic form $Q_M$ on $M^\vee/M$ by
$$
Q_M(\la)=\exp(\pi i (\la,\la)) \text{ for }\la \in M^\vee.
$$
Since $M$ is an even lattice, $Q_M$ is well-defined.

Let $V_M$ be the lattice vertex operator algebra.
Then, by \cite{DL}, we have:
\begin{prop}
\label{prop_lattice}
The module category of the lattice vertex operator algebra $V_M$ is equivalent to $\Vect_{M^\vee/M}^{Q_M}$ as braided tensor categories.
\end{prop}

\subsection{Quasi-triangular Hopf algebra}
\label{sec_quasi}

In this section, we recall the definition of a quasi-triangular Hopf algebra and its standard results from \cite{Ka,ES,KS}.
Let $H$ be a Hopf algebra. We denote the unit morphism of $H$ by $\eta: \C \rightarrow H$ 
and the counit $\epsilon: H \rightarrow \C$, the multiplication $m:H\otimes H \rightarrow H$,
the co-multiplication $\Delta:H \rightarrow H\otimes H$
and the antipode $S:H\rightarrow H$.
Throughout of this paper, the antipode is assumed to be invertible.

Let $(H,\eta,\epsilon,m,\Delta,S)$ be a Hopf algebra
and $P_{21}:H\otimes H \rightarrow H\otimes H$ a transposition defined by $a \otimes b \mapsto b \otimes a$
for $a,b \in H$ and set 
\begin{align*}
m_\op = m \circ P_{21}: H \otimes H \rightarrow H,\\
\Delta^\cop = P_{21}\circ \Delta: H \otimes H \rightarrow H.
\end{align*}
Then, $(H,\eta,\epsilon,m_\op,\Delta,S^{-1})$, 
$(H,\eta,\epsilon,m,\Delta^\cop,S^{-1})$ and $(H,\eta,\epsilon,m_\op,\Delta^\cop,S)$
 are again Hopf algebras,
which are denoted by $H_\op$, $H^\cop$ and $H_\op^\cop$, respectively.

A Hopf algebra $H$ is {\it quasi-triangular} if there exists an invertible element $R \in H\otimes H$ such that
\begin{enumerate}
\item
$R \Delta(x) R^{-1} = \Delta^\op(x)$ for all $x \in H$;
\item
$(\Delta\otimes 1)(R) = R_{13}R_{23}$;
\item
$(1\otimes \Delta)(R) = R_{13}R_{12}$.
\end{enumerate}
The element $R \in H\otimes H$ is called an {\it R-matrix}.

Let $(H,R)$ be a quasi-triangular Hopf algebra.
Then, the category of left $H$-modules, denoted by $(H,R)\modu$, inherits a natural braided tensor category structure.
The braiding $B_{M,N}: M\otimes N \rightarrow N \otimes M$ is defined by
$$
B_{M,N} = P_{M,N}\circ (\rho_M \otimes \rho_N)(R),
$$
where $\rho_M:H\rightarrow \End\, M$ and 
$\rho_N:H\rightarrow \End\, N$ are the structure homomorphisms.
Quasi-triangular Hopf algebras $(H_1,R_1)$ and $(H_2,R_2)$ are said to be isomorphic if
there exits a Hopf algebra isomorphism $f: H_1 \rightarrow H_2$ such that $(f \otimes f )(R_1)=R_2$.
If quasi-triangular Hopf algebras $(H_1,R_1)$ and $(H_2,R_2)$ are isomorphic,
then their module categories $(H_1,R_1)\modu$ and $(H_2,R_2)\modu$ are isomorphic as
braided tensor categories.

Importantly, an R-matrix of a Hopf algebra $H$ is not unique.
In fact, for an R-matrix $R \in H\otimes H$, $R_{21}^{-1} \in H\otimes H$ is again an R-matrix,
where $R_{21} = P_{21}(R)=\sum_{i} \be_i \otimes \al_i$ for $R = \sum_{i} \al_i \otimes \be_i$, and the quasi-triangular Hopf algebras $(H,R)$ and $(H,R_{21}^{-1})$ are not isomorphic in general.
The following lemma is clear from the definition:
\begin{lem}
\label{lem_rev}
Let $(H,R)$ be a quasi-triangular Hopf algebra.
Then, $(H,R_{21}^{-1})\text{-mod}$ and $((H,R)\modu)^\rev$
are equivalent as braided tensor categories.
\end{lem}
The following lemma is also clear from the definition:
\begin{lem}
\label{op_triangle}
Let $(H,R)$ be a quasi-triangular Hopf algebra.
Then, 
\begin{enumerate}
\item
$R^{-1}$ and $R_{21}$ are R-matrices of $H_\op$;
\item
$R^{-1}$ and $R_{21}$ are R-matrices of $H^\cop$.
\end{enumerate}
\end{lem}

We will give an example of a commutative algebra object in a braided tensor category,
introduced in the previous section.
Let $(H,\eta,\epsilon,m,\Delta)$ be a finite-dimensional Hopf algebra
and $H^\vee$ is a dual vector space of $H$.
Then, $H^\vee$ is canonically a Hopf algebra, called a {\it dual Hopf algebra}.
The unit, counit, multiplication and co-multiplication are given by
\begin{align}
\eta^\vee &= \epsilon \in H^\vee, \nonumber \\
\epsilon^\vee(f) &= f(\eta(1)), \nonumber \\
m^\vee(f\otimes g) &= (f \otimes g) \circ \Delta \in H^\vee \label{eq_multiplication}  \\
\Delta^\vee (f) &= f \circ m \in H^\vee\otimes H^\vee \nonumber
\end{align}
for $f,g \in H$.
Let $\langle - \rangle: H^\vee \otimes H \rightarrow \C$ be the canonical pairing.
Then, it satisfies
\begin{align}
\langle m^\vee(f\otimes g), a  \rangle &= \langle f\otimes g, \Delta(a) \rangle, \nonumber \\
\langle f, m(a\otimes b) \rangle &= \langle \Delta^\vee(f), a\otimes b \rangle, \label{eq_pairing}
\end{align}
for $a,b \in H$ and $f,g \in H^\vee$.

%Furthermore, $H^\vee$ is an $H$-$H$ bimodule where the bimodule structure is defined by
%$$
%a \rhd f \lhd b = f(b-a)
%$$
%for $a,b \in H$ and $f \in H^\vee$.
%Hence, $H^\vee$ is a left $H \otimes H$-module.
%Since the antipode map $S^{-1}$ gives a Hopf algebra isomorphism between $H_\op \rightarrow H^\cop$,

Furthermore, $H^\vee$ is a left $H\otimes H$ module where the left module structure is defined by
\begin{align}
(a \otimes b) \cdot f = f(S(b)-a)
\label{eq_bimodule}
\end{align}
for $a,b \in H$ and $f \in H^\vee$.

As we will see in the following lemma, it is natural to view $H^\vee$ as an $H\otimes H^\cop$-module.
\begin{lem}
\label{module_hom}
The unit $\eta^\vee: \C \rightarrow H^\vee$
and the multiplication $m^\vee: H^\vee \otimes H^\vee \rightarrow H^\vee$
are morphisms in $H\otimes H^\cop\modu$.
\end{lem}
\begin{proof}
Let $a,b \in H$, $f,g \in H^\vee$ and $x \in H$.
By
$(a\otimes b) \eta^\vee(1)=(a\otimes b) \epsilon =\epsilon(S(b) - a)
=\epsilon(S(b))\epsilon(-)\epsilon(a) =\epsilon(b)\epsilon(-)\epsilon(a) =\eta^\vee((a\otimes b)1)$,
$\eta^\vee: \C \rightarrow H^\vee$ is 
a $H\otimes H^\cop$-module homomorphism.
By
\begin{align*}
\langle (a\otimes b)\cdot m^\vee(f\otimes g), x  \rangle
&=\langle m^\vee(f\otimes g), S(b) x a  \rangle\\
&=\langle f\otimes g, \Delta(S(b) x a)  \rangle\\
&=\langle f\otimes g, \Delta(S(b)) \Delta(x) \Delta(a))  \rangle\\
&=\langle f\otimes g, S (\Delta^\cop(b)) \Delta(x) \Delta(a)) \rangle\\
&=\langle (a \otimes b) \cdot (f\otimes g), \Delta(x) \rangle\\
&=\langle m^\vee( (a \otimes b)\cdot (f\otimes g)),x \rangle\\
\end{align*}

\end{proof}

Let us assume that $R=\sum_{i} \al_i \otimes \be_i \in H\otimes H$ is an R-matrix of $H$. Then, by Lemma \ref{op_triangle}, $(H^\cop,R^{-1})$ is a quasi-triangular Hopf algebra.
Let $R^{-1} = \sum_{j} \al'_j \otimes \be'_j$
and set $R^{1,-1} \equiv \sum_{i,j} \al_i \otimes \al'_j \otimes \be_i \otimes \be'_j \in H\otimes H\otimes H\otimes H$.
Then, $R^{1,-1}$ is an R-matrix of $H \otimes H^\cop$.
%We will abuse notation and write $R$ as $R_1 \otimes R_2$
%and $R_{13}R_{42}=R_1\otimes 

The following result is mentioned in many literatures
(see \cite{CW} or \cite{DM}):
\begin{prop}
\label{dual_commutative}
Let $(H,R)$ be a quasi-triangular Hopf algebra.
Then, $(H^\vee,m^\vee,\eta^\vee)$ is a commutative algebra object in
$(H \otimes H^\cop,R^{1,-1})\modu$.
\end{prop}
\begin{proof}
As discussed above, $H^\vee$ is an $H \otimes H^\cop$-module
and  by Lemma \ref{module_hom},
 $m^\vee$ and $\eta^\vee$ are $H \otimes H^\cop$-module homomorphisms.
Since the dual Hopf algebra is a unital associative algebra,
(CA1) and (CA2) hold.
To show (CA3), we recall that $(S\otimes S) (R) = R$ holds for any R-matrix.
Then, by \eqref{eq_pairing}, for any $x \in H$ and $f,g \in H^\vee$, we have
\begin{align*}
\langle  m^\vee(P\circ R^{1,-1}\cdot (f\otimes g)), x  \rangle 
&=
\langle  P\circ R^{1,-1}\cdot (f\otimes g), \Delta(x)  \rangle \\
&=
\langle  R^{1,-1}\cdot (f\otimes g), \Delta(x)^\cop  \rangle \\
&=
\langle  f\otimes g, S(R^{-1}) \Delta(x)^\cop R \rangle \\
&=
\langle  f\otimes g, R^{-1} \Delta^\cop(x) R \rangle\\
&=\langle  f\otimes g, \Delta(x) \rangle\\
&=\langle m^\vee( f\otimes g), x \rangle.
\end{align*}
Hence, (CA3) holds.
\end{proof}

\subsection{Affine vertex algebra and Drinfeld category}
\label{sec_affine1}
In the previous section, braided tensor categories were constructed by using quasi-trianglular Hopf algebras. Another way to construct braided tensor categories is to use the representation theory of vertex algebras \cite{Hua}.
In the case of affine vertex algebras, the resulting braided tensor category coincides with   Drinfeld categories defined from the Knizhnik-Zamolodchikov equations \cite{KZ,Dr1,TK,ES}.
In this section, we briefly explain the relationship between affine vertex algebras and  Drinfeld categories based on \cite{EFK}.

Let $k \in \C\setminus \Q$ and $\fg$ be a simple complex Lie algebra.
We will use the notations introduced in section \ref{sec_notation}.
%% $\fh$ a Cartan subalgebra of $\fg$, $\Delta$ the set of roots, $\lP \subset \fh^*$ the weight lattice, $\lQ \subset \lP$ the root lattice, $\lP^+$ the dominant integral weights and
%	Let $(\cdot,\cdot)$ be the natural bilinear form on $\fh$, which is normalized as $(\alpha,\alpha)=2$ for long roots $\alpha$.
%For each $\alpha \in \Delta$, let $\alpha^\vee := 2\alpha/(\alpha,\alpha)$ be the coroot. 
%We fix a set $\Pi$ of simple roots and let $\lP^+$ be the set of dominant weights with respect to $\Pi$.
Let $\hat{\g}=\g\otimes \C[t,t^{-1}] \oplus \C c$ be the affine Lie algebra which is naturally
$\Z$-graded: $\mathrm{deg}(a\otimes t^n)=-n$,
$\mathrm{deg}(c)=0$.
Define $\hat{\g}^+ = \g \otimes t \C[t],
\hat{\g}^- = \g \otimes t^{-1} \C[t^{-1}],
\hat{\g}^0 = \g \oplus \C c,$ and 
$
\hat{\g}^{\geq 0} =\hat{\g}^+ \oplus \hat{\g}^{0}.$
Let $L(\la)$ denote the finite dimensional irreducible representation of $\g$ with
highest weight $\ga \in \lP^+$, which can be regarded as a $\hat{\g}^{\geq 0}$-module by letting
$\hat{\g}^+$ act trivially and the central element $c$ act by the scalar $k$.
The induced module $L_{\g,k}(\la) =\mathrm{Ind}_{\hat{\g}^{\geq 0}}^{\hat{\g}} L({\la})$ is called {\it the Weyl module}.
It has a natural $\mathbb{C}$-grading, $L_{\g,k}(\la)=\bigoplus_{n=0}^\infty (L_{\g,k}(\la))_{\Delta(\la) + n}$
with $\Delta(\la)= \frac{(\la+2\rho,\la)}{2(k+h^\vee)}$
and $(L_{\g,k}(\la))_{\Delta(\la)}=L(\la)$.

Then, $L_{{\g,k}}(0)$ inherits a vertex operator algebra structure, called an {\it affine vertex operator algebra}
and $L_{\g,k}(\la)$ is a $L_{\g,k}(0)$-module.
We consider a category of $L_{\g,k}(0)$-modules
whose object is a direct sum of $L_{\g,k}(\la)$'s for $\la \in \lP^+$
and morphisms are $L_{\g,k}(0)$-module homomorphisms.
We denote this $\C$-linear abelian category by $D(\g,k)$.

Following \cite{EFK}, we will briefly review the fact that the abelian category $D(\g,k)$ has a
balanced braided tensor category structure.
Let $\la_0,\la_1,\la_2,\la_3,\la \in \lP^+$.
It is well-known that
for any $\g$-module homomorphisms,
 $f \in \mathrm{Hom}_\g(L({\la_1})\otimes L({\la_2}),L({\la_0}))$,
there exists a unique intertwining operator
$$
I_f(- ,z): L_{\g,k}(\la_1) \rightarrow \mathrm{Hom}\,(L_{\g,k}(\la_2),L_{\g,k}(\la_0))[[z^\C]]
$$
such that  it can be written as
\begin{align}
I_f(a,z)=\sum_{n \in \Z}a(n-\Delta)z^{-n-1+\Delta} \text{ with }\Delta =\Delta(\la_0)-\Delta(\la_1)-\Delta(\la_2)
\label{eq_inter}
\end{align}
for any $a \in L_{\g,k}(\la_1)$
and 
$$a(-\Delta-1)b =f(a\otimes b)$$
 for any $a \in L_{\g,k}(\la_1)_{\Delta(\la_1)} = L({\la_1})$
and $b \in L_{\g,k}(\la_2)_{\Delta(\la_2)} = L({\la_2})$ (See \cite[Theorem 3.1.1]{EFK}
or \cite[Theorem 1.5.3]{FZ}).

We denote the space of $\g$-module homomorphisms,
 $\mathrm{Hom}_\g(L({\la_1})\otimes L({\la_2}),L({\la_0}))$ by $V_{\la_1\la_2}^{\la_0}$
and
$\mathrm{Hom}_\g(L({\la_1})\otimes L({\la_2})\otimes L({\la_3}),L({\la_0}))$ by $V_{\la_1\la_2\la_3}^{\la_0}$.
Then, define a tensor product by
$$
L_{\g,k}(\la_1)\otimes L_{\g,k}(\la_2)
= \bigoplus_{\al \in P^+} V_{\la_1\la_2}^\al L_{\g,k}(\al).
$$
Then,
by the natural $\C$-linear isomorphism
$
V_{\la_1\la_2\la_3}^{\la_0} \cong \bigoplus_{\al \in \lP^+} V_{\la_1\al}^{\la_0}\otimes V_{\la_2\la_3}^\al,
$
we have
\begin{align}
L_{\g,k}(\la_1)\otimes (L_{\g,k}(\la_2)\otimes L_{\g,k}(\la_3))
&\cong L_{\g,k}(\la_1) \otimes (\bigoplus_{\al} V_{\la_2\la_3}^\al\otimes L_{\g,k}(\al))\nonumber \\
&\cong \bigoplus_{\la_0}\bigoplus_{\al}V_{\la_1\al}^{\la_0}\otimes
V_{\la_2\la_3}^{\al}\otimes L_{\g,k}(\la_0)  \label{eq_decomposition} \\
&\cong V_{\la_1\la_2\la_3}^{\la_0}\otimes L_{\g,k}(\la_0).\nonumber
\end{align}

Set $Y_2 =\{(z_1,z_2) \in \C^2 \mid z_1 \neq z_2, z_1 \neq 0 \text{ and }z_2 \neq 0 \}$,
an open subset of $\C^2$.
Then, a composition of vertex operator defines a multi-valued holomorphic function
on $Y_2$ and the monodromy defines a braided tensor category structure on $D(\g,k)$.

More precisely,
set $Y_2^> = \{(z_1,z_2) \in Y_2 \mid |z_1|>|z_2|, |\mathrm{Arg }\,z_2| < \pi,
|\mathrm{Arg }\,z_1| < \pi \}$
and $Y_2^< = \{(z_1,z_2) \in Y_2 \mid |z_1| < |z_2|, |\mathrm{Arg }\,z_2| < \pi,
|\mathrm{Arg }\,z_1| < \pi \}$.
Let $Mul(Y_2^>)$ (resp. $Mul(Y_2^<)$) be the space of holomorphic functions on $Y_2^>$
(resp. $Y_2^<$) such that it has an analytic continuation to a multi-valued holomorphic function on $Y_2$.

Let $f \in V_{\la_1\al}^{\la_0}$ and $g \in V_{\la_2\la_3}^\al$
and $a_i \in L_{\g,k}(\la_i)$ ($i=1,2,3$) and $u$ be an element of the restricted dual space 
$L_{\g,k}(\la_0)^\vee = \bigoplus_{h} L_{\g,k}(\la_0)_{h}^*$.
Then, the composition of the intertwining operators
$$
u(I_f(a_1,z_1)I_g(a_2,z_2)a_3) \in \C[[z_1^\C,z_2^\C]]
$$
is absolutely convergent to a holomorphic function on $Y_2^>$
and has an analytic continuation to a multivalued holomorphic function on
$Y_2$.
This follows since the formal power series
$u(I_f(a_1,z_1)I_g(a_2,z_2)a_3)$ satisfies
the Knizhnik-Zamolodchikov equation, see \cite{KZ,TK,ES}.
Thus, we have a linear map
$$
C(u,a_1,a_2,a_3;z_1,z_2): V_{\la_1\al}^{\la_0} \otimes V_{\la_2\la_3}^\al \rightarrow Mul(Y_2^>)
$$
for $a_i \in L_{\g,k}(\la_i)$ and $u \in L_{\g,k}(\la_0)^\vee$.
Hence, by
$
V_{\la_1\la_2\la_3}^{\la_0} \cong \bigoplus_{\al \in \lP^+} V_{\la_1\al}^{\la_0}\otimes V_{\la_2\la_3}^\al,
$
we have 
$$
C(u,a_1,a_2,a_3;z_1,z_2): V_{\la_1\la_2\la_3}^{\la_0} \rightarrow Mul(Y_2^>)
$$
and similarly
$$
C(u,a_1,a_2,a_3;z_2,z_1): V_{\la_1\la_2\la_3}^{\la_0} \rightarrow Mul(Y_2^<)
$$
by changing the role of $z_1$ and $z_2$.

Let $\ga$ be a path from a point in $Y_2^>$ to a point in $Y_2^<$
and $A(\ga):Mul(Y_2^>)\rightarrow Mul(Y_2^<)$ a linear map defined by
the analytic continuation of a function in $Mul(Y_2^>)$ along the path $\ga$.
Then, by using the Knizhnik-Zamolodchikov equation,
there exists 
a unique linear isomorphism 
$$M(\ga): V_{\la_1\la_2\la_3}^{\la_0} \rightarrow V_{\la_2\la_1\la_3}^{\la_0}$$
such that for any $a_i \in V_{\g,k}(\la_i)$ and $u \in V_{\g,k}(\la_0)^\vee$
the following diagram commutes:
$$
\begin{array}{ccc}
V_{\la_1\la_2\la_3}^{\la_0}    &\stackrel{M(\ga)}{\longrightarrow}&
V_{\la_2\la_1\la_3}^{\la_0} \\
\downarrow_{C(u,a_1,a_2,a_3;z_1,z_2)}
   &&
\downarrow_{C(u,a_2,a_1,a_3;z_2,z_1)}   \\
Mul(Y_2^>)
   &\stackrel{A(\ga)}{\longrightarrow}&
Mul(Y_2^<).
\end{array}
$$
It is noteworthy that $M(\ga)$ is independent of the choice of $a_i \in L_{\g,k}(\la_i)$ and $u \in L_{\g,k}(\la_0)^\vee$
and thus by \eqref{eq_decomposition} this linear map $M(\ga)$ gives ``a structure morphism of braided tensor category'',
$$
M(\ga): L_{\g,k}(\la_1)\otimes (L_{\g,k}(\la_2)\otimes L_{\g,k}(\la_3)) \rightarrow 
L_{\g,k}(\la_2)\otimes (L_{\g,k}(\la_1)\otimes L_{\g,k}(\la_3)).
$$

In fact, by setting $\la_3=0$,
it gives an isomorphism 
$L_{\g,k}(\la_1)\otimes L_{\g,k}(\la_2) \rightarrow L_{\g,k}(\la_2)\otimes L_{\g,k}(\la_1)$,
which is the braiding $B_{L_{\g,k}(\la_1),L_{\g,k}(\la_2)}$
and
the composition of the braiding and $M(\ga)$ gives
the associative isomorphism
$L_1\otimes (L_2\otimes L_3) 
\stackrel{1\otimes B}{\rightarrow} L_1\otimes (L_3 \otimes L_2)
\stackrel{M}{\rightarrow} L_3\otimes (L_1 \otimes L_2)
\stackrel{B}{\rightarrow}  (L_1 \otimes L_2)\otimes L_3$.
In this way, the monodromy of the intertwining operators define a braided tensor category structure on $D(\g,k)$ (see for example \cite{ES,Ka}).
This braided tensor category is introduced by Drinfeld by using the Knizhnik-Zamolodchikov equation and called a {\it Drinfeld category}.

We remark that we can consider antiholomorphic formal variable $\z_1,\z_2$.
Then, the formal power series 
$$
u(I_f(a_1,\z_1)I_g(a_2,\z_2)a_3) \in \C[[\z_1^\C, \z_2^\C]]
$$
is absolutely convergent to an antiholomorphic function on $Y_2^>$.
All of the above discussions can be carried out in parallel,
 and a braided tensor category can be defined.
We denote this antiholomorphic Drinfeld category by $\overline{D(\g,k)}$.
An object in $\overline{D(\g,k)}$ is written as $\overline{L_{\g,k}(\la)}$ to emphasize that it is 
an object in $\overline{D(\g,k)}$.

Since the holomorhic and antiholomorphic solutions of a KZ-equation
are the same on the real subspace $\R^2\cap Y_2$,
the associative isomorphisms of $D(\fg,k)$ and $\overline{D(\fg,k)}$ are the same.
Furthermore, the difference of the braidings is an inverse of each other.
This is because for $\al \in \C$ the monodromy of the holomorphic function $z^\al$
around the origin (counterclockwise) is $\exp(2\pi i \al)$,
and the monodromy of an antiholomorphic function  $\z^\al$ around the origin is $\exp( - 2\pi i \al)$.
Hence, we have:
\begin{lem}
For $k \in \C \setminus \Q$, 
$\overline{D(\fg,k)}$ and $D(\fg,k)^\rev$ are equivalent as braided tensor categories.
\end{lem}
We note that $D(\g,k)$ is a balanced braided tensor category with 
the balance 
\begin{align}
\theta_{L_{\g,k}(\la)} =\exp (2\pi i \Delta(\la)) = \exp\left(\pi i \frac{(\la+2\rho,\la)}{k+h^\vee}\right)
\label{eq_balance}
\end{align}
and the balance on $\overline{D(\g,k)}$ is given by
\begin{align}
\theta_{\overline{L_{\g,k}(\la)}} = \exp\left(- \pi i \frac{(\la+2\rho,\la)}{k+h^\vee}\right).
\label{eq_balance2}
\end{align}

Finally, we briefly review the fact that
for any commutative algebra object in the Drinfeld categories $D(\g,k)\otimes D(\g,k')$ with
$k,k' \in \C \setminus \Q$,
we can construct a vertex algebra as an 
extension of the affine vertex algebra $L_{\g,k}(0)\otimes L_{\g,k'}(0)$ (see for more detail \cite{HKL}).

Let $(V,m,\eta)$ be a commutative algebra object
in $D(\fg,k) \otimes D(\fg,k')$ such that the balance $\theta$ trivially acts on $V$,
$\theta_V=\mathrm{id}$.
Then, $V = \bigoplus_{\la,\la' \in \lP^+} (L_{\g,k}(\la) \otimes L_{\g,k'}(\la'))^{n_{\la,\la'}}$
where $n_{\la,\la'} \in \Z_{\geq 0}$ is the multiplicity.
By \eqref{eq_inter} and \eqref{eq_balance},
$n_{\la,\la'}=0$ unless
$$
\frac{(\la+2\rho,\la)}{2(k+h^\vee)} + \frac{(\la'+2\rho,\la')}{2(k'+h^\vee)} \in \Z,
$$
that is, $V$ is $\Z$-graded.
By the definition of $D(\g,k)$, the multiplication $m:V\otimes V \rightarrow V$
corresponds to an intertwining operator $I_m(-,z): V \rightarrow \mathrm{End}\,(V,V)[[z^{\pm}]]$.
For any $a_1,a_2,a_3 \in V$ and $u \in V^\vee=\bigoplus_{n \in \Z}V_n^*$
and any path $\ga$ from $Y_2^>$ to $Y_2^<$,
\begin{align*}
A(\ga)u(I_m(a_1,z_1)I_m(a_2,z_2)a_3)
&=A(\ga)C(u,a_1,a_2,a_3;z_1,z_2)(m(\id\otimes m))\\
&=C(u,a_2,a_1,a_3;z_2,z_1)(M(\ga)m(\id \otimes m))\\
&=C(u,a_2,a_1,a_3;z_2,z_1)(m(\id \otimes m))\\
&=u(I_m(a_2,z_2)I_m(a_1,z_1)a_3).
\end{align*}
Here, we used the assumption that $V$ is a commutative associative algebra object.
%, i.e., $M(\ga)m(\id\otimes m)=m(\id \otimes m)$ for any $\ga$.
Hence, $I_m(-,z)$ satisfies the locality condition,
which implies that $V$ is a vertex algebra.
The vacuum vector $\1$ is defined by the non-zero morphism
$\eta:L_{\g,k}(0)\otimes L_{\g,k'}(0) \rightarrow V$,
i.e., $\1 = \eta(\1_{L_{\g,k}(0)}\otimes \1_{L_{\g,k'}(0)})$,
where $\1_{L_{\g,k}(0)}$ (resp. $\1_{L_{\g,k'}(0)}$)
is the vacuum vector of  $L_{\g,k}(0)$ (resp. $L_{\g,k'}(0)$).
Hence, we have (We omit the details here, but refer instead to \cite{HKL}):
\begin{prop}
\label{vertex}
Let $k,k' \in \C \setminus \Q$
and $(V,m,\eta)$ be a commutative algebra object
in $D(\fg,k) \otimes D(\fg,k')$ such that $\theta_V=\mathrm{id}$.
Then, $(V,I_m(-,z),\1)$ is a $\Z$-graded vertex algebra.
\end{prop}
%
%It is also possible to 
%show that from a commutative algebra object in Drinfeld categories $D(\g,k)\otimes \overline{D(\g,k')}$ with
%$k,k' \in \C \setminus \Q$,
%we can construct a non-chiral vertex algebra as an 
%extension of the holomorphic and antiholomorphic affine vertex algebra $V_{\g,k}(0)\otimes  \overline{V_{\g,k'}(0)}$. We will explain it in the following section.
Similarly, let $(V=V_0\oplus V_1,\{m_{i,j}\}_{i,j=01},\eta)$ be a supercommutative algebra object
in $D(\fg,k) \otimes D(\fg,k')$ such that the balance $\theta|_{V_i}=(-1)^i$ for $i=0,1$.
By the definition of $D(\g,k)$, the multiplication $m:V\otimes V \rightarrow V$
corresponds to an intertwining operator $I_m(-,z): V \rightarrow \mathrm{End}\,(V,V)[[z^{\pm \frac{1}{2}}]]$.
Since $m$ does not have any component which is included in 
$\Hom(V_i\otimes V_j,V_{1+i+j})$,
the vertex operator $I_m(-,z) \in \mathrm{End}\,(V,V)[[z^{\pm}]]$.
The rest of the argument is completely similar to the above (for more detail, see \cite{CKM}).
Hence, we have:
\begin{prop}
\label{super_vertex}
Let $k,k' \in \C \setminus \Q$
and $(V=V_0\oplus V_1,\{m_{i,j}\}_{i,j=01},\eta)$ be a supercommutative algebra object
in $D(\fg,k) \otimes D(\fg,k')$ such that $\theta_{V_i}=(-1)^i$.
Then, $(V,I_m(-,z),\1)$ is a vertex superalgebra.
\end{prop}

\subsection{Non-chiral case}
\label{sec_affine2}
In the last section we introduce a complex conjugate of the Drinfeld category, $\overline{D(\g,k)}$
which is defined by
the intertwining operator $I(-,\z)$ with an antiholomophic formal variable.
Thus, a commutative algebra object in the Drinfeld categories $D(\g,k)\otimes \overline{D(\g,k')}$
should corresponds to a non-chiral vertex algebra 
whose vertex operator $Y(-,\uz)$ consisting of holomorphic and antiholomorphic formal variables,
$$
Y(a,\uz)=\sum_{r,s \in \C} a(r,s)z^{-r-1}\z^{-s-1}.
$$
Such an algebra is introduced in \cite{M1} and called a full vertex algebra.
In this section, we briefly recall the definition of a full vertex algebra
and show a result similar to Proposition \ref{vertex} for the FULL case.

%Recall that in section \ref{sec_notation},
%we introduce a space of formal variables $$

For a $\C^2$-graded vector space $F=\bigoplus_{h,\h \in \C^2} F_{h,\h}$, 
set $F^\vee=\bigoplus_{h,\h \in \C^2} F_{h,\h}^*$,
where $F_{h,\h}^*$ is the dual vector space of $F_{h,\h}$.
{\it A full vertex algebra} is a $\C^2$-graded $\C$-vector space
$F=\bigoplus_{h,\h \in \C^2} F_{h,\h}$ equipped with a linear map 
$$Y(-,\uz):F \rightarrow \End (F)[[z^\pm,\z^\pm,|z|^\C]],\; a\mapsto Y(a,\uz)=\sum_{r,s \in \C}a(r,s)z^{-r-1}\z^{-s-1}$$
and a non-zero element $\va \in F_{0,0}$ %and a $\R^2$-graded subspace $F^\vee=\bigoplus_{r,s} F_{r,s}^\vee$ of the dual vector space $Hom_\C(F,\C)$ 
satisfying the following conditions:
%
% F^\vee is not nesesary?
\begin{enumerate}
\item[FV1)]
%There exists $N \in \R$ such that $F_{h,\h}=0$ for any $h \geq N$ or $\h \geq N$.
For any $a,b \in F$, $Y(a,\uz)b \in F((z,\z,|z|^\C))$;
\item[FV2)]
$F_{h,\h}=0$ unless $h-\h \in \Z$;
%For any $a \in F$ and $u \in F^\vee$, there exists $N \in \R$ such that
%$u(a(-r,-s)-)=0$ for any $r \geq N$ or $s \geq N$.
\item[FV3)]
For any $a \in F$, $Y(a,\uz)\va \in F[[z,\z]]$ and $\lim_{z \to 0}Y(a,\uz)\va = a(-1,-1)\va=a$;
\item[FV4)]
$Y(\va,\uz)=\mathrm{id} \in \End F$;
\item[FV5)]
%convergence
For any $a,b,c \in F$ and $u \in F^\vee$,
the formal power series
$u\left(Y(a,\uz_1)Y(b,\uz_2)c\right)$, $u\left(Y(b,\uz_2)Y(a,\uz_1)c\right)$ and 
$u\left(Y(Y(a,\uz_0)b,\uz_2)c\right)$
are absolutely convergent in the region $\{|z_1|>|z_2|\} =Y_2^>$), $\{|z_2|>|z_1|\} =Y_2^<$ and 
$\{|z_2|>|z_0|\}$ respectively
and have analytic continuations to the same single-valued real analytic function on $Y_2$
by taking $z_0=z_1-z_2$ (See \cite{M1} for more detailed definition).
%.
%\item[FV5.5)]
%For any $a,b,c \in F$ and $u \in F^\vee$,
%there exists a single-valued real analytic function $\mu(z_1,z_2)$ on $Y_2$ such that:
%
%
%and has an analytic continuation to the single-valued real analytic function on $Y_2$.
%% $\mu(z_1,z_2)$
%
% there exists $\mu(z_1,z_2) \in \GCor_2$ such that
%\begin{align*}
%\langle u, Y(a,\uz_1)Y(b,\uz_2)c \rangle &= \mu(z_1,z_2)|_{|z_1|>|z_2|}, \\
%\langle u, Y(Y(a,\uz_0)b,\uz_2)c \rangle &= \mu(z_0+z_2,z_2)|_{|z_2|>|z_0|},\\
%\langle u, Y(b,\uz_2)Y(a,\uz_1)c \rangle&=\mu(z_1,z_2)|_{|z_2|>|z_1|};
%\end{align*}
\item[FV6)]
$F_{h,\h}(r,s)F_{h',\h'} \subset F_{h+h'-r-1,\h+\h'-s-1}$ for any $r,s,h,h',\h,\h' \in \C$.
%\item[FV6)]
%$F^\vee=\bigoplus_{n,m \in \R^2} Hom_\C (F_{n,m},\C) \cap F^\vee$
%finiteness of fusion
%\item[FV8)]
%For $n,m,n',m' \in \R$, there exists $l \in \Z_{>0}$ and $n_1,m_1,\dots,n_l,m_l \in \R$ such that 
%the subspace spanned by $\{a(r,s)b\}$
%for all $a \in F_{n,m}, b\in F_{n',m'}, r,s\in \R$ is contained in $\bigoplus_{i=1}^l F_{n_i,m_i}.$
\end{enumerate}
\begin{rem}
In the original definition in \cite{M1},
a full vertex algebra $F$ is assumed to be $\R^2$-graded
and $Y(a,\uz) \in \End F[[z,\z,|z|^\R]]$
since any conformal field theory satisfies this condition.
\end{rem}

Let $(F,Y,\1)$ be a full vertex algebra.
Set $\bar{F}=F$ and
$\bar{F}_{h,\h}=F_{\h,h}$ for $h,\h\in \C$.
Define $\bar{Y}(-,\uz):\bar{F} \rightarrow \End (\bar{F})[[z,\z,|z|^\C]]$
by $\bar{Y}(a,\uz)=\sum_{r,s \in \C}a(r,s)z^{-s-1}\z^{-r-1}$.
Then, $(\bar{F},\bar{Y},\1)$ is a full vertex algebra.
We call it a conjugate full vertex algebra of $(F,Y,\1)$.

We will consider a full vertex algebra $L_{\g,k}(0)\otimes \overline{L_{\g,k'}(0)}$
which is the tensor product of the vertex algebra $L_{\g,k}(0)$
and the conjugate (full) vertex algebra $\overline{L_{\g,k'}(0)}$.

Let $(F,m,\eta)$ be a commutative algebra object
in $D(\g,k)\otimes \overline{D(\g,k')}$ such that $\theta_F=\mathrm{id}$.
Then, $F = \bigoplus_{\la,\la' \in \lP^+} (L_{\g,k}(\la) \otimes \overline{L_{\g,k'}(\la')})^{n_{\la,\la'}}$
where $n_{\la,\la'} \in \Z_{\geq 0}$ is the multiplicity.
Similarly to the last section, by \eqref{eq_inter} and \eqref{eq_balance2},
$n_{\la,\la'}=0$ unless
$$
\frac{(\la+2\rho,\la)}{2(k+h^\vee)} - \frac{(\la'+2\rho,\la')}{2(k'+h^\vee)} \in \Z,
$$
that is, $F$ satisfies (FV2).
By the definition of the Drinfeld category, the multiplication $m:F\otimes F \rightarrow F$
corresponds to an intertwining operator $I_m(-,\uz): F \rightarrow \mathrm{End}\,(F,F)[[z^{\C},\z^{\C}]]$.
(FV6) follows from the definition of the intertwining operators
and (FV2) and (FV6) imply 
$I_m(a,\uz) \in \mathrm{End}\,(F,F)[[z,\z,|z|^{\C}]]$ for any $a \in F$.
Thus, (FV1), (FV3) and (FV4) follows from the definition of the intertwining operators
and the assumption (CA1).
(FV5) follows from the same argument in the last section.
Hence, we have:
\begin{prop}
\label{full_vertex}
Let $k,k' \in \C \setminus \Q$
and $(F,m,\eta)$ be a commutative algebra object
in $D(\fg,k) \otimes \overline{D(\fg,k')}$ such that $\theta_F=\mathrm{id}$.
Then, $(F,I_m(-,\uz),\1)$ is a full vertex algebra.
\end{prop}

\section{Quantum group and isomorphisms}
\label{sec_quantum_group}
As shown in the previous section, in order to construct extensions of affine vertex algebras we would like to construct 
a commutative algebra object in $D(\g,k)\otimes D(\g,k')$.

In Section \ref{sec_quantum_def} and \ref{sec_bilinear}, we review the equivalence between the Drinfeld category
$D(\g,k)$ and a module category of a quantum group $U_q(\g)$,
and prepare the settings and results needed later.
% for the proof of the main theorem.
Then, we will prove (twisted) equivalences among
the module categories at different levels by using isomorphisms among quantum groups and the explicit description of the R-matrix
in Section \ref{sec_isomorphism}
and by the characterization of the braided tensor category 
in Section \ref{sec_twist}.

\subsection{Definition of quantum group}
\label{sec_quantum_def}
In this section, we recall the definition of a quantum group based on \cite{Ja,KS}.

Fix an element $q \in \C \setminus \mathbb{Q}$.
Recall the notations for the simple Lie algebra $\g$ (see Section \ref{sec_notation}),
in particular, the normalized invariant bilinear form $\lla-,-\rra$
which satisfies $\lla \al,\al\rra =2$ for short roots $\al$.
	Put $q_\al := q^{\lla \alpha,\alpha \rra/2}$ for $\al \in \Pi$,
	\[n_q := \frac{q^n - q^{-n}}{q - q^{-1}},\]
	\[n_q ! := n_q (n-1)_q \dots 1_q,\]
	\[\left( \begin{matrix} n \\ m \end{matrix} \right)_q := \frac{n_q!}{m_q! (n-m)_q!}.\]

	\begin{dfn}
The quantized enveloping algebra $U_q(\fg)$ is the algebra defined by generators $\{E_\alpha, F_\alpha, K_{\al}, K_{-\al} \mid \alpha \in \Pi\}$ and relations
	\[K_0 = 1, \quad K_\al K_\be = K_{\al+\be},\]
	\[K_\al E_\be K_{-\al} = q^{\lla\alpha,\be\rra} E_\be,\quad K_\al F_\be K_{-\al} = q^{-\lla\alpha,\be\rra} F_\be,\]
	\[[E_\alpha,F_\beta] = \delta_{\alpha,\beta} \frac{K_\alpha - K_{\alpha}^{-1}}{q_\alpha - q_\alpha^{-1}},\]
	\[\sum_{r=0}^{1-a_{\al\be}} (-1)^r \left( \begin{matrix} 1-a_{\al\be} \\ r \end{matrix} \right)_{q_\al} E_\al^r E_\be E_\al^{1-a_{\al\be}-r} = 0,\]
	\[\sum_{r=0}^{1-a_{\al\be}} (-1)^r \left( \begin{matrix} 1-a_{\al\be} \\ r \end{matrix} \right)_{q_\al} F_\al^r F_\be F_\al^{1-a_{\al\be}-r} = 0,\]
where $a_{\al\be}=\frac{2(\al,\be)}{(\al,\al)}=\frac{2\lla\al,\be\rra}{\lla\al,\al\rra}$, the Cartan matrix, for any $\al,\be \in \Pi$.
	One may define a Hopf algebra structure on $U_q(\fg)$ by
	\[\dD(K_\al) = K_\al \otimes K_\al, \quad \vep(K_\al) = 1, \quad \dS(K_\al) = K_{\al}^{-1},\]
	\[\dD(E_\alpha) = E_\alpha \otimes 1 + K_\alpha \otimes E_\alpha,\quad \vep(E_\alpha) = 0,\quad \dS(E_\alpha) = -K_{\alpha}^{-1} E_\alpha,\]
	\[\dD(F_\alpha) = F_\alpha \otimes K_{\alpha}^{-1} + 1 \otimes F_\alpha,\quad \vep(F_\alpha)=0,\quad \dS(F_\alpha) = -F_\alpha K_\alpha.\]
	\end{dfn}

The algebra $U_q(\g)$ is graded by the root lattice $\lQ=\sum_{\al \in \Pi} \Z \al$:
The grading is given by
$\mathrm{deg}\, E_\al =\al, \mathrm{deg}\,F_\al =-\al$ and $\mathrm{deg}\,K_\al =0$.
	Let $U_q^+(\g)$ (resp.\ $U_q^0(\g)$, $U_q^-(\g)$) be the subalgebra of $U_q(\g)$ generated by
all $E_\alpha$'s (resp.\ $K_\al$'s, $F_\alpha$'s) with $\al \in \Pi$.
Then, these subalgebras are graded subalgebras of $U_q(\g)=\bigoplus_{\mu \in \lQ} U_q(\g)_{\mu}$.
Since $U_q^0(\g)$ is a commutative algebra,
we can define for each $\la\in \lQ$ an element $K_\la$ in $U_q^0(\g)$ by
$$
K_\la=\Pi_{\be \in \Pi} K_\be^{m_\be} \text{ if }\la=\sum_{\be \in \Pi}m_\be \be.
$$

%Then the multiplication
%	\[U_q(\fn^-) \otimes U_q(\fh) \otimes U_q(\fn^+) \to U_q(\fg)\]
%	is an isomorphism as vector spaces.

	For each $\lambda \in \lP$, let $L_q(\lambda)$ be the unique irreducible highest module of highest weight $\lambda$, that is, there exists nonzero $v_\lambda \in L_q(\lambda)$ such that
	\[K_\mu v_\lambda = q^{\lla \mu,\lambda\rra} v_\lambda,\;\;\;\; E_\alpha v_\lambda = 0.\]
%	We also denote the corresponding representation by $\pi^\lambda$.
	If $\lambda \in \lP^+$, then $L_q(\lambda)$ is finite dimensional. We say a $U_q(\fg)$-module is of {\it type 1} if it decomposes into a direct sum of $L_q(\lambda)$'s for $\lambda \in \lP^+$.
%Denote by $U_q(\g)\modu$ the category of type 1 $U_q(\g)$ modules.  
Notice that any subquotient of type 1 module is also of type 1. Hence, $U_q(\g)\modu$ is a locally finite semisimple abelian category. We remark that the abelian category structure does not depend on $q \in \C\setminus \Q$, however the tensor category structure does.

For any type 1 module $M$,
set for all $\la \in \lP$
$$
M_{\la} = \{m\in M \mid K_\al m= q^{\lla \la,\al\rra}m \text{ for all } \al \in \Pi \}.
$$
Then, we have 
$$M= \bigoplus_{\la \in \lP} M_{\la}.$$

The following Hopf algebra isomorphisms are important in this paper:
%see for example \cite[Section 6.1.4.]{KS}:
% is very important in this paper:
\begin{lem}
\label{lem_isom}
For any simple Lie algebra $\g$,
\begin{enumerate}
\item
There exist Hopf algebra isomorphisms
$\om: U_q(\g) \rightarrow U_q(\g)^\cop$,
$\theta:U_q(\g) \rightarrow U_{q^{-1}}(\g)_\op$ and
$\psi:U_q(\g) \rightarrow U_{q^{-1}}(\g)^\cop$
%and $\psi:U_q(\g) \rightarrow U_{-q}(\g)$
such that: For any $\al \in \Pi$
\begin{align*}
\om(E_\al)=F_{\al},\;\;\;\; \om(F_\al)&=E_\al,\;\;\;\; \om(K_\al)=K_\al^{-1},\\
\theta(E_\al)=E_{\al},\;\;\;\; \theta(F_\al)&=F_\al,\;\;\;\; \theta(K_\al)=K_\al,\\
\psi(E_\al)=-K_\al^{-1}E_\al,\;\;\;\; \psi(F_\al)&=-F_{\al}K_\al,\;\;\;\; \psi(K_\al)=K_\al^{-1}.
%\phi(E_\al)=E_{\al},\;\;\;\; \psi(F_\al)&=(-1)^{d_\al}F_\al,\;\;\;\; \psi(K_\al)=K_\al,
\end{align*}
\item
There exists an algebra isomorphism
$\tau:U_q(\g) \rightarrow U_q(\g)_\op$,
$$
\tau(E_\al)=E_{\al},\;\;\;\; \tau(F_\al)=F_\al,\;\;\;\; \tau(K_\al)=K_\al^{-1}.
$$
\end{enumerate}
\end{lem}
\begin{proof}
It is easy to show that $\theta$ and $\omega$ are Hopf algebra isomorphisms.
Since $\psi = S \circ \theta$ and
$U_q(\g) \stackrel{\theta}{\rightarrow} U_{q^{-1}}(\g)_\op
\stackrel{S}{\rightarrow} U_{q^{-1}}(\g)^\cop$,
$\psi$ gives a Hopf algebra isomorphism.
%We will show that $\psi$ is an algebra isomorphism.
%Since
%\begin{align*}
%[\psi(E_\alpha),\psi(F_\beta)]&=
%[F_\al K_\al, K_\beta^{-1} E_\beta]\\
%&=F_\al K_\al K_\beta^{-1} E_\beta - K_\beta^{-1} E_\beta F_\al K_\al\\
%&= K_\beta^{-1} [F_\al, E_\beta] K_\al\\
%&= - K_\beta^{-1} \delta_{\al,\be} \frac{K_\al-K_{-\al}}{q_\al^{-1}-q_\al} K_\al =\psi([E_\al,F_\be])
%\end{align*}
%and
%\begin{align*}
%\psi(\sum_{r=0}^{1-(\alpha,\beta^\vee)} &(-1)^r \left( \begin{matrix} 1-(\alpha,\beta^\vee) \\ r \end{matrix} \right)_{q_\beta} E_\beta^r E_\alpha E_\beta^{1-(\alpha,\beta^\vee)-r})\\
%&=
%\sum_{r=0}^{1-(\alpha,\beta^\vee)} (-1)^r \left( \begin{matrix} 1-(\alpha,\beta^\vee) \\ r \end{matrix} \right)_{q_\beta} (F_\be K_\be)^r F_\al E_\al
%E_\beta^r E_\alpha E_\beta^{1-(\alpha,\beta^\vee)-r})\\
%\end{align*}
%
\end{proof}

\subsection{Bilinear form and R-matrix}
\label{sec_bilinear}
In this section, we recall the definitions of a bilinear form on $U_q(\g)$ and
the R-matrix. Here, we sometimes simplify our notations and write only
$U=U_q(\g)$, $U^+=U_q(\g)^+$ and $U^- = U_q(\g)^-$.
This section is mainly based on \cite{Ja} (see also \cite{Lu1,KS}).
The subalgebra of $U$ generated by $U^+$ and $U^0$ (resp. $U^-$ and $U^0$) is denoted by
$U^{\geq 0}$ (resp. $U^{\leq 0}$).
Recall that $U$ is a $\lQ$-graded algebra and
for all $\mu \in \lQ$,
$$
U_\mu=\{u \in U \mid K_\al u K_{\al}^{-1}=q^{\lla \mu,\al \rra}u \fora \al \in \Pi\}.
$$
Set $U_\mu^{+} = U^+ \cap U_{\mu}$
and $U_\mu^{-} = U^+ \cap U_{- \mu}$
for $\mu \in \lQ$ with $\mu \geq 0$.

\begin{prop}{\cite[Proposition 6.12 and Corollary 8.30]{Ja}}
\label{bilinear}
There exists a unique bilinear pairing $(-,-):U^{\leq 0} \times U^{\geq 0} \rightarrow \C$
such that for all $x,x' \in U^{\geq 0}$,
all $y,y' \in U^{\leq 0}$,
all $\mu,\nu \in \lQ$, and all $\al,\be \in \Pi$
\begin{align*}
(y,xx')=(\Delta y,x'\otimes x),\;\;\;\;\; &(yy',x)=(y\otimes y',\Delta x),\\
(K_\mu,K_\nu)=q^{-\lla\mu,\nu\rra},\;\;\;\;\; &(F_\al,E_\be) = - \delta_{\al,\be}(q_\al-q_\al^{-1})^{-1}\\
(K_\mu,E_\al)=0  \;\;\;\;\; &(F_\al, K_\mu) =0.
\end{align*}
Furthermore, the restriction of $(-,-)$ to any $U_\mu^- \times U_\mu^+$ with $\mu \in Q$, $\mu \geq 0$
is a non-degenerate paring.
\end{prop}

The bilinear form is preserved by the isomorphisms $\om:U_q(\g)\rightarrow U_q(\g)^\cop$: 
%and $S:U_q(\g)\rightarrow U_q(\g)_\op^\cop$:
\begin{lem}{\cite[Lemma 6.16.]{Ja}}
\label{invariance}
%\begin{enumerate}
%\item
For all $x \in U^+$ and $y \in U^-$,
$
(\om (x),\om(y))=(y,x).
$
%\item
%For all $x \in U^{\geq 0}$ and $y \in U^{\leq 0}$,
%$
%(S(y),S(x))=(y,x).
%$
%\end{enumerate}
\end{lem}
We will construct the R-matrix according to \cite{Ja}.
Choose for each $\mu \in Q$, $\mu  \geq 0$ a basis $u_1^\mu,u_2^\mu,\dots,u_r(\mu)^\mu$
of $U_\mu^+$. By Proposition \ref{bilinear} we can find a basis 
$v_1^\mu,v_2^\mu,\dots v_{r(\mu)}^\mu$ of $U_{-\mu}^-$ such that 
$(v_j^\mu,u_i^\mu)=\delta_{i,j}$ for all $i$ and $j$.
Set
$$
\Theta_\mu = \sum_{i=1}^{r(\mu)}v_i^{\mu} \otimes u_i(\mu) \in U_{\mu}^-\otimes U_{\mu}^+.
$$
By linear algebra, $\Theta_\mu$ does not depend on the choice of the basis $(u_i^\mu)_i$.
Thus, by Lemma \ref{invariance}, we have:
\begin{lem}[\cite{Ja}]
\label{omega_invariance}
For all $\mu \in Q$, $\mu \geq 0$,
$$(\om \otimes \om) \Theta_\mu =P_{21}(\Theta_\mu),$$
where $P_{21}$ is the transposition.
\end{lem}

Let ${U\hat{\otimes} U}$
 be the completion of the vector space $U\otimes U$
with respect to the descending sequence of vector spaces
$$
(U^+U^0\sum_{\wt\, \mu \geq N} U_\mu^-)\otimes U
+ U\otimes U^- U^0 \sum_{\wt\, \mu \geq N} U_\mu^+)
$$
for $N=1,2,\dots$
and similarly
$U\hat{\otimes_{\rev}}U$
the completion with respect to
$$
(U^-U^0\sum_{\wt\, \mu \geq N} U_\mu^+)\otimes U
+ U\otimes U^+ U^0 \sum_{\wt\, \mu \geq N} U_\mu^-)
$$
for $N=1,2,\dots$.
Then, the algebra structure on $U\otimes U$ extends to algebra structures on
$U\hat{\otimes}U$ and $U\hat{\otimes_\rev}U$,
and $\Theta = \sum_{\mu \geq 0}\Theta_\mu$ is in $U\hat{\otimes}U$
and $P_{21} \Theta$ is in $U\hat{\otimes_\rev}U$.

Set 
$\Delta^\tau = (\tau \otimes \tau) \circ \Delta \circ \tau^{-1},$ c.f. Section \ref{sec_quantum_def}.
Then, we have:
\begin{prop}[\cite{Ja}]
\label{theta_property}
The element $\Theta \in U\hat{\otimes}U$ satisfies
$$
\Delta(u)\circ \Theta=\Theta \circ \Delta^\tau (u)
$$
for any $u \in U$.
\end{prop}
Furthermore, $\Theta$ is unique in the following sense:
\begin{thm}
\label{thm_unique}
\cite[Theorem 4.1.2.]{Lu1}
Let $\Gamma_{\mu} \in U_{\mu}^-\otimes U_{\mu}^+$ be a family of
elements (with $\mu \geq 0$) such that:
\begin{enumerate}
\item
$\Gamma_0=1 \otimes 1$;
\item
$\Gamma = \sum_{\mu \geq 0}\Gamma_\mu$ satisfies
$\Delta(u)\Gamma=\Gamma \Delta^\tau(u)$ for all $u\in U$ (identify in $U\hat{\otimes}U$).
\end{enumerate}
Then, $\Gamma_\mu=\Theta_\mu$ for all $\mu \geq 0$.
\end{thm}

Fix a complex number $\rho$ with 
$q=\exp(\pi i \rho).$
Note that if $\rho$ and $\rho'$ satisfy $q=\exp(\pi i \rho)=\exp(\pi i \rho')$,
then $\rho-\rho' \in 2\Z$.
Define for all type 1 $U$-modules $M$ and $N$
a bijective linear map $f_\rho: M\otimes N \rightarrow M\otimes N$ by
$$
f_\rho (m\otimes n)=\exp\left(- \pi i \rho \lla \la,\mu\rra\right) m\otimes n \fora m\in M_\la \text{ and }n\in N_{\mu}
$$
and for all $\mu, \la \in \lP$.
By definition, we have:
\begin{lem}
\label{lem_difference}
Let $n_\g$ be a minimal positive integer such that
$n_\g \lla \la,\mu\rra \in \Z$ for any $\la,\mu \in \lP$.
Then, $f_\rho = f_{\rho+2n_gN}$ for any $N \in \Z$ (as a linear map).
\end{lem}

Let $M$ and $N$ be type 1 $U$-modules.
Since for any $m \in M$ and $n\in N$, $\Theta_{\mu} (m\otimes n)=0$ for almost all $\mu \geq 0$,
we can define a linear map
$$
\Theta=\Theta_{M,N}:M\otimes N \rightarrow M\otimes N,\;\;\;\;\;\; \Theta =\sum_{\mu \geq 0} \Theta_{\mu}.
$$
Since $\Theta$ is unipotent, and thus, invertible,
we can define a linear map
$$R(\rho)=(\Theta \circ f_\rho)^{-1}:M\otimes N \rightarrow M\otimes N
$$
for all type 1 $U$-modules $M$ and $N$.
Then, as an operator acting on the tensor product of type 1 $U$-modules,
$R(\rho)$ satisfies
\begin{enumerate}
\item[R1)]
$R(\rho) \Delta(x) R(\rho)^{-1} = \Delta^\op(x)$ for all $x \in U$;
\item[R2)]
$(\Delta\otimes 1)(R(\rho)) = R_{13}(\rho)R_{23}(\rho)$;
\item[R3)]
$(1\otimes \Delta)(R(\rho)) = R_{13}(\rho)R_{12}(\rho)$.
\end{enumerate}
Thus, in the same way as for quasi-triangular Hopf algebra, 
we can define a braided tensor category structure on the category of type 1 $U$-modules.
We denote it by $(U_q(\g),R(\rho))\modu$.
%monodromic systems

Set $r^\vee=\frac{\lla \al,\al \rra}{2}$ for a long root $\al$,
which is the ratio of the norm of long roots and short roots.
Note that $r^\vee=1$ when the Lie algebra is simply-laced.
The result is due to Drinfeld, Kazhdan and Lusztig \cite{Dr1,Dr2,KL,Lu2} (see also \cite{BK}):
\begin{thm}[Drinfeld, Kazhdan-Lusztig]
\label{thm_DK}
Let $k \in \C \setminus \Q$ and $\rho(k)=\frac{1}{r^\vee(k+h^\vee)} \in \C$.
Then, there exists an equivalence of balanced braided tensor categories between $D(\g,k)$
and $(U_{\exp(\pi i \rho(k))}(\g),R(\rho(k))\modu$.
\end{thm}

We note that
$R_{21}(\rho)^{-1}= (P_{21}\Theta)\circ f_\rho$ 
also satisfies (R1)-(R3). 
Denote by $(U_q(\g),R_{21}(\rho)^{-1})\modu$
 the braided tensor category defined by $R_{21}(\rho)^{-1}$.
Then, similarly to Lemma \ref{lem_rev}, by the above theorem, we have:
\begin{cor}
There exists an equivalence of balanced braided tensor categories between $D(\g,k)^\rev$
and $(U_q(\g),R(\rho(k))_{21}^{-1})\modu$.
\end{cor}

It is noteworthy that the braided tensor category structure depends
on the choice of $\rho \in \C$.
However,  by Lemma \ref{lem_difference}, we have:
\begin{cor}
\label{cor_difference}
There exists an equivalence of balanced braided tensor categories between 
$(U_q(\g),R(\rho))\modu$ and
$(U_q(\g),R(\rho+ 2n_\g N))\modu$
for any $N \in \Z$.
\end{cor}
A comparison between $(U_q(\g),R(\rho))\modu$ and
$(U_q(\g),R(\rho+ N))\modu$, in a more general situation, will be made in Section \ref{sec_twist}.

\subsection{Isomorphism between $q$ and $q^{-1}$}
\label{sec_isomorphism}
Recall that our first aim is to construct a commutative algebra object
in $D(\g,k) \otimes D(\g,k)^\rev$ (or $D(\g,k) \otimes D(\g,k')$) (see Proposition \ref{full_vertex} and \ref{vertex}).
By Theorem \ref{thm_DK}, it suffices to consider it in $(U_q(\g),R(\rho))\modu$
and $(U_q(\g),R_{21}(\rho^{-1}))\modu$.
In this section, we will show that
three quasi-triangular Hopf algebras 
$(U_q(\g),R(\rho)_{21}^{-1})$, $(U_q(\g)^\cop,R(\rho)^{-1})$
and $(U_{q^{-1}}(\g),R(-\rho))$ are
isomorphic and thus their module categories
are equivalent as balanced braided tensor categories.

We first consider the equivalence between 
$(U_q(\g),R(\rho)_{21}^{-1})\modu$ and $(U_q(\g)^\cop,R(\rho)^{-1})\modu$.
Recall that by Lemma \ref{lem_isom} 
there is a Hopf algebra isomorphism $\om:U_q(\g)\rightarrow U_q(\g)^\cop$.
We will show that this isomorphism induces an equivalence of balanced braided tensor categories.

Let $M$ be a type 1 $U_q(\g)^\cop$-module
and $\om^*M$ a type 1 $U_q(\g)$-module defined by
$$
a \cdot_\om m = \om(a) \cdot m \;\;\;\;\;\;\;\;\;\fora a \in U_q(\g) \text{ and }m\in M.
$$

Denote by $\om^*M\otimes \om^*N$ (resp. $M \otimes N$) the tensor product of $M$ and $N$ as $U_q(\g)$-modules (resp. $U_q(\g)^\cop$-modules).
Since $\om^*M\otimes \om^*N$ and $M \otimes N$
have the same underlying vector space,
we can compare the R-matrices acting on them.

Since $K_\al \cdot_\om v = q^{-\lla\al,\la\rra}v$ for $v \in M_{\la}$,
$v \in (\om^*M)_{-\la}$
and thus the action of $f(\rho)$ on $\om^*M\otimes \om^*N$ and $M \otimes N$ are the same.

Thus, it suffices to compare
$\Theta$ with $(\omega\otimes \om)(P_{21}\Theta)$.
They are the same by Lemma \ref{omega_invariance}.
Thus, we have:
\begin{prop}
\label{op_isomorphism}
The Hopf algebra isomorphism $\om:U_q(\g)\rightarrow U_q(\g)^\cop$
induces an equivalence between $(U_q(\g),R(\rho)_{21}^{-1})\modu$ and $(U_q(\g)^\cop,R(\rho)^{-1})\modu$
as balanced braided tensor categories.
\end{prop}

We will next consider the equivalence between 
$(U_q(\g),R(\rho))\modu$ and \\
$(U_{q^{-1}}(\g)^\cop,R(-\rho)^{-1})\modu$.
Recall that by Lemma \ref{lem_isom} 
there is a Hopf algebra isomorphism $\psi: U_q(\g)\rightarrow U_{q^{-1}}(\g)^\cop$.
We will show that this isomorphism induces an equivalence of balanced braided tensor categories.

Let $M$ and $N$ be a type 1 $U_{q^{-1}}(\g)^\cop$-modules.
%and $\om^*M$ a type 1 $U_q(\g)$-module defined by
%$$
%a \cdot_\om m = \om(a) \cdot m \;\;\;\;\;\;\;\;\;\fora a \in U_q(\g) \text{ and }m\in M.
%$$
Denote by $\psi^*M\otimes \psi^*N$ (resp. $M \otimes N$) the tensor product of $M$ and $N$ as $U_q(\g)$-modules (resp. $U_{q^{-1}}(\g)^\cop$-modules).
Since $\psi^*M\otimes \psi^*N$ and $M \otimes N$
have the same underlying vector space,
we can compare the R-matrices again.
We note that $\Theta^q = \Theta$ for $U_q(\g)$ and $\Theta^{q^{-1}}=\Theta$ for $U_{q^{-1}}(\g)$ are different. It suffices to show that
$R(-\rho)^{-1} = \Theta^{q^{-1}} \circ f_{-\rho}$ and $\psi^* R(\rho)= \psi^* (f_\rho^{-1}\circ (\Theta^q)^{-1})$ are the same as an linear maps on $M \otimes N$.
We need the following lemma:
\begin{lem}
\label{f_transform}
Let $u \in U_q(\g)_{\mu}^-$ and $u' \in U_q(\g)_{\mu}^+$ for $\mu \in \lQ$ with $\mu \geq 0$.
\begin{enumerate}
\item
For any type 1 $U_q(\g)$-modules $M$ and $N$,
$$
f_{\rho}^{-1} \circ (u\otimes u') \circ f_{\rho} =uK_{\mu} \otimes K_{-\mu}u'
$$
as linear maps acting on $M \otimes N$.
\item
$\psi(uK_{\mu}) \in U_{q^{-1}}(\g)_{\mu}^-$
and $\psi(K_{-\mu}u') \in U_{q^{-1}}(\g)_{\mu}^+$.
\end{enumerate}
\end{lem}
\begin{proof}
Let $\la,\la' \in \lP$.
Since $(u\otimes u') \cdot M_{\la}\otimes N_{\la'} \subset M_{\la-\mu} \otimes N_{\la'+\mu},$
for any $v \in M_\la$ and $w\in N_{\la'}$
\begin{align*}
f_{\rho}^{-1}\circ (u\otimes u') \circ f_{\rho} (v\otimes w)
=\exp\left(\pi i \rho(\lla \la,\mu \rra-\lla\la',\mu\rra-\lla\mu,\mu\rra)\right) (u\otimes u') (v\otimes w).
\end{align*}
Hence,
$
f_{\rho}^{-1}\circ (u\otimes u') \circ f_{\rho}
= q^{-\lla\mu,\mu\rra} (u\otimes u') (K_{\mu} \otimes K_{-\mu})
=(uK_{\mu} \otimes K_{-\mu}u')$, which implies (1).
(2) follows from the definition of $\psi$.
\end{proof}

Since 
$(\Theta^q)^{-1}$ is an infinite sum of elements $U_q(\g)_{\mu}^- \otimes U_q(\g)_\mu^+$ with $\mu \geq 0$,
by the above lemma,
there exits a family of elements $\Gamma_\mu \in U_{q^{-1}}(\g)_\mu^-\otimes U_{q^{-1}}(\g)_\mu^+$
such that
$$
\Gamma =\sum_{\mu \geq 0}\Gamma_\mu = (\psi\otimes \psi)(f_\rho^{-1}\circ (\Theta^q)^{-1} \circ f_\rho).
$$
Thus, to obtain the equivalence of the categories,
it suffices to show
that
$
\Gamma = \Theta^{q^{-1}}
$ as an element in $U_{q^{-1}}(\g) \hat{\otimes} U_{q^{-1}}(\g)$,
which follows from Theorem \ref{thm_unique}.
Hence, we have:
\begin{prop}
\label{psi_isomorphism}
The Hopf algebra isomorphism $\psi:U_q(\g)\rightarrow U_{q^{-1}}(\g)^\cop$
induces an equivalence between 
$(U_q(\g),R(\rho))\modu$ and $(U_{q^{-1}}(\g)^\cop,R(-\rho)^{-1})\modu$
as balanced braided tensor categories.
\end{prop}

\subsection{Graded twists of Drinfeld categories}
\label{sec_twist}
In this section, we will consider a graded twist of the braided tensor category $(U_q(\g),R(\rho))\modu$ (see Section \ref{sec_twist_intro}).

Let $S$ be a subset of the weight lattice $\lP$ such that $\lQ+S \subset S$ (a subset of the coset $\lD$).
Let $(U_q(\g),R(\rho),S)\modu$ be the full subcategory of $(U_q(\g),R(\rho))\modu$ whose object is isomorphic to a direct sum of $L_q(\mu)$ for $\mu \in S \cap \lP^+$.
Since $L_q(\mu)\otimes L_q(\mu') \in (U_q(\g),R(\rho),\mu+\mu'+\lQ)\modu$ for any $\mu,\mu' \in \lP$,
the subcategories $\{(U_q(\g),R(\rho),\mu+\lQ)\modu)\}_{\mu \in \lD}$
define a $\lD$-grading on $(U_q(\g),R(\rho))\modu$ (see Section \ref{sec_twist_intro}).
We note that if $L$ is a subgroup of $\lP$ with $\lQ \subset L$,
then $(U_q(\g),R(\rho),L)\modu$ is closed under the tensor product and thus
a braided tensor subcategory.

Let $Q_\g:\lP/\lQ\rightarrow \C^\times$ be a map defined by
$$
Q_\g(\la)=\exp(\pi i \lla \la,\la \rra).
$$
Since $\lQ$ is an even lattice with respect to $\lla-,-\rra$,
the map is well-defined.
Since
$$
\frac{Q_\g(\la+\mu)}{Q_\g(\la)Q_\g(\mu)}= \exp(2\pi i \lla \la,\mu \rra)
$$
is a bicharacter on $\lP/\lQ$, $Q_\g$ is a quadratic form.
We note that $Q_\g^N$ is also a quadratic form for any $N\in\Z$.
By Theorem \ref{thm_coho}, we can consider the graded twist of $(U_q(\g),R(\rho))\modu$ 
associated with these quadratic forms.
We first show the following proposition:
\begin{prop}
\label{even_twist}
For any $N \in \Z$, the identify functor gives an equivalence between $(U_q(\g),R(\rho))\modu^{Q_\g^{2N}}$ 
and $(U_q(\g),R(\rho+2N))\modu$ as braided tensor categories.
\end{prop}
\begin{proof}
We note that $(U_q(\g),R(\rho))\modu$ and $(U_q(\g),R(\rho+2N))\modu$ has the same underlying monoidal category and only the difference is the braiding.
Since $R(\rho) = (\Theta \circ f_{\rho})^{-1}$, the difference comes from $f_{\rho}$.

Let $B_{N}:\lP/\lQ\times \lP/\lQ\rightarrow \C^\times$ be a bicharacter defined by
$$
B_N(\la,\mu)=\exp(2\pi i N\lla \la,\mu \rra)
$$
and $1_{\lP/\lQ}:\lP/\lQ\times \lP/\lQ\times  \lP/\lQ\rightarrow \C^\times$ be a trivial 3-cocycle.
Then, $\mathrm{EM}(Q_\g^{2N})^{-1} = (1_{\lD},B_N) \in H_\ab^3(\lD,\C^\times)$ by Lemma \ref{lem_degenerate_quad}.
%Since $Q_\g^{2N}(\la)=B_N(\la,\la)$ for any $\la \in \lP/\lQ$,
%the graded twist by $Q_\g^{2N}$ changes only the braiding and not associative isomorphism by Lemma \ref{lem_degenerate_quad}.
%In fact, $(1_\lP/\lQ,B_N) \in Z_\ab^3(\lP/\lQ,\C^\times)$
%is the abelian cocycle associated with $Q_\g$.
Thus, $(U_q(\g),R(\rho))\modu^{Q_\g^{2N}}$ 
is equivalent to $(U_q(\g), B_N \circ R(\rho))\modu$.
Since
\begin{align*}
B_N\circ f_{\rho}^{-1} |_{L_q(\la)_{\la+\al}\otimes L_q(\mu)_{\mu+\be}}
&= \exp(2\pi i N \lla \la,\mu \rra) 
\exp(\pi i \rho \lla  \la+\al,\mu+\be \rra)\\
&=\exp(\pi i (\rho+2N) \lla  \la+\al,\mu+\be \rra)\\
&= f_{\rho+2N}^{-1},
\end{align*}
the assertion holds.
\end{proof}

Hereafter, we will consider the twisted braided tensor category
$(U_q(\g),R(\rho))\modu^{Q_\g^{N}}$ for odd integers $N \in \Z$.
We note that $\exp(\pi i (\rho+N))=(-1)^N q$.
Hence, if $N$ is an odd number, $q$ will change.
We conjecture that the following statement holds:
\begin{conj}
\label{conj}
Let $\g$ be a simple Lie algebra.
For any $N \in \Z$, $(U_q(\g),R(\rho))\modu^{Q_\g^{N}}$ and
$(U_{(-1)^N q}(\g),R(\rho+N))\modu$ are equivalent as braided tensor categories.
\end{conj}

The conjecture is true if $N$ is even by Proposition \ref{even_twist}.
We will prove this conjecture for the Lie algebra of type ABC
and partially for type D.
In the case of type D, we show the conjecture for some full subcategory of $(U_q(\mathrm{so}_{2n}),R(\rho))\modu$.

More precisely, let $\la_1$ be the fundamental weight of $\mathrm{so}_{2n}$
such that $L_q(\la_1)$ is isomorphic to the vector representation of $\mathrm{so}_{2n}$.
Let $\Lambda_v$ be the subgroup of $\lP$ generated by the root lattice $\lQ$ and $\la_1$.
%$\la_{(n-1)/2}$ (resp. $\la_1$) for type B (resp. type D).
%We note that for type B since $\la_{(n-1)/2} \in \lQ$, $\Lambda_v$ is equal to the root lattice $\lQ$.
Then, the monoidal subcategory of $(U_q(\mathrm{so}_{2n}),R(\rho))\modu$ generated by the vector representation is equal to $(U_q(\mathrm{so}_{2n}),R(\rho),\Lambda_v)\modu$.
We remark that $(U_q(\mathrm{so}_{2n}),R(\rho),\Lambda_v)\modu$ is graded by $\Lambda_v/\lQ=\Z/2\Z$
and the restriction of the quadratic form $Q_{\mathrm{so}_{2n}}:\lP/\lQ \rightarrow \C^\times$
gives a quadratic form on $\Lambda_v/\lQ \cong \Z/2\Z$. We denote it by the same symbol $Q_{\mathrm{so}_{2n}}$.
%,

\begin{thm}
\label{odd_twist}
Conjecture \ref{conj} is true if $\g$ is a simple Lie algebra of type ABC.
In the case of $\g=\mathrm{so}_{2n}$, i.e., of type $D_n$,
$(U_q(\mathrm{so}_{2n}),R(\rho),\Lambda_v)\modu^{Q_\g^{N}}$ and
$(U_{(-1)^N q}(\mathrm{so}_{2n}),R(\rho+N),\Lambda_v)\modu$ are equivalent as braided tensor categories for any $N\in \Z$.
%Here, $(U_q(\mathrm{so}_{2n}),R(\rho),\Lambda_v)\modu$ is a full subcategory of
%$(U_q(\mathrm{so}_{2n}),R(\rho))\modu$ generated by the vector representation of $\mathrm{so}_{2n}$.
\end{thm}

If $\g$ is of type B, that is $\g \cong \so$, there is a Hopf algebra isomorphism $\phi: U_q(\so)\rightarrow U_{-q}(\so)$.
So in this case, the above theorem can be proved as in the previous section.
However, this case is a little more complicated because of the existence of twist, and the proof is given in appendix (Theorem \ref{thm_B}). In this section, we will prove the theorem for the simple Lie algebra of type ACD.

We first consider the case of $\g = \mathrm{sl}_n$, denote $Q_{\mathrm{sl}_n}$ by $Q$ for short.
In \cite{KW}, semisimple rigid monoidal tensor categories with fusions rule of $\mathrm{sl}_n$
is classified.
Such categories $C$ is parametrized by the pairs $(q_C,\tau_C)$ of nonzero complex numbers
defined up to replacing $(q_C,\tau_C)$ by $(q_C^{-1},\tau_C^{-1})$,
such that $q_C^{n(n-1)/2} = \tau_C^n$ and $q_C$ is not a nontrivial root of unity \cite{Jo}.
It is easy to verify that for $(U_q(\mathrm{sl}_n),R(\rho))\modu^{Q^N}$ 
the corresponding parameters are $q_C= q^2$ and $\tau_C =(-1)^{(n-1)N} q^{n-1}$.
Thus, $(U_q(\mathrm{sl}_n),R(\rho))\modu^{Q_{\mathrm{sl}_n}}$ and $(U_{-q}(\mathrm{sl}_n),R(\rho+2k+1))\modu$ are equivalent as a monoidal category for any $k\in \Z$.
We need more work to determine the integer $k \in \Z$ (or the braiding structure).
Let $X$ be a vector representation of $U_q(\mathrm{sl}_n)$.
Then, $X\otimes X$ is a direct sum of two simple objects,
$X^{\otimes 2} \cong X_a \oplus X_s$, where $X_s$ is the symmetric tensor and
$X_a$ is the antisymmetric tensor.
The braiding $B_{X,X} \in \mathrm{End}X^{\otimes 2}$ acts as a scalar on each component.
In the untwisted $(U_q(\mathrm{sl}_n),R(\rho))$ case,
these values are 
\begin{align*}
B_{X,X}|_{X_s} &= \exp(\pi i \rho(1-\frac{1}{n})),\\
B_{X,X}|_{X_a} &= -\exp(\pi i \rho(-1-\frac{1}{n}))
\end{align*}
%$\exp(\pi i \rho(1-\frac{1}{n}))$ on $X_s$ and $-\exp(\pi i \rho(-1-\frac{1}{n}))$ on $X_a$
(see for example \cite[Section 8.4.3]{KS}).

Hereafter, we will use the labeling of the Dynkin diagram in \cite{Hu}
and $\{\la_i\}_{i=1,\dots,n-1}$ are the fundamental weights with respect to this labeling.
Then, the vector representation is isomorphic to the highest weight representation $L_q(\la_1)$.
Since $Q(\la_1) =\exp ( \pi i \lla \la,\la\rra)= \exp (\pi i (1-\frac{1}{n}))$,
for the twisted braided tensor category $(U_q(\mathrm{sl}_n),R(\rho))^Q$,
the values of the twisted braiding $B_{X,X}^Q$
are
\begin{align*}
B_{X,X}^Q|_{X_s} &= \exp (\pi i (1-\frac{1}{n}))\exp(\pi i \rho(1-\frac{1}{n}))=\exp (\pi i (\rho+1)(1-\frac{1}{n})),\\
B_{X,X}^Q|_{X_a} &= - \exp (\pi i (1-\frac{1}{n})) \exp(\pi i \rho(-1-\frac{1}{n}))
=-\exp(\pi i (-(\rho+1)-\frac{\rho+1}{n}+2))\\
&=-\exp(\pi i (\rho+1) (-1-\frac{1}{n})),
\end{align*}
%$ \exp (\pi i (1-\frac{1}{n}))\exp(\pi i \rho(1-\frac{1}{n}))=\exp (\pi i (\rho+1)(1-\frac{1}{n}))$ on $X_s$
% and $- \exp (\pi i (1-\frac{1}{n})) \exp(\pi i \rho(-1-\frac{1}{n}))
%=-\exp(\pi i (-(\rho+1)-\frac{\rho+1}{n}+2))
%=-\exp(\pi i (\rho+1) (-1-\frac{1}{n}))$ on $X_a$.
which are the same with that of $(U_{-q}(\mathrm{sl}_n),R(\rho+1))$.
By \cite[Section 2]{KW}, the braiding is uniquely determined by the 
eigenvalues of $B_{X,X}:X\otimes X \rightarrow X\otimes X$.
Hence, the conjecture is true for $\g=\mathrm{sl}_n$.
%\begin{prop}
%\label{A_twist}
%For any $n \geq 2$ and $N \in \Z$,
%there exists an equivalence between 
%$(U_q(\mathrm{sl}_n),R(\rho))\modu^{Q^N}$ 
%and $(U_{(-1)^Nq}(\mathrm{sl}_n),R(\rho+N))\modu$ as braided tensor categories.
%\end{prop}
\begin{rem}
For odd $n$, this result is proved by using the generator and relation of
the quantum coordinate ring $O_q(\mathrm{SU}(n))$ \cite{PR} or \cite{BY}
and for $n=2$ is shown, for example, in \cite{BNY}.
\end{rem}

There is a partial extension of the result of Kazhdan and Wenzl to type BCD obtained by Tuba and Wenzl \cite{TW}. 
%We will first consider the case of type C.
In \cite[Theorem 9.4]{TW}, semisimple rigid monoidal tensor categories with fusions rule of $\mathrm{sp}_{2n}$
(resp. $\mathrm{so}_{n}$) are classified.
By using their characterization, for any $n \in \Z_{>}$
there is an equivalence of monoidal categories,
\begin{align}
(U_q(\mathrm{sp}_n),R(\rho))\modu^{Q_{\mathrm{sp}_{2n}}} &\cong (U_{(-1)q}(\mathrm{sp}_{2n}),R(\rho+1))\modu \nonumber \\
(U_q(\mathrm{so}_n),R(\rho),\Lambda_v)\modu^{Q_{\mathrm{so}_n}} &\cong (U_{(-1)q}(\mathrm{so}_{n}),R(\rho+1),\Lambda_v)\modu. \label{eq_TW}
\end{align}
It is important to note that their characterization of the representation category of type BD is for a smaller
category that does not include the spin representations
and there is no such characterization for the whole category
$(U_q(\mathrm{so}_n),R(\rho))\modu$.
So, to get a monoidal equivalence
$(U_q(\mathrm{so}_n),R(\rho))\modu^{Q_{\mathrm{so}_n}} \cong (U_{(-1)q}(\mathrm{so}_{n}),R(\rho+1))\modu,$
which includes the spin representation, we need to think further. This can be obtained explicitly in the case of type B using the Hopf algebra isomorphism in Appendix.

Hereafter, we will show that
the monoidal equivalences \eqref{eq_TW} are equivalences as braided tensor categories.

We will first consider the case of type $C_n$.
Let $X$ be the vector representation of $U_q(\mathrm{sp}_{2n})$.
In order to show that they are the same braided tensor category,
it suffices to show that the braided on $X$ is the same as above.
The tensor product $X\otimes X$ is a direct sum of three simple objects,
$X^{\otimes 2} \cong X_s \oplus X_a \oplus X_1$, where $X_s$ is the symmetric tensor and
$X_1$ is the trivial representation.
%the antisymmetric tensor is decomposed into the direct sum of the trivial representation $X_1$
%and irreducible rep $X_a$ is the traceless antisymmetric tensor, and $X_1$ is the trivial representation.
%The braiding $B_{X,X} \in \mathrm{End}X^{\otimes 2}$ acts as a scalar on each component.
%In the untwisted $(U_q(\mathrm{sl}_n),R(\rho))$ case,
The eigenvalues of the braiding $B_{X,X} \in \mathrm{End}X^{\otimes 2}$
are 
\begin{align*}
B_{X,X}|_{X_s} &= q,\\
B_{X,X}|_{X_a} &= -q^{-1},\\
B_{X,X}|_{X_1} &= -q^{-(1+2n)}
\end{align*}
(see \cite[Section 8.4.3]{KS}).
Since the vector representation is isomorphic to
the highest weight representation $L_q(\la_n)$ and
$Q_{\mathrm{sp}_{2n}}(\la_n) =\exp ( \pi i \lla \la_n,\la_n\rra)= -1$,
the braidings of $(U_q(\mathrm{sp}_{2n}),R(\rho))\modu^{Q}$
and $(U_{(-1)q}(\mathrm{sp}_{2n}),R(\rho+1))\modu$ are the same on $X^{\otimes 2}$.
Hence, the conjecture is true for $\mathrm{sp}_{2n}$.

Finally, we will consider the case of type $D_n$.
Let $X$ be the vector representation of $\mathrm{so}_{2n}$,
which is isomorphic to the highest weight representation $L_q(\la_1)$.
The tensor product $X\otimes X$ is a direct sum of three simple objects,
$X^{\otimes 2} \cong X_s \oplus X_a \oplus X_1$, where $X_s$ is the traceless symmetric tensor and $X_a$ is the antisymmetric tensor, and $X_1$ is the trivial representation.
The eigenvalues of the braiding $B_{X,X} \in \mathrm{End}X^{\otimes 2}$
are 
\begin{align*}
B_{X,X}|_{X_s} &= q,\\
B_{X,X}|_{X_a} &= -q^{-1},\\
B_{X,X}|_{X_1} &= q^{1-2n}
\end{align*}
(see \cite[Section 8.4.3]{KS} again).
Since $Q(\la_1) =\exp ( \pi i \lla \la_1,\la_1\rra)= -1$,
the braidings of $(U_q(\mathrm{so}_{2n}),R(\rho))\modu^{Q}$
and $(U_{(-1)q}(\mathrm{so}_{2n}),R(\rho+1))\modu$ are the same on $X^{\otimes 2}$.
Hence, the conjecture is true for the subcategory of $\mathrm{so}_{2n}$.

%$L_q(\la_{(n-1)/2})$ (resp. $L_q(\la_1)$) if $n$ is odd, i.e., of type B (resp. if $n$ is even).
%Let $\Lambda_v$ be the subgroup of $\lP$ generated by the root lattice $\lQ$ and $\la_1$.
%$\la_{(n-1)/2}$ (resp. $\la_1$) for type B (resp. type D).
%We note that for type B since $\la_{(n-1)/2} \in \lQ$, $\Lambda_v$ is equal to the root lattice $\lQ$.
%Then, the monoidal subcategory of $(U_q(\mathrm{so}_{2n}),R(\rho))\modu$ generated by the vector representation $X$ is equal to $(U_q(\mathrm{so}_n),R(\rho),\Lambda_v)\modu$.
%We remark that $(U_q(\mathrm{so}_n),R(\rho),\Lambda_v)\modu$ is graded by $\Lambda_v/\lQ=\Z/2\Z$
%and the restriction of the quadratic form $Q:\lP/\lQ \rightarrow \C^\times$
%gives a quadratic form on $\Lambda_v/\lQ \cong \Z/2\Z$. We denote it by the same symbol $Q$.
%,
%which is isomorphic to $\Z_2$ (resp. $0$) for type D (resp. type B).
%Tuba and Wenzl characterize this monoidal category $(U_q(\mathrm{so}_{2n}),R(\rho),\Lambda_v)$ \cite[Theorem 9.4]{TW}.

\section{Constructions}
\label{sec_construction}
By Lemma \ref{lem_rev}, Proposition \ref{op_isomorphism} and Proposition \ref{psi_isomorphism},
we have equivalences of the balanced braided tensor categories
\begin{align*}
((U_{q}(\fg),R(\rho))\modu)^\rev 
&\cong (U_{q}(\fg),R(\rho)_{21}^{-1})\modu\\
&\cong (U_{q}(\fg)^\cop,R(\rho)^{-1})\modu \\
&\cong (U_{q^{-1}}(\fg),R(-\rho))\modu.
\end{align*}

In Section \ref{sec_quantum_coordinate},
we will give commutative algebra objects in 
$(U_{q}(\g),R(\rho)) \otimes (U_{q}(\g)^\cop, R(\rho)^{-1})\modu$.
In Section \ref{sec_lax}, by using it,
we will construct a lax monoidal functor
$$
\mO_\g:\Vect_{\lP/\lQ} \rightarrow (U_{q}(\g),R(\rho))\otimes (U_{q^{-1}}(\g), R(-\rho))\modu.
$$
Then, by Theorem \ref{odd_twist} the grading twist associated with the quadratic form
$Q_\g:\lD \rightarrow \C^\times$ gives us a lax braided monoidal functor
$$
\mO_\g^N:\Vect_{\lD}^{Q_\g^N} \rightarrow (U_{(-1)^N q}(\g),R(\rho+N))\otimes (U_{q^{-1}}(\g), R(-\rho))\modu.
$$
The main Theorem will be proved by using this functor.
As an application, we construct many vertex superalgebras.
This will be discussed in Section \ref{sec_application}.

\subsection{Quantum coordinate ring}
\label{sec_quantum_coordinate}
Let $U_q(\g)^*$ be the dual vector space of $U_q(\g)$.
As we have seen in Section \ref{sec_quasi},
if a Hopf algebra is finite dimensional, then its dual is a Hopf algebra.
In the case of infinite dimension, in order to make the product
\eqref{eq_multiplication} well-defined, we must consider an appropriate subspace of $U_q(\g)^*$.

Let $G$ be the simply-connected simple Lie group associated with the Lie algebra $\g$.
%Let $L$ be a subgroup of weight lattice $\lP$
%such that $\lQ \subset L$.
%and  a finite-dimensional $U_q(\fg)$-module
%and $M^\vee$ be a dual module.
For $\la \in \lP^+$,
$m \in L_q(\la)$ and $f \in L_q(\la)^*$,
define a linear map $c_{m,f}^\la: U_q(\fg) \rightarrow \C$ by
\begin{align*}
c_{m,f}^\la(a) = f(a\cdot m)
\end{align*}
for $a \in U_q(\fg)$.
Let $\mO_q(G)$ be the subspace of $U_q(\g)^*$
spanned by all matrix coefficients
$
\{c_{m,f}^\la \}_{\la \in \lP^+,m \in L_q(\la), f\in L_q(\la)^*}.
$
Then, $\mO_q(G)$ is closed under the product defined by
$$a\cdot b = (a\otimes b) \circ \Delta \text{ for }a,b \in \mO_q(G).$$

As discussed in Section \ref{sec_quasi},
there is a bimodule structure, or equivalently, a left $(U_q(\fg)\otimes U_q(\fg)^\cop)$-module structure on $\mO_q(G)$.
Then, we have a $(U_q(\fg)\otimes U_q(\fg)^\cop)$-module
isomorphism
$$
\oplus_{\la \in \lP^+}c_{\bullet,\bullet}^{\la}:
\bigoplus_{\la\in \lP^+} L_q(\la)\otimes L_q(\la)^* \cong \mO_q(G).
$$

Similarly to the case of dual Hopf algebras,
$\mO_q(G)$ inherits a natural Hopf algebra structure defined by \eqref{eq_multiplication}.
%The multiplication and the co-multiplication are well-defined because of Lemma \ref{grading}.
This Hopf algebra is called {\it a quantum coordinate ring} \cite{KS}.
%\begin{rem}
%We remark that $\mO_q(G)$ is a subalgebra of $\C_q(\g,\lP)$.
%The quantum coordinate ring $\C_q(\g,\lP)$
%is usually denoted by $O({G}_q)$,
%where ${G}$ is the simply-connected simple Lie group
%associated with $\g$.
%Since the center of $G$ is isomorphic to $\lP/\lQ$,
%the choice of the sublattice $L \subset \lP$
%corresponds to the central quotient of ${G}$,
%$G \twoheadrightarrow G_L$.
%Then, the finite dimensional irreducible representations of $G_L$
%is parametrized by $L\cap \lP^+$.
%\end{rem}

Similarly to Lemma \ref{dual_commutative},
by Lemma \ref{lem_rev} and Proposition \ref{op_isomorphism},
we have:
\begin{prop}
\label{commutative_algebra}
The quantum coordinate ring $\mO_q(G)$ is a commutative algebra object in $(U_q(\g),R(\rho)) \otimes (U_q(\fg)^\cop, R(\rho)^{-1})\modu$
and satisfies $\theta_{\mO_q(G)}=\mathrm{id}$.
The unit of $\mO_q(G)$ is given by a natural injection 
$\epsilon:L_{q}(0)\otimes L_{q}(0) \rightarrow \mO_q(G)$.
Furthermore,
if $I \subset \mO_q(G)$ satisfies
\begin{enumerate}
\item
$a m \in I$ for any $m \in I$ and $a \in \mO_q(G)$;
\item
$(u\otimes v)\cdot a \in I$ for any $a \in I$
and $u \otimes v \in U_q(\g)\otimes U_q(\g)^\cop$,
\end{enumerate}
then $I=0$ or $I=\mO_q(G)$.
\end{prop}
\begin{proof}
For $\la \in P^+$, set $\la^* = - w_0(\la)$, where $w_0$ is the longest element in
the Weyl group of $\g$.
Then, the dual module $L_q(\la)^*$ is isomorphic to $L_q(\la^*)$.
By \eqref{eq_balance} and \eqref{eq_balance2},
\begin{align*}
\theta_{L_q(\la)\otimes L_q(\la)^*}
=\exp\left(\frac{\pi i}{k+h^\vee}((\la+2\rho,\la) - (\la^*+2\rho,\la^*))
\right).
\end{align*}
Since
$(\la+2\rho,\la) - (\la^*+2\rho,\la^*)
=(\la+2\rho,\la) - (w_0(\la)-2\rho,w_0(\la))\\
=(\la+2\rho,\la) - (\la-2w_0(\rho),\la)=0$,
we have $\theta_{\mO_q(G)}=\mathrm{id}$.

Assume that 
%$I$ is a proper left ideal,
%i.e., $I\neq 0$ and 
$I\neq \mO_q(G)$.
Since $I$ is stable under the bimodule action,
by $I\neq \mO_q(G)$,
$I\cap L_q(0)\otimes L_q(0)=0$.
Thus, $\epsilon^\vee(I)=0$
or more specifically $a(1)=0$ for any $a \in I$.
By \eqref{eq_bimodule}, this implies that
$a=0$. Hence, $I=0$.
\end{proof}

Consequently, by Proposition \ref{full_vertex}, we have:
\begin{prop}
\label{prop_WZW}
For any $k \in \C \setminus \Q$, there exists a simple full vertex algebra structure on
$$
F_{G,k} = \bigoplus_{\la \in P^+\cap L} L_{\g,k}(\la) \otimes \overline{L_{\g,k}(\la^*)}
$$
as an extension of the full vertex algebra $L_{\g,k}(0)\otimes \overline{L_{\g,k}(0)}$.
\end{prop}
\begin{proof}
It suffices to show the simplicity.
Let $I \subset F_{G,k}$ be a left ideal
such that $I\neq F_{G,k}$.
Since $F_{G,k}$ has a subalgebra which
is isomorphic to $L_{\g,k}(0)\otimes \overline{L_{\g,k}(0)}$,
$I$ is an $L_{\g,k}(0)\otimes \overline{L_{\g,k}(0)}$-module.
Thus, there is a $D(\g,k)\otimes \overline{D(\g,k)}$ homomorphism
$i:I\rightarrow F_{G,k}$
and by the equivalence of categories
it corresponds to
$i': I' \rightarrow \mO_q(G)$.
Then, $I'$ is a left ideal of $\mO_q(G)$
and stable under the action of $U_q(\g)\otimes U_q(\g)^\cop$.
Thus, by Proposition \ref{commutative_algebra},
$I'=0$ and $I=0$, which implies that $F_{G,k}$
is a simple full vertex algebra.
\end{proof}
This full vertex algebra is the underlying algebra of the (analytic continuation of) WZW-model associated with the Lie group $G$ and level $k$.

By Proposition \ref{psi_isomorphism} and Proposition \ref{vertex}, we also have:
\begin{prop}
%For any subgroup $L \subset P$ such that $Q \subset L$,
By $\psi: U_q(\g)\rightarrow U_{q^{-1}}(\g)^\cop$,
$\mO_q(G)$ is a commutative algebra object in $(U_q(\g),R(\rho)) \otimes (U_{q^{-1}}(\fg), R(-\rho))\modu$
and satisfies $\theta_{\mO_q(G)}=\mathrm{id}$.
In particular, for any $k\in \C$ there exists a simple vertex algebra structure on
$$
\bigoplus_{\la \in P^+\cap L} L_{\g,k}(\la) \otimes L_{\g,\bar{k}}(\la^*),
$$
where $\bar{k}=-k-2h^\vee$,
as an extension of the vertex algebra $L_{\g,k}(0)\otimes L_{\g,\bar{k}}(0)$.
\end{prop}
This vertex algebra is called a chiral differential operator constructed by many different ways
 \cite{AG,FS,GMS1,GMS2,Zh}.

%Let $N_{\g,L}$ be a minimal positive integer such that
%$N_{\g,L} \lla \la,\mu\rra \in \Z$ for any $\la,\mu \in L$.
%Then, $N_{\g}=N_{\g,P}$, c.f., Lemma \ref{lem_difference}.
%Let $k,k' \in \C\setminus \Q$ and $t \in \Z$ satisfy 
%$$
%\frac{1}{D(k+h^\vee)} - \frac{1}{D(k'+h^\vee)} = 2tN_{\g,L}.
%$$
%Then, by Corollary \ref{cor_difference}
%and Proposition \ref{commutative_algebra},
%there exists a simple full vertex algebra structure on
%$$
%F_{\fg,L}(k,k',t) = \bigoplus_{\la \in P^+\cap L} V_{\g,k}(\la) \otimes \overline{V_{\g,k'}(\la^*)}.
%$$
%In the next section,
%we will consider the case that the integer $t$ is not in $N_{\g,P}\Z$.

%In this section, we will construct various vertex algebras
%by using the commutative algebra object $\C_q(\g,\lP)$ and abelian cocycles.
%
%
%Set $r^\vee=\frac{\lla \al,\al \rra}{2}$ for a long root $\al$.
%Let $L$ be a subgroup of the weight lattice $\lP$ with $\lQ \subset L$
%and $n_{\g,L}$ be a minimal positive integer
%such that $n_{\g,L}\lla \la,\mu\rra\in \Z$ for any $\la,\mu \in L$.
%Let $k,k' \in \C \setminus \Q$ and $T \in \Z$
%satisfy 
%$$
%\frac{1}{r^\vee(k+h^\vee)}+\frac{1}{r^\vee(k'+h^\vee)}=T.

\subsection{Quantum coordinate ring as lax monoidal functor}
\label{sec_lax}
Recall that in Section \ref{sec_twist_intro} we introduce a category of $\lD$-graded vector space $\Vect_{\lD}$,
which is a braided tensor category with the trivial associative isomorphism and the trivial braiding,
and for $\la \in \lD$, $\C \delta_\la \in \Vect_{\lD}$ is the one-dimensional vector space with the grading $\la$.
We will construct a lax braided monoidal functor
$\mO_\g:\Vect_\lD\rightarrow (U_q(\g),R(\rho))\otimes (U_{q^{-1}}(\g),R(-\rho))\modu$.

For $\la \in \lP/\lQ$,
set
$$
\mO_q(G)_\la = \oplus_{\mu \in (\la +\lQ) \cap \lP^+} L_q(\mu) \otimes L_q(\mu)^*.
$$
Then, 
$\mO_q(G) = \bigoplus_{\la \in \lD} \mO_q(G)_\la$ and $\mO_q(G)$ is a $\lD$-graded algebra,
that is, $\mO_q(G)_\la \cdot \mO_q(G)_{\la'} \subset 
\mO_q(G)_{\la+\la'}$ for any $\la,\la' \in \lD$.

Define a $\C$-linear functor $\mO_\g:\Vect_\lD\rightarrow (U_q(\g),R(\rho))\otimes (U_{q^{-1}}(\g),R(-\rho))\modu$ as follows:
For an object $V=\bigoplus_{\la \in \lD} V_\la \in \Vect_\lD$, 
$$\mO_\g(V) = \bigoplus_{\la \in \lD}
\mO_q(G)_\la \otimes_\C V_\la,$$
where $-\otimes_\C-$ is the tensor product of $\C$-vector spaces.
For a morphism $\{f_\la: V_\la \rightarrow W_\la\}_{\la \in \lD}$,
$$
\mO_\g(f)=\bigoplus_{\la \in \lD}\mathrm{id}_{\mO_q(G)_\la}\otimes f_\la:
\bigoplus_{\la \in \lD}\mO_q(G)_\la \otimes_\C V_\la\rightarrow \bigoplus_{\la \in \lD}\mO_q(G)_\la \otimes_\C W_\la.
$$

Since $\mO_q(G)$ is a $\lD$-graded algebra,
we can define linear maps
$m(\la,\la'): \mO_\g(\la) \times \mO_\g(\la') \rightarrow \mO_\g(\la+\la')$ by 
using the product $\cdot:\mO_q(G)\times \mO_q(G)\rightarrow \mO_q(G)$ for $\la,\la' \in \lD$.
Since $\mO_q(G)$ is a left $U_q(\g)\otimes U_{q^{-1}}(\g)$-module and the product $\cdot$
is compatible with this left module structure,
$m(\la,\la') \in \mathrm{Hom}_{U_q(\g)\otimes U_{q^{-1}}(\g)}(\mO_\g(\la) \otimes \mO_\g(\la'), \mO_\g(\la+\la'))$.

Thus, for any objects $V, W \in \Vect_\lD$,
we have a natural transformation
\begin{align*}
m_{V,W}:\mO_\g(V)\otimes \mO_\g(W) \rightarrow \mO_\g(V\otimes W)
\end{align*}
defined by 
\begin{align*}
m_{V,W}:
&(\mO_q(G)_\la \otimes_\C V_\la)\otimes_{U_q} (\mO_q(G)_{\la'} \otimes_\C W_{\la'})
\cong (\mO_q(G)_\la \otimes_{U_q} \mO_q(G)) \otimes_\C (V_\la \otimes_\C W_{\la'})\\
&\overset{m(\la,\la')\otimes \id_{V_\la\otimes W_{\la'}}}{\longrightarrow} \mO_q(G)_{\la+\la'} \otimes_\C (V_\la \otimes_\C W_{\la'})
\end{align*}
for any $\la,\la'\in \lD$ and the linear extension of it.
Let $\epsilon:L_{q}(0)\otimes L_{q}(0) \rightarrow \mO_q(G)_0$
be the natural injection.
Then, we have:
\begin{lem}
\label{lem_lax0}
The functor $\mO_\g:\Vect_\lD\rightarrow (U_q(\g),R(\rho))\otimes (U_{q^{-1}}(\g),R(-\rho))\modu$
together with a morphism $\epsilon:L_{q}(0)\otimes L_{q}(0) \rightarrow \mO_\g(\1_{\Vect_\lD})$
and the natural transformation $m_{V,W}:\mO_\g(V)\otimes \mO_\g(W) \rightarrow \mO_\g(V\otimes W)$
is a lax braided monoidal functor.
\end{lem}
\begin{proof}
Since both $\Vect_\lD$ and $(U_q(\g),R(\rho))\otimes (U_{q^{-1}}(\g),R(-\rho))\modu$
are strict monoidal categories, that is, the associative isomorphism $\al$
and units $l,r$ is trivial,
(LM1) and (LM2) follows from the fact that $\mO_q(G)$ is an associative algebra.
Furthermore, since $\mO_q(G)$ is a commutative algebra object,
$\mO_\g$ is a lax braided monodial functor.
\end{proof}

We consider the $\lD$-grading on $(U_q(\g),R(\rho))\otimes (U_{q^{-1}}(\g),R(-\rho))\modu$
obtained from the $\lD$-grading on the left component, $(U_q(\g),R(\rho))\modu$.
Let us consider the grading twist of 
$\Vect_\lD$ and $(U_q(\g),R(\rho))\otimes (U_{q^{-1}}(\g),R(-\rho))\modu$
by the quadratic form $Q_\g^N:\lD\rightarrow \C^\times$ for $N\in \Z$.
Since the grading twist does not change the underlying category structure, $\mO_\g$ still gives a functor $\Vect_\lD^{Q_\g} \rightarrow (U_q(\g),R(\rho))\otimes (U_{q^{-1}}(\g),R(-\rho))\modu^{Q_\g}$. We denote it by $\mO_\g^N$.
Since $\mO_\g$ preserves the gradings on $\Vect_\lD$
and $(U_q(\g),R(\rho))\otimes (U_{q^{-1}}(\g),R(-\rho))\modu$,
the following proposition clearly follows from the above lemma:
\begin{prop}
\label{prop_lax}
The functor $\mO_\g^N:\Vect_\lD^{Q_\g^N} \rightarrow (U_{q}(\g),R(\rho))\modu^{Q_\g^N} \otimes (U_{q^{-1}}(\g),R(-\rho))\modu$
is a lax braided monoidal functor.
\end{prop}

Now, we can prove the main theorem.
Let $M$ be an even lattice, $M^\vee$ a dual lattice, and
$Q_M:M^\vee/M\rightarrow \C^\times$ a quadratic form given by
$Q_M(\la)=\exp(\pi i (\la,\la))$.
Let $k,k' \in \C \setminus \Q$ and $N \in \Z$
satisfy 
$$
\frac{1}{r^\vee(k+h^\vee)}+\frac{1}{r^\vee(k'+h^\vee)}=N.
$$
For $\la \in \lD$, set 
$$
L_{\g,k,k'}^N(\la+\lQ) = \bigoplus_{\mu \in (\la +\lQ) \cap \lP^+} L_{\g,k}(\mu) \otimes L_{\g,k}(\mu)^*.
$$
Then, we have:
\begin{thm}
\label{thm_vertex}
Let $r \in \Z_{>}$ and $\g_i$ be simple Lie algebras
and $k_i,k_i' \in \C \setminus \Q$ and $N_i \in \Z$
satisfy 
$$
\frac{1}{r^\vee(k_i+h_i^\vee)}+\frac{1}{r^\vee(k_i'+h_i^\vee)}=N_i
$$
for $i=1,\dots,r$.
Let $M$ be an even lattice
and $(A,Q)$ a quadratic space defined by
\begin{align*}
A=(\bigoplus_{i=1}^r \lP_i/\lQ_i) \oplus M^\vee/M,\\
Q=(\bigoplus_{i=1}^r Q_{\g_i}^{N_i})\oplus Q_M.
\end{align*}
Let $(I,p)$ be a super isotropic subspace of the quadratic space $(A,Q)$.
Set 
$$
V_{\vec{\g},\vec{k},\vec{k'},M}^{\vec{N}}(I)=\bigoplus_{(\la_1,\dots,\la_r,\mu) \in I} \bigotimes_{i=1}^r L_{\g_i,k_i,k_i'}^{N_i}(\la_i+\lQ_i)\otimes V_{\mu+M}
$$
for $(\g,k,k',M,I)$ with $\vec{\g}=(g_1,\dots,g_r)$ and $\vec{k}=(k_1,\dots,k_r),\vec{k'}=(k'_1,\dots,k'_r),\vec{N}=(N_1,\dots,N_r)$.
Assume that for each $a=1,\dots,r$
one of the following conditions is satisfied:
%Suppose that $N$ is even or one of the following conditions hold:
\begin{enumerate}
\item
$N_a$ is even;
\item
The simple Lie algebra $\g_a$ is of type ABC;
\item
The simple Lie algebra $\g_a$ is of type D
and $\mathrm{pr}_a(I) \subset \Lambda_v/\lQ_i$,
where $\mathrm{pr}_a:(\oplus_{i=1}^r \lP_i/\lQ_i) \oplus M^\vee/M
\rightarrow \lP_a/\lQ_a$ is the projection to the $a$-th component.
\end{enumerate}
Then, there is a simple vertex superalgebra structure on $V_{\vec{\g},\vec{k},\vec{k'},M}^{\vec{N}}(I)$
as an extension of $\left(\bigotimes_{i=1}^r L_{\g_i,k_i}(0)\otimes {L_{\g_i,k'_i}(0)}\right) \otimes V_M$.
Furthermore, the even part ($s=0$) and the odd part ($s=1$) are given by
$$
V_{\vec{\g},\vec{k},\vec{k'},M}^{\vec{N}}(I)_s=\bigoplus_{\substack{(\la_1,\dots,\la_r,\mu) \in I\\p(\la_1,\dots,\la_r,\mu)=s}} \bigotimes_{i=1}^r L_{\g_i,k_i,k_i'}^{N_i}(\la_i+\lQ)\otimes V_{\mu+M}.
$$
\end{thm}
\begin{proof}
For the sake of simplicity, we consider only the case of $r=1$.
Let $q \in \C \setminus \Q$.
First, we assume that $\g$ is not of type D.
By the assumption, Conjecture \ref{conj} is true for $(\g,N)$ by Proposition \ref{even_twist} and
Theorem \ref{odd_twist}.
Combining this with Proposition \ref{prop_lax},
we have a lax braided monoidal functor
$$
\mO_\g^N:\Vect_\lD^{Q_\g^N} \rightarrow (U_{(-1)^N q}(\g),R(\rho+N)) \otimes (U_{q^{-1}}(\g),R(-\rho))\modu.
$$
By Proposition \ref{prop_lattice}, we also have a lax braided monoidal functor 
$$
\Vect_{\lD}^{Q_g^N} \otimes \Vect_{M^\vee/M}^{Q_M} \rightarrow (U_{(-1)^N q}(\g),R(\rho+N)) \otimes (U_{q^{-1}}(\g),R(-\rho))\modu \otimes V_M \modu.
$$
By Lemma \ref{lem_commutative},
we can associate to a super isotropic subspace $(I,p)$
a supercommutative algebra object 
$S(I) \in \Vect_{\lD\oplus M^\vee/M}^{Q_\g^N\oplus Q_M} = \Vect_{\lD}^{Q_\g^N} \otimes \Vect_{M^\vee/M}^{Q_M}$.
By Lemma \ref{lem_lax}, the image of $S(I)$ by the lax braided monoidal functor 
is again a supercommutative algebra object.
Then, by Theorem \ref{thm_DK} and Proposition \ref{super_vertex},
$S(I)$ corresponds to an extension of $L_{\g,k}(0)\otimes {L_{\g,k'}(0)} \otimes V_M$
with
\begin{align*}
\rho +N = \frac{1}{r^\vee(k+h^\vee)},\;\;\;\;\;\;\;
-\rho = \frac{1}{r^\vee(k'+h^\vee)},
\end{align*}
which implies that 
$$
\frac{1}{r^\vee(k+h^\vee)}+\frac{1}{r^\vee(k'+h^\vee)}=N.
$$
In the case of type D, by Theorem \ref{odd_twist}, we have a lax braided monoidal functor
$$
\mO_\g^N:\Vect_{\Lambda_v/\lQ}^{Q_{\mathrm{so}_{2n}}^N} \rightarrow (U_{(-1)^N q}(\mathrm{so}_{2n}),R(\rho+N),\Lambda_v) \otimes (U_{q^{-1}}(\mathrm{so}_{2n}),R(-\rho),\Lambda_v)\modu,
$$
which is enough to show the assertion by $\mathrm{pr}_1(I) \subset \Lambda_v/\lQ$.
The simplicity follows from Proposition \ref{commutative_algebra} and the definition of $S(I)$.
\end{proof}

\begin{rem}
A similar statement is true for a full vertex algebra.
In this case, the condition for the levels is
$$
\frac{1}{r^\vee(k+h^\vee)}- \frac{1}{r^\vee(k'+h^\vee)}=N.
$$
\end{rem}

\begin{rem}
\label{rem_AIA}
As a consequence of the above proof, we see that
\begin{align*}
\bigoplus_{\la \in \lD} L_{\g,k,k'}^N(\la+\lQ)
\end{align*}
is an abelian intertwining algebra with abelian cocycle $\mathrm{EM}^{-1}(\lQ_{\g}^N) \in H_\ab^3(\lD,\C^\times)$.
\end{rem}

%\begin{cor}
%\label{thm_cor}
%Let $\g$ be a simple Lie algebra and $k,k' \in \C \setminus \Q$ and $N \in \Z$
%satisfy 
%$$
%\frac{1}{r^\vee(k+h^\vee)}+\frac{1}{r^\vee(k'+h^\vee)}=N.
%$$
%Let $I$ be a subgroup of $\lD$ such that $Q_\g(\la)^N \in \{1,-1\}$ for any $\la \in I$. Set
%$$
%V_{\g,k,k'}^N(I) = \bigoplus_{\substack{\la \in I}}L_{\g,k,k'}^N(\la)
%$$
%for $(\g,k,k',I)$.
%Assume that one of the following conditions is satisfied:
%%Suppose that $N$ is even or one of the following conditions hold:
%\begin{enumerate}
%\item
%$N$ is even;
%\item
%The simple Lie algebra $\g$ is of type ABCF;
%\item
%The simple Lie algebra $\g$ is of type D
%and $I \subset \Lambda_v/\lQ$.
%\end{enumerate}
%Then, there is a simple super vertex algebra structure on $V_{\g,k,k'}^N(L)$
%as an extension of $L_{\g,k}(0)\otimes L_{\g,k'}(0)$.
%\end{cor}

\subsection{Applications}
\label{sec_application}
In this section, we introduce some examples of vertex (super)algebras which can be constructed from Theorem \ref{thm_vertex}.

%The following lemma is convenient:
%\begin{cor}
%\label{cor_period}
%Let $n_\g$ be the smallest positive integer such that $Q^{n_\g}(\la)=1$ for any $\la \in \lP$.
%Then, 
%$(U_q(\g),R(\rho))\modu$ 
%and $(U_{(-)^{n_\g}q}(\g),R(\rho+n_\g))\modu$ are equivalent as braided tensor categories.
%\end{cor}
%
%By Corollary \ref{cor_period}, we have
%\begin{lem}
%\label{lem_period}
%Assume that $\g$ is not of type $E_6$ or $G_2$.
%If $(\g,N,M,I)$ is an admissible quadruple,
%then $(\g,N+n_\g,M,I)$ is also an admissible quadruple.
%\end{lem}
%\begin{rem}
%If Conjecture \ref{conj} is true for type $E_6$ and $G_2$,
%then the lemma is true for them.
%And even if we don't assume the conjecture, we can at least show the following claims:
%If $(\g,N,M,I)$ is an admissible quadruple,
%then $(\g,N+2n_\g,M,I)$ is also an admissible quadruple.
%\end{rem}

In the construction, it is important to know the quadratic space $(\lD,Q_\g)$.
We will use the labeling of the Dynkin diagram in \cite{Hu}.
Let $\{\la_i\}_{i=1,\dots,n}$ be the fundamental weight of $\g$.
Table \ref{table_quadratic} summarizes the information about the quadratic space $(\lD,Q_\g)$:
\begin{table}[h]
\caption{Quadratic space}
\label{table_quadratic}
  \begin{tabular}{|c|c|c|c|} \hline
type & $\lD$ & Generator & Value of $Q_\g$ \\ \hline \hline
$A_{n-1}$ & $\Z/n\Z$ & $\la_1$ & $Q(\la_1)=\exp(\frac{n-1}{n}\pi i)$\\
$B_{n}$ & $\Z/2\Z$ & $\la_n$ & $Q(\la_n)=i^n$\\
$C_n$ & $\Z/2\Z$ & $\la_n$ & $Q(\la_n)=-1$\\
$D_{2n}$ & $\Z/2\Z \times \Z/2\Z$ & $\la_1,\la_{2n},\la_{{2n}-1}$ & $Q(\la_1)=-1,Q(\la_{{2n}-1})=Q(\la_{2n})=i^n$\\
$D_{2n+1}$ & $\Z/4\Z$ & $\la_{2n+1}$ & $Q(\la_{2n+1})=
\exp(\frac{2n+1}{4}\pi i)$\\
$E_6$ & $\Z/3\Z$ & $\la_6$ & $Q(\la_6)=\exp(\frac{4}{3}\pi i)$\\
$E_7$ & $\Z/2\Z$ & $\la_7$ & $Q(\la_7)=-i $\\
$E_8, F_4, G_2$ & $0$ & - & - \\ \hline
\end{tabular}
\end{table}

We note that for type D,
$\la_1$ corresponds to the vector representation,
and $\la_{n-1},\la_n$ correspond to the spin representation
and its conjugate representation. For type D, let $\La_s$ (resp. $\La_c$ and $\La_v$) be a subgroup of the weight lattice generated by $\lQ$ and $\la_{n}$ (resp. $\la_{n-1}$ and $\la_1$).
%Type D needs to be considered carefully.
%In the case of type D, there are three possibilities for $L\subsetneq \lP$.  According to \cite{Hu}, the root system of type D can be written as 
%$\{\pm e_i\pm e_j\}$,
%where $\{e_i\}_{i=1,2,\dots,n}$ is the standard basis of $\R^n$,
%and
%the simple roots as 
%$\al_1=e_1-e_2,\al_2=e_2-e_3,\dots, \al_{n-2}=e_{n-2}-e_{n-1},
%\al_{n-1}=e_{n-1}-e_n, \al_n=e_{n-1}+e_n$.
%Then, the weight lattice is spanned by
%$\{e_i, \frac{e_1+e_2+\dots+e_n}{2}\}_{i=1,2,\dots,n}$.
%%Set
%\begin{align*}
%\La_s = \frac{e_1+e_2+\dots,e_n}{2} 
%\end{align*}

\begin{rem}
To check the value of $Q_\g$ in the case of type $E_6$ or $E_7$, 
it is convenient to use the fact that $\Vect_{A_k}^{Q_{A_k}}$ and $\Vect_{E_{8-k}}^{Q_{E_{8-k}}}$ are braided reverse equivalent.
In fact, the unimodular lattice $E_8$ is an index $2$ (resp. index $3$) extension of $A_1\oplus E_7$ (resp.
$A_2 \oplus E_6$).
Thus, $Q(\la_7^{E_7})Q(\la_1^{A_1})$ (resp. $Q(\la_6^{E_6})Q(\la_1^{A_2})$) should be equal to $1$.
\end{rem}

%\begin{rem}
%\label{rem_rare}
%$N_{\g,\lP}=1$ if $\g = B_{2n}, C_n, E_8,F_4,G_2$.
%If $N_{\g,L} \in 2\Z$,
%then the type of $\g$ is one of $A_{2n+1}, B_{2n+1}, D_n, E_7.$ For type $B_n$, there is another construction of vertex algebras, see Theorem \ref{thm_B}.
%\end{rem}
%
%\begin{table}[h]
%\caption{Non trivial values}
%\label{table_shift}
%  \begin{tabular}{|c|c|} \hline
%type & $N_{\g,\lP}$ \\ \hline \hline
%$A_{n}$ & $n+1$ \\
%$B_{2n+1}$ & $2$\\
%$D_{2n+1}$ & $4$ \\
%$D_{2n}$ & $2$ \\
%$E_6$ & $3$ \\
%$E_7$ & $2$ \\ \hline
%\end{tabular}
%\end{table}
The following lemma is obvious:
\begin{lem}
\label{lem_isotropic}
Let $(A,Q)$ be a quadratic space and $I$ a subgroup of $A$.
If $Q(\al) \in \{1,-1\}$ for any $\al \in A$ and
$I$ is a cyclic group, then $I$ is a super isotropic subspace.
\end{lem}
\begin{rem}
If $I$ is not cyclic, then the above lemma is not true.
For example, $D_{2n}$ satisfies $Q_{D_{2n}}(\al)\in \{1,-1\}$
 for any $\al \in \Z/2\Z\times \Z/2\Z$, however, is not super isotropic.
\end{rem}

Then, for example, we have:
\begin{prop}
\label{example_GL}
Let $n \geq 2$ and $k,k' \in \Z$, $N\in \Z$ and $s=\{1,-1\}$ satisfy
\begin{align*}
\frac{1}{k+n}+\frac{1}{k'+n}=s+ n N.
\end{align*}
Then, $\bigoplus_{a=0}^{n-1}L_{\mathrm{sl}_n,k,k'}^{1+nN}(a\la_1+\lQ) \otimes V_{\frac{a\sqrt{sn}}{n}+\sqrt{sn}\Z}$
is a simple vertex superalgebra,
where $V_{\sqrt{sn}\Z}$ is a lattice vertex (super)algebra
associated with a rank one lattice such that  $(\al,\al)=sn$.
\end{prop}
\begin{proof}
We first assume that $n$ is even.
Then, $\sqrt{sn}\Z$ is an even lattice.
Let $I \subset \lD\oplus (\sqrt{sn}\Z)^\vee/\sqrt{sn}\Z\cong (\Z/n\Z)^2$
be a subgroup generated by $(\la_1,\frac{\sqrt{sn}}{n})$.
Since 
\begin{align*}
Q(\la_1,\frac{\sqrt{sn}}{n})=
Q_{\mathrm{sl}_n}(\la_1)^s Q_{\sqrt{sn}\Z}(\frac{\sqrt{sn}}{n})
=\exp(s\pi i \frac{n-1}{n})\exp(s\pi i\frac{1}{n})
=-1,
\end{align*}
by Lemma \ref{lem_isotropic}, $I$ is a super isotropic subspace.
Thus, the assertion follows from Theorem \ref{thm_vertex}.

We next consider the case of $n$ is odd
and consider the even lattice $\sqrt{4sn}\Z$.
Let $I \subset \lD\oplus (\sqrt{4sn}\Z)^\vee/\sqrt{4sn}\Z\cong 
\Z/n\Z\times \Z/4n\Z$
be a subgroup generated by $(\la_1,\frac{\sqrt{4sn}}{2n})$.
Then, $I \cong \Z/2n\Z$ and by the above calculus $I$ is a super isotropic. 
Hence,
\begin{align}
V = \bigoplus_{b=0}^{2n-1}L_{\mathrm{sl}_n,k,k'}^{1+nN}(b\la_1+\lQ) \otimes V_{\frac{b\sqrt{sn}}{n}+\sqrt{4sn}\Z} \label{eq_gl}
\end{align}
is a vertex superalgebra.
Since $n(\la_1,\frac{\sqrt{4sn}}{2n})=(0,\sqrt{sn}) \in \lD\oplus (\sqrt{4sn}\Z)^\vee/\sqrt{4sn}\Z$,
$V$ contains a subalgebra isomorphic to the lattice vertex superalgebra $V_{\sqrt{sn}\Z}$.
Rewrite \eqref{eq_gl} as the sum of $V_{\sqrt{sn}\Z}$-modules to get the assertion.
\end{proof}

If $r=1$ and $M = 0$ in Theorem \ref{thm_vertex}, then
the super isotropic subspace $I$ corresponds to a subgroup of 
$\lD$, which is denoted by $V_{\g,k,k'}^N(I)$.
Even when the shift $N$ is even and $M=0$, various nontrivial vertex superalgebras can be constructed, for example:
\begin{prop}
\label{example_even}
For any $n \geq 1$, 
\begin{enumerate}
\item
For any $N \in \Z$,
$V_{\mathrm{so}_{4n+a},k,k'}^{2+4N}(\lP)$
is a vertex superalgebra if $a=0,3$,
and a non-super vertex algebra if $a=1$,
and an abelian intertwining algebra if $a=2$;
\item
For any $N \in \Z$,
$V_{\mathrm{so}_{4n+a},k,k'}^{1+2N}(\Lambda_v)$
is a vertex superalgebra if $a=2$,
and a non-super vertex algebra if $a=0,1,3$.
\item
For any $N \in \Z$,
$V_{\mathrm{e}_{7},k,k'}^{2+4N}(\lP)$ is a vertex superalgebra;
\item
For any $N \in \Z$,
$V_{\mathrm{sl}_{n},k,k'}^{N}(m\lP+\lQ)$
is a vertex superalgebra if $\exp(\frac{m^2N(n-1)}{n}\pi i)=-1$
and a non-super vertex algebra if $\exp(\frac{m^2N(n-1)}{n}\pi i)=1$.
\end{enumerate}
\end{prop}

All possible choices of $(\g,N,I)$ (except for A type) is summarized in Table \ref{table1}:
In the case of type $A_{n}$, 
 there are various ways of the subgroup $I \subset \lD$ (see
for example Proposition \ref{example_even} (4)).

\begin{table}[htbp]
  \begin{tabular}{cc}
    \begin{minipage}[c]{0.5\hsize}
      \centering
  \begin{tabular}{|c||c|c|c|} \hline
type & shift & $I$ & super \\ \hline \hline
$B_{2n}$ & $1+2\Z$ & $\lP$ & S \\
$B_{2n}$ & $2\Z$ & $\lP$ & \\
$B_{2n+1}$ & $\Z$ & $\lQ$ & \\
$B_{2n+1}$ & $2+4\Z$ & $\lP$ & S \\
$B_{2n+1}$ & $4\Z$ & $\lP$ &\\ \hline
$C_n$ & $1+2\Z$ & $\lP$ & S \\
$C_n$ & $2\Z$ & $\lP$ & \\ \hline
$D_{n}$ & $1+2\Z$ & $\La_v$ & S \\
$D_{n}$ & $2\Z$ & $\La_v$ & \\
$D_{4n+2}$ & $2+4\Z$ & $\lP$ & S\\
$D_{4n}$ & $2\Z$ & $\lP$ & \\
$D_{2n}$ & $4\Z$ & $\lP$ & \\
$D_{2n+1}$ & $4+8\Z$ & $\lP$ & S \\
$D_{2n+1}$ & $8\Z$ & $\lP$ & \\ \hline
 \hline
\end{tabular}
\caption{List of non-super triple}
\label{table1}
    \end{minipage} &
    \begin{minipage}[c]{0.5\hsize}
      \centering
  \begin{tabular}{|c||c|c|c|} \hline
type & shift & $I$ & super\\ \hline \hline
$E_{6}$ & $6\Z$ & $\lP$ &  \\
$E_{6}$ & $2\Z$ & $\lQ$ &  \\
$E_{7}$ & $2+4\Z$ & $\lP$ & S \\
$E_{7}$ & $4\Z$ & $\lP$ & \\
$E_{8}$ & $2\Z$ & $\lP$ &  \\ \hline
$F_{4}$ & $2\Z$ & $\lP$ &  \\ \hline
$G_{2}$ & $2\Z$ & $\lP$ &  \\
 \hline
\end{tabular}
\caption{List of super triple}
\label{table_S}
%\vspace{5mm}
%      \centering
%  \begin{tabular}{|c||c|c|c|} \hline
%type & shift & $I$  & super \\ \hline \hline
%$A_{2n-1}$ & $2n+4n\Z$ & $\lP$ & S \\ 
%$A_{2n}$ & $(2n+1)+(4n+2)\Z$ & $\lP$ & \\ 
%$A_{n}$ & $\Z$ & $\lQ$ & \\  \hline
%%$A_5$ & $2+4\Z$ & $3\lP+\lQ$& S \\
%%$A_5$ & $6\Z$ & $2\lP+\lQ$& \\
%%$A_7$ & $2+4\Z$ & $2\lP+\lQ$& S \\
%%$A_9$ & $2+4\Z$ & $5\lP+\lQ$& S \\
%%%$A_9$ & $10\Z$ & $2\lP+\lQ$&  \\
%$A_{11}$ & $4+8\Z$ & $3\lP+\lQ$& S \\
%$A_{11}$ & $6\Z$ & $2\lP+\lQ$& \\
%%$A_{13}$ & $2+4\Z$ & $7\lP+\lQ$& S \\
%%$A_{13}$ & $14\Z$ & $2\lP+\lQ$& \\
%&$\cdots$ & & \\ \hline
% \hline
%\end{tabular}
%\caption{Examples of super and non-super A type triple}
%\label{table_A}
    \end{minipage}
  \end {tabular}
\end{table}

Finally, we give examples of vertex superalgebras with $r \geq 2$:
\begin{prop}
\label{example_r}
For any $n,m \geq 1$, 
\begin{enumerate}
\item
For any $N_1,\dots,N_{n-1} \in \Z$,
\begin{align*}
\bigoplus_{\la \in \lD} \bigotimes_{i=1}^{n} 
L_{\mathrm{sl}_n,k,k'}^{1+n N_i}(\la+\lQ)
\end{align*}
is a vertex superalgebra if $n$ is even,
and a non-super vertex algebra if $n$ is odd.
\item
For any $N_A,N_B \in \Z$,
\begin{align*}
\bigoplus_{a \in \Z/2n\Z}
L_{\mathrm{sl}_{2n},k,k'}^{n+2nN_A}(a\la_1+\lQ)
\otimes L_{\mathrm{so}_{4m+3},k,k'}^{1+2N_B}(a\la_{2m+1}+\lQ)
\end{align*}
is a vertex superalgebra if $(-1)^{n+m+N_A+N_B}=-1$,
and a non-super vertex algebra if $(-1)^{n+m+N_A+N_B}=1$.
\end{enumerate}
\end{prop}

\section*{Appendix}
\label{sec_appendix}
%An important key to the above construction is the Hopf algebra isomorphism (Proposition \ref{psi_isomorphism})
%$\psi: U_q(\g) \rightarrow U_{q^{-1}}(\g)^\cop.$
In Appendix, we will prove Conjecture \ref{conj} for the simple Lie algebras of type $B_n$ ($n \geq 1$).
The reader may wonder why our proof only applies to type B.
In order to clarify the path of the proof, we will first explain the reason.
%It is basically for the reasons explained below.
Let $\g$ be a simple Lie algebra (not assumed to be of type B).
We want to relate the representation of $U_q(\g)$ to the representation of $U_{-q}(\g)$.
So we think of a representation that is not type 1.
That is, for each $\la \in \lP$ let $\C\chi_\la$ be a one-dimensional representation of $U_q(\g)$ defined by
$$
E_\al \cdot \chi_\la=F_\al \cdot \chi_\la=0,\;\;\;\;\; K_\al \cdot \chi_\la = (-1)^{\lla \al, \la \rra}\chi_\la \text{ for }\al \in \Pi.
$$
Then, for a type 1 module $M$, $M\otimes \C\chi_\la$ and $\C\chi_\la \otimes M$ are $U_q(\g)$-modules, which are not of type 1.
For example, for $L_q(\la)\otimes \C\chi_\la$, we have
$$
K_\al \cdot v_\la\otimes \chi_\la = q^{\lla \al,\la \rra}(-1)^{\lla \al,\la \rra} v_\la\otimes \chi_\la
$$
for the highest weight vector $v_\la \in L_q(\la)$.
In order for the correspondence of sending $L_q(\la)$ to $L_q(\la)\otimes \C\chi_\la$
to be compatible with ``the graded twist'' on $(U_q(\g),R(\rho))\modu$,
we expect $\chi_{\la+\al}=\chi_\la$ to hold for any $\al \in \lQ$ and $\la \in \lP$.
However, we have:
\begin{lem}
\label{lem_compatible}
The following conditions are equivalent:
\begin{enumerate}
\item
$\chi_{\la+\al}=\chi_\la$ for any $\al \in \lQ$ and $\la \in \lP$;
\item
$\lla \al,\be\rra \in 2\Z$ for any $\al,\be \in \lQ$;
\item
The simple Lie algebra $\g$ is of type $A_1$ or of type $B_n$ ($n\geq 2$).
\end{enumerate}
\end{lem}
Therefore, only in the case of type B, the correspondence of sending $L_q(\la)$ to $L_q(\la)\otimes \chi_\la$ is compatible with the $\lP/\lQ$-grading.

The following lemma is very important (this lemma itself holds for any simple Lie algebra).
\begin{lem}
\label{natural_h}
Let $\ga \in \lP$.
For any type 1 module $M \in U_q(\g)\modu$, define a linear map 
$h_M^\ga:M\otimes \C\chi_\ga \rightarrow \C\chi_\ga \otimes M$
by
\begin{align*}
h_M^\ga (m_\la \otimes \chi_\ga)=\exp(\pi i (\ga,\la))\chi_\ga \otimes m_\la
\end{align*}
for any $\la \in \lP$ and $m_\la \in M_\la$.
Then, $h_M^\ga$ is a $U_q(\g)$-module homomorphism.
In particular, the family of the maps $\{h_M^{\ga}\}_{M \in U_q(\g)\modu}$ is
a natural transformation of $- \otimes \C \chi_\ga$ and $\C \chi_\ga \otimes -$.
\end{lem}
\begin{proof}
Let $\la \in \lP$ and $m_\la \in M_{\la}$ and $\al \in \Pi$.
Since
\begin{align*}
E_\al \cdot (m_\la \otimes \chi_\ga) &= \Delta(E_\al) \cdot (m_\la \otimes \chi_\ga)
= (E_\al\otimes 1+K_\al\otimes E_\al )\cdot (m_\la \otimes \chi_\ga)= (E_\al\cdot m_\la) \otimes \chi_\ga,
\end{align*}
and $E_\al\cdot m_\la \in M_{\la+\al}$
we have
$h_{M}^\ga(E_\al \cdot (m_\la \otimes \chi_\ga))= \exp(\pi i (\ga,\la+\al)) (\chi_\ga \otimes E_\al \cdot m_\la).$

Similarly, since $E_\al \cdot (\chi_\ga \otimes m_\la) = (E_\al\otimes 1+K_\al\otimes E_\al )\cdot (\chi_\ga \otimes m_\la)= 
(K_\al \cdot \chi_\ga \otimes E_\al \cdot m_\la)$,
and $K_\al\cdot \chi_\ga = (-1)^{\lla \ga,\al \rra}\chi_\ga$,
we have
\begin{align*}
E_\al \cdot h_{M}^\ga(m_\la \otimes \chi_\ga)= \exp(\pi i (\ga,\la)) E_\al \cdot (\chi_\ga \otimes m_\la)
=\exp(\pi i (\ga, \la+\al))(\chi_\ga \otimes E_\al\cdot  m_\la).
\end{align*}
Hence, $h_{M}^\ga(E_\al \cdot (m_\la \otimes \chi_\ga))=E_\al \cdot h_{M}^\ga(m_\la \otimes \chi_\ga)$ for any $\al \in \Pi$.
It is easy to check this for $F_\al$ and $K_\al$ and thus $h_M^\ga$ is a $U_q(\g)$-module homomorphism.
The naturality is obvious.
\end{proof}

We will now proceed to the case of type B.
We first recall the explicit descriptions of the root system of those Lie algebras.
According to \cite{Hu}, the root system of type $B_n$ can be written as 
$$\{\pm e_i\pm e_j, e_i\}_{1 \leq i,j \leq n},$$
where $\{e_i\}_{i=1,2,\dots,n}$ is the standard basis of $\R^n$,
and
the simple roots and the fundamental weights as 
\begin{align*}
(\al_1,\al_2,\dots,\al_{n-1},\al_n)&=(e_1-e_2,e_2-e_3,\dots, e_{n-1}-e_n, e_n),\\
(\la_1,\la_2,\dots,\la_{n-1},\la_n) &=(e_1,e_1+e_2,\dots,e_1+e_2+\dots+e_{n-1}, \frac{e_1+e_2+\dots+e_n}{2}).
\end{align*}
The weight lattice is spanned by $\{e_i, \lambda_n \}_{i=1,2,\dots,n}$
and $\lP/\lQ \cong \Z_2$ is generated by $\lambda_n$.
It is noteworthy that by the normalization $\lla e_i, e_i \rra=2$ for any $i=1,\dots,n$ (see Lemma \ref{lem_compatible})
and $\lla \la_n, \la_n \rra=\frac{n}{2}$.
Let us denote $E_{\al_i},F_{\al_i},K_{\al_i}$ by
$E_i,F_i,K_i$ for short.

%Similarly, the root system of type $F_4$ can be written as 
%$$\{\pm e_i\pm e_j\}_{1 \leq i,j \leq 4} \cup
% \{\frac{\pm e_1\pm e_2\pm e_3\pm e_4}{2}\},$$
%where the parity of $\frac{\pm e_1\pm e_2\pm e_3\pm e_4}{2}$ must be even,
%that is, the inner product with the vector $e_1+e_2+e_3+e_4$ must be an even integer.
%Also, the simple roots and the fundamental weights are written as
%\begin{align*}
%(\al_1,\al_2,\al_3,\al_4) &= 
%\left(\frac{e_1-e_2+e_3-e_4}{2},e_2-e_3,e_3-e_4,e_4\right)\\
%(\la_1,\la_2,\la_3,\la_4)&=(,,,-2e_4).
%\end{align*}
%Set $\lambda_n= \frac{e_1+e_2+\dots+e_n}{2}$.
%Then, the weight lattice is spanned by
%$\{e_i, \lambda_n \}_{i=1,2,\dots,n}$
%and $\lP/\lQ \cong \Z_2$ is generated by $\lambda_n$.
%It is noteworthy that by the normalization $\lla e_i, e_i \rra=2$ for any $i=1,\dots,n$.

We will define a type 2 module of $U_q(\mathrm{so}_{2n+1})$.
We first observe that by Lemma \ref{lem_compatible}
the one dimensional representation $\C\chi_{\ga}$ is only depends on $\ga \in \lP/\lQ=\Z/2\Z$.
Denote $\chi_{\la_n}$ by $\chi$.

	For each $\lambda \in \lP^+$, let $L_q^\tw(\lambda)$ be the unique irreducible highest module 
defined by
%of highest weight $\lambda$, that is, there exists nonzero $v_\lambda \in L_q^\tw(\lambda)$ such that
	\[K_i v_\lambda = (-q)^{\lla \al_i,\lambda\rra} v_\lambda,\;\;\;\; E_i v_\lambda = 0\;\;\;\;\text{ for }i=1,\dots,n.\]
%	We also denote the corresponding representation by $\pi^\lambda$.
We say a $U_q(\mathrm{so}_{2n+1})$-module is of {\it type 2} if it decomposes into a direct sum of $L_q^\tw(\lambda)$'s for $\lambda \in \lP^+$.
Denote the category of type 2 (resp. of type 1) $U_q(\mathrm{so}_{2n+1})$-modules by $C^\tw$ (resp. $C^I$).
For $S=I,\tw$ and $i \in \lP/\lQ=\Z/2\Z$, let $C_i^{S}$ be a full subcategory of $C^S$ consisting of
modules which is isomorphic to a direct sum of $L_q^S(\la)$'s for $\la \in i\la_n+\lQ$.
The following lemma is clear from the definition:
\begin{lem}
Let $\la \in \lQ\cap \lP^+$ and $\la'\in (\la_n+\lQ)\cap \lP^+$. Then,
$L_q^\tw(\la)$ is isomorphic to $L_q(\la)$
and $L_q(\la')\otimes \chi$ and $\chi \otimes L_q(\la')$ are isomorphic to $L_q^\tw(\la)$ as $U_q(\so)$-modules.
\end{lem}
Thus, we can define a functor 
$F: C^I \rightarrow C^\tw$
by $F(M)=(\C\chi^0\otimes  M_0) \oplus (\C\chi \otimes M_1)$
for any $M=M_0\oplus M_1 \in C^I=C_0^I\oplus C_1^I$,
where $\C\chi^0$ is the trivial representation.
Then, $F$ gives an equivalence of abelian categories.

Let $M,N \in C^\tw$. Since $U_q(\so)$ is a Hopf algebra, $M\otimes N$ is a $U_q(\so)$-module
and it is easy to show that $M\otimes N \in C^\tw$.
Thus, $C^\tw$ is naturally a monoidal category.
Let $\rho \in \C$ satisfy $\exp(\pi i \rho)=q$
and denote the braided tensor category $(U_q(\so),R(\rho))\modu$ by $C^I(\rho)$.
In this section, we will prove Conjecture \ref{conj} for type B in three steps:
\begin{enumerate}
\item
To give a braided tensor category structure on $C^\tw$, denoted by $C^\tw(\rho+1)$;
\item
To show that a Hopf algebra isomorphism 
$\phi:U_q(\so)\rightarrow U_{-q}(\so)$ induces an equivalence of braided tensor categories between
$C^\tw(\rho+1)$ and $(U_{-q}(\so),R(\rho+1))\modu$;
\item
To show that $F:C^I(\rho)^{Q_\so}\rightarrow C^\tw(\rho+1)$ gives an equivalence of braided tensor categories.
\end{enumerate}

We will first consider Step (1).
For any type 2 module $M^\tw$,
set for all $\la \in \lP$
$$
M_{\la}^\tw = \{m\in M^\tw \mid K_i m= (-q)^{\lla \la,\al_i \rra}m \text{ for } i=1,\dots,n \}.
$$
Then, we have 
$$M^\tw= \bigoplus_{\la \in \lP} M_{\la}^\tw.$$
In order to define the R-matrix for type 1 representations,
we consider a linear map $f_\rho$ (see Section \ref{sec_bilinear}).
Define for all type 2 $U_q(\so)$-modules $M^\tw$ and $N^\tw$ a bijective linear map
$f_\rho^\tw:M^\tw\otimes N^\tw\rightarrow M^\tw\otimes N^\tw$ by
$$
f_\rho (m\otimes n)=\exp\left(- \pi i (\rho+1) \lla \la,\mu\rra\right) m\otimes n \fora m\in M_\la^\tw \text{ and }n\in N_{\mu}^\tw
$$
and for all $\mu, \la \in \lP$.

Then, a statement similar to Lemma \ref{f_transform} holds for type 2 modules  by replacing $f_\rho$ with $f_\rho^\tw$.
\begin{lem}
\label{ftw_transform}
Let $u \in U_q(\so)_{\mu}^-$ and $u' \in U_q(\so)_{\mu}^+$ for $\mu \in \lQ$ with $\mu \geq 0$.
For any type 2 $U_q(\so)$-modules $M^\tw$ and $N^\tw$,
$$
(f_{\rho}^\tw)^{-1} \circ (u\otimes u') \circ f_{\rho}^\tw =uK_{\mu} \otimes K_{-\mu}u'
$$
as linear maps acting on $M^\tw \otimes N^\tw$.
\end{lem}
Let us define a linear map $R(\rho)^\tw$ by 
$$
R(\rho)^\tw = (\Theta\circ f_\rho^\tw)^{-1}:M^\tw\otimes 
N^\tw\rightarrow M^\tw\otimes N^\tw
$$
for all type 2 $U_q(\so)$-modules $M^\tw$ and $N^\tw$.
Then, by the above lemma,
$R(\rho)^\tw$ satisfies the axiom of an R matrix (R1-R3 in Section \ref{sec_bilinear})
as an operator on $C^\tw$ (see for example \cite[Section 3]{Ja}).
Denote by $C^\tw(\rho+1)$ the braided tensor category defined by $R(\rho)^\tw$.

We next consider Step (2).
%will first consider a Hopf algebra isomorphism
%between $U_q(\mathrm{so}(2n+1))$ and $U_{-q}(\mathrm{so}(2n+1))$. 
%Denote $E_{\al_i},F_{\al_i},K_{\al_i}$ by $E_i,F_i,K_i$.
We remark that among $\al_1,\al_2,\dots,\al_n$, only $\al_n$ is a short root.
\begin{prop}
There exist Hopf algebra isomorphisms
$\phi: U_q(\mathrm{so}(2n+1))\rightarrow U_{-q}(\mathrm{so}(2n+1))$
such that:
\begin{align*}
\phi(E_i)=E_i,\;\;\;\; \phi(F_i)&=F_{i},\;\;\;\; \phi(K_i)=K_i\;\;\;\;\;\; \text{ for }i=1,\dots,n-1
\end{align*}
and
\begin{align*}
\phi(E_n)=- E_n,\;\;\;\; \phi(F_n)&=F_{n},\;\;\;\; \phi(K_n)=K_n.
\end{align*}
\end{prop}
\begin{proof}
The assertion follows from an easy computation.
The point is that $q_{\al_i}=q^{\frac{\lla\al_i,\al_i\rra}{2}}=q^2$ for any $i=1,2,\dots,n-1$
since they are long roots
and $\lla \al,\be \rra \in 2\Z$ for any $\al,\be \in \lQ$ (see Lemma \ref{lem_compatible}).
In particular, $q_{\al} = (-q)_{\al}$ for long roots.
The only non-trivial relation is 
	\[\sum_{r=0}^{1-a_{\al\be}} (-1)^r \left( \begin{matrix} 1-a_{\al\be} \\ r \end{matrix} \right)_{q_\al} E_\al^r E_\be E_\al^{1-a_{\al\be}-r} = 0,\]
for $\al = \al_n$ and $\be =\al_{n-1}$,
which follows from
$
\left( \begin{matrix} 3 \\ 1 \end{matrix} \right)_{-q} 
= 3_{-q}=(-1)^{3+1}3_{q} = \left( \begin{matrix} 3 \\ 1 \end{matrix} \right)_{q}$.
\end{proof}

\begin{rem}
We note that the above proposition is also applicable to the case of $A_1 = B_1$,
that is,
%We note that for $B_1 =A_1$ the above proposition gives us a Hopf algebra isomorphism
\begin{align*}
\phi: U_q(\mathrm{sl}_2)\rightarrow U_{-q}(\mathrm{sl}_2),\;\;\;\;
\phi(E)=-E,\;\;\;\;\;\phi(F)=F,\;\;\;\;\;\; \phi(K)=K.
\end{align*}
It is noteworthy that the Hopf algebras $U_q(\mathrm{sl}_2)$
and $U_{q'}(\mathrm{sl}_2)$ are isomorphic if and only if $q'=\pm q^{\pm}$ (see \cite[Proposition 6 in Section 3]{KS}).
\end{rem}

%\begin{rem}
%\label{rem_short}
%It is only of type B that the sub root system consisting of all short roots is isomorphic to 
%the direct sum of the copies of $A_1$.
%\end{rem}

As shown in Section \ref{sec_isomorphism},
the pullback by $\phi$ induces an equivalence of categories between $U_q(\mathrm{so}(2n+1))\modu$
and $U_{-q}(\mathrm{so}(2n+1))\modu$,
but in general the pullback of a type 1 representation does not necessarily a type 1 representation.

In fact, let $M$ be a type 1 $U_{-q}(\mathrm{so}(2n+1))$-module
and $\phi^*M$ an $U_{q}(\mathrm{so}(2n+1))$-module defined by
$$
a \cdot_\phi m = \phi(a) \cdot m \;\;\;\;\;\;\;\;\;\fora a \in U_{q}(\mathrm{so}(2n+1)) \text{ and }m\in M.
$$
Then, we have
$$K_i \cdot_\phi v = (-q)^{\lla\al_i,\la\rra}v =(-1)^{\lla\al_i,\la\rra} q^{\lla\al_i,\la\rra}v
\text{ for }v \in M_{\la}.
$$
Hence, $\phi^*M$ is a type 2 module and $\phi$ define a functor
$\phi^*: (U_{-q}(\so),R(\rho+1))\modu \rightarrow C^\tw(\rho+1)$.
Then, the following lemma follows from a similar argument in Section \ref{sec_isomorphism}:
\begin{lem}
\label{tw_minus}
The Hopf algebra isomorphism $\phi:U_{q}(\mathrm{so}(2n+1)) \rightarrow U_{-q}(\mathrm{so}(2n+1))$
induces an equivalence between 
$C^\tw(\rho+1)$ and $(U_{-q}(\mathrm{so}(2n+1)),R(\rho+1))\modu$ as balanced braided tensor categories.
\end{lem}

This completes step (1) and step (2). Next, we will show the last step,
that is, the functor $F:C^I\rightarrow C^\tw$ gives a braided tensor equivalence between $(C^I)^{Q_\so}$ and $C^\tw$.
For any $M \in C^I$, define a linear map $h_M:M\otimes \C\chi \rightarrow \C\chi\otimes M$
by
\begin{align*}
h_M(m_\la \otimes \chi)=\exp(\pi i (\la_n,\la))\chi\otimes m_\la
\end{align*}
for any $\la \in \lP$ and $m_\la \in M_\la$.
Then, $h_\bullet$ is a natural transformation by Lemma \ref{natural_h}.

We note that $C^\tw$ is a full subcategory of a strict monoidal category, the category of all $U_q(\so)$-modules, and thus $C^\tw$ is also strict. Hence, we can assume that $\chi^0\otimes M=M$ for any $M \in C^\tw$.
Define a $U_q(\so)$-module isomorphism 
$\epsilon_2:\C\chi\otimes \C\chi \rightarrow \C$ by $\epsilon_2(\chi \otimes \chi)=1$.

Let $M_i \in C_i^I$ and $N_j \in C_j^I$ for $i,j=0,1$.
Define a natural transformation
$g_{M_i,N_j}:F(M_i) \otimes F(N_j) \rightarrow F(M_i\otimes N_j)$
by
\begin{align*}
g_{M_0,N_0}&:
(\chi^0 \otimes M_0)\otimes (\chi^0\otimes N_0)
%\stackrel{\al^{\mathcal{D}}_{F(M),F(N),F(L)}}{\longrightarrow}
= \chi^0 \otimes M_0 \otimes N_0,\\
g_{M_1\otimes N_0}&:
(\chi^1 \otimes M_1)\otimes (\chi^0\otimes N_0)=
\chi^{1}\otimes M_1 \otimes N_0,\\
g_{M_0\otimes N_1}&:
(\chi^0 \otimes M_0) \otimes (\chi^1\otimes N_1)
\stackrel{\id_{\chi^0}\otimes h_{M_0} \otimes \id_{N_1}}{\longrightarrow}
\chi^{1}\otimes M_0\otimes N_1,\\
g_{M_1\otimes N_1}&:
(\chi^1 \otimes M_1)\otimes (\chi^1\otimes N_1)
\stackrel{\id_{\chi^1}\otimes h_{M_1} \otimes \id_{N_1}}{\longrightarrow}
\chi^{2}\otimes M_1\otimes N_1
\stackrel{\epsilon_2\otimes \id_{M_1\otimes N_1}}{\rightarrow}\chi^{0}\otimes M_1\otimes N_1.
\end{align*}
%Here, we identify $\chi^{2k}\otimes M=M$ since the category of all $U_q(\so)$-modules are strict monoidal
%and define $h_{M}^{2k}=\id_{M}: M\otimes \chi^{2k} \rightarrow \chi^{2k} \otimes M$
%for $M \in C^I$ and $k\in \Z$ by $h_{M}^{2k}=\id_{M}$ under the identification $M\otimes \chi^{2k}=M=\chi^{2k} \otimes M$.

Then, we have:
\begin{prop}
\label{prop_tw_monoidal}
The functor $F: C^I\rightarrow C^\tw$ together with the natural transformation 
$g_{M,N}:F(M)\otimes F(N) \rightarrow F(M\otimes N)$
and $\epsilon:\1 = F(\1)$ is a monoidal functor
between $(C^I)^{Q_\so}$ and $C^\tw$.
\end{prop}

Before giving the proof, 
we remark that the value $\exp(\pi i p \lla\la_n,\la_n\rra)=\exp(\frac{\pi i np}{2})$ is not well-defined for
$p \in \Z/2\Z$, which is the source of the 3-cocycle $\al:(\Z/2\Z)^3\rightarrow \C^\times$.
In fact, let $\iota:\Z/2\Z \rightarrow \Z$ be a map defined by
$a\mapsto a$. Then, we have:
%In order to prove this proposition, we need the following lemmas:
\begin{lem}
\label{lem_cocycle_explicit}
The explicit form of the abelian cocycle $(\al_n,c_n) \in Z_\ab^3(\Z/2\Z,\C^\times)$ such that 
$c(a,a)=Q_{\so}(a)$ for $a \in \Z/2\Z=\lP/\lQ$ can be give by
\begin{align*}
\al_n(a,b,c)&=\begin{cases}
(-1)^n & (a=b=c=1)\\
1 & (\text{otherwise})\\
\end{cases}\\
&=\exp\left(\pi i a\frac{\iota(b)+\iota(c)-\iota(a+b)}{2}\right),\\
c_n(a,b)&=\begin{cases}
i^n & (a=b=1) \\
1 & (\text{otherwise}).
\end{cases}
\end{align*}
\end{lem}

\begin{proof}[proof of Proposition \ref{prop_tw_monoidal}]
We will verify the conditions (LM1) and (LM2) in Section \ref{sec_braided}.
Since both $C^I$ and $C^\tw$ are strict monoidal categories,
the associative isomorphisms are trivial before twisting.
Let $p_i \in \Z/2\Z$ and $M_i \in C_{p_i}^I$, $\be_i \in \lP$, and $v_i \in (M_i)_{\be_i}$ for $i=1,2,3$.
%and define $\iota:\lP\rightarrow \Z$ by
%\begin{align*}
%\iota(\ga)=
%\begin{cases}
%0 & (\text{if } \ga \in \lQ),\\
%1 & (\text{if } \ga \in \lP \setminus \lQ).
%\end{cases}
%\end{align*}
Then, we have
\begin{align*}
g_{M_1\otimes M_2,M_3} &\circ (g_{M_1,M_2}\otimes \id_{M_3})
\left((\chi^{p_1} \otimes v_1 \otimes \chi^{p_2} \otimes v_2)\otimes \chi^{p_3} \otimes v_3\right)\\
&=g_{M_1\otimes M_2,M_3}
(\exp(p_2\pi i \lla \la_n,\be_1\rra) (\chi^{p_1+p_2} \otimes v_1 \otimes \otimes v_2) \otimes \chi^{p_3} \otimes v_3)\\
&= \exp(\pi i (p_2\lla \la_n,\be_1\rra+p_3\lla \la_n,\be_1+\be_2 \rra)
 (\chi^{p_1+p_2+p_3} \otimes (v_1 \otimes \otimes v_2) \otimes v_3
\end{align*}
and
\begin{align*}
g_{M_1,M_2\otimes M_3} &\circ (\id_{M_1}\otimes g_{M_2,M_3})
( \chi^{p_1} \otimes v_1 \otimes (\chi^{p_2} \otimes v_2 \otimes \chi^{p_3} \otimes v_3))\\
&=g_{M_1,M_2\otimes M_3}(
\exp(\pi i p_3 \lla \la_n,\be_2 \rra)
 \chi^{p_1} \otimes v_1 \otimes (\chi^{p_2+p_3} \otimes v_2 \otimes v_3)
)\\
&=
\exp(\pi i (p_3 \lla \la_n,\be_2 \rra+\iota(p_2,p_3)\lla \la_n,\be_1\rra)
 \chi^{p_1+p_2+p_3} \otimes v_1 \otimes (v_2 \otimes v_3).
\end{align*}
Thus, in order to verify (LM1), it suffices to show that
\begin{align*}
\al(p_1,p_2,p_3)\exp(\pi i (p_2\lla \la_n,\be_1\rra+p_3\lla \la_n,\be_1+\be_2 \rra)
=\exp(\pi i (p_3 \lla \la_n,\be_2 \rra+\iota(p_2,p_3)\lla \la_n,\be_1\rra),
\end{align*}
which follows from Lemma \ref{lem_cocycle_explicit}.
(LM2) is obvious. Hence, the assertion holds.
\end{proof}

Finally, we will prove that $F:C^I(\rho)^{Q_\so}\rightarrow C^\tw(\rho+1)$ is a braided monoidal functor.
Let $p_i \in \Z/2\Z$ and $M_i \in C_{p_i}^I$, $\be_i \in \lP$, and $m_i \in (M_i)_{\be_i}$ for $i=1,2$.
It suffices to show that
%for any $i,j \in \Z/2\Z$ and $M_i \in C_i^I$ and $N_j \in C_j^I$,
the following diagram commutes:
\begin{align*}
\begin{array}{ccc}
    F(M_1) \otimes_{{\tw}}  F(M_2)
      &\overset{(B_{F(M_1),F(M_2)}^\tw)^{-1}}{\longleftarrow}&
    F(M_2) \otimes_{{\tw}}  F(M_1)
    \\
    {}^{{g_{M_1,M_2}}}\downarrow 
      && 
    \downarrow^{{g_{M_2,M_1}}}
    \\
    F(M\otimes_{{I}} M_2) 
      &\overset{c_n(p_1,p_2)F(B_{M_1,M_2}^I)^{-1}}{\longleftarrow}&
    F(M_2\otimes_{{I}} M_1),
\end{array}
\end{align*}
where $c_n(p_1,p_2)$ is given in Lemma \ref{lem_cocycle_explicit}.
%Let $\al,\be \in \lP$ and $v \in (M_i)_\al$ and $w \in (N_j)_\be$.
Recall $\Theta=\sum_{\mu \geq 0}\Theta_\mu$ and 
$\Theta_\mu=\sum_{i=0}^{r(\mu)}v_i^\mu \otimes u_i^\mu  \in U_\mu^-\hat{\otimes}U_\mu^+$ (see Section \ref{sec_bilinear}).
%Let $\Theta_\mu^1 \in U_\mu^-$ and $\Theta_{-\mu}^2 \in U_\mu^+$ such that
%$\Theta_\mu = \Theta_\mu^1\otimes \Theta_{-\mu}^2$ for short.
Then, we have:
\begin{align*}
g_{M_1,M_2} &\circ (B_{F(M_1),F(M_2)}^\tw)^{-1} (\chi^{p_2} \otimes m_2)\otimes (\chi^{p_1} \otimes m_1)\\
&=g_{M_1,M_2} \circ \Theta \circ f_\rho^\tw\circ P_{21} (\chi^{p_2} \otimes v_2)\otimes (\chi^{p_1} \otimes v_1)\\
&= 
\exp(-\pi i (\rho+1)\lla \be_1,\be_2 \rra) \sum_{\mu \geq 0}\sum_{i=0}^{r(\mu)} 
g_{M_1,M_2}  (v_i^\mu \cdot (\chi^{p_1} \otimes m_1)) \otimes (u_i^\mu \cdot (\chi^{p_2} \otimes m_2)).
\end{align*}
Since $\Delta(E_i)=E_i \otimes 1 + K_i\otimes E_i$ and $\Delta(F_i)=F_i\otimes K_i^{-1}+1\otimes F_i$,
we have
\begin{align*}
(v_i^\mu \cdot (\chi^{p_1} \otimes m_1)) \otimes (u_i^\mu \cdot (\chi^{p_2} \otimes m_2))
&=(\chi^{p_1} \otimes v_i^\mu \cdot m_1) \otimes (K_\mu\cdot \chi^{p_2} \otimes u_i^\mu  \cdot m_2)\\
&=(-1)^{p_2 \lla \la_n,\mu \rra}
(\chi^{p_1} \otimes v_i^\mu \cdot m_1) \otimes (\chi^{p_2} \otimes u_i^\mu  \cdot m_2).
\end{align*}
Hence, we have:
\begin{align*}
&g_{M_1,M_2}\circ (B_{F(M_1),F(M_2)}^\tw)^{-1} (\chi^{p_2} \otimes m_2)\otimes (\chi^{p_1} \otimes m_1)\\
&= \exp(\pi i (\rho+1)\lla \be_1,\be_2 \rra) \sum_{\mu \geq 0}\sum_{i=0}^{r(\mu)} 
\exp(\pi i \iota(p_2) \lla \la_n, \be_1+\mu \rra)
(-1)^{p_2\lla \la_n,\mu \rra}
(\chi^{p_1+p_2} \otimes v_i^\mu \cdot m_1 \otimes u_i^\mu  \cdot m_2)\\
&= \exp(\pi i (\rho+1)\lla \be_1,\be_2 \rra)
\exp(\pi i \iota(p_2) \lla \la_n, \be_1\rra)
 \sum_{\mu \geq 0}\sum_{i=0}^{r(\mu)} 
(\chi^{p_1+p_2} \otimes v_i^\mu \cdot m_1 \otimes u_i^\mu  \cdot m_2).
\end{align*}
Similarly, we have
\begin{align*}
&c_n(p_1,p_2)F(B_{M_1,M_2}^I)^{-1} \circ g_{M_1,M_2} (\chi^{p_2} \otimes m_2)\otimes (\chi^{p_1} \otimes m_1)\\
&=c_n(p_1,p_2)\exp(\pi i \iota(p_1)\lla \la_n, \be_2 \rra) F(B_{M_1,M_2}^I)^{-1}
(\chi^{p_1+p_2} \otimes m_2 \otimes m_1)\\
&=c_n(p_1,p_2)\exp(\pi i \iota(p_1)\lla \la_n, \be_2 \rra) 
\chi^{p_1+p_2} \otimes 
\left( 
\Theta \circ f_\rho \circ P_{21} (m_2 \otimes m_1)\right)\\
&=c_n(p_1,p_2)\exp(\pi i \iota(p_1)\lla \la_n, \be_2 \rra) 
\exp(-\pi i \rho (\be_1,\be_2))
\sum_{\mu \geq 0}\sum_{i=0}^{r(\mu)}
\chi^{p_1+p_2} \otimes 
(v_i^\mu \cdot m_1 \otimes u_i^\mu m_2).
\end{align*}
Thus, the proof of the conjecture comes down to the following Lemma:
\begin{lem}
If $(M_i)_{\be_i} \neq 0$ for $i=1,2$,
then
\begin{align*}
c_n(p_1,p_2)= \exp(\pi i (\lla \be_1,\be_2 \rra + \iota(p_2) \lla \la_n, \be_1\rra - \iota(p_1)\lla \la_n, \be_2 \rra).
\end{align*}
\end{lem}
\begin{proof}
Let $k:\lP\times \lP \rightarrow \C^\times$ be a map defined by 
$k(\be_1,\be_2)=\exp(\pi i (\lla \be_1,\be_2 \rra + \iota(\be_2) \lla \la_n, \be_1\rra - \iota(\be_1)\lla \la_n, \be_2 \rra)$, where $\iota:\lP\rightarrow \lP/\lQ = \{0,1\}$ is defined by the composition of the projection
and the identification.

We claim that 
$k(\be_1+\al,\be_2)=k(\be_1,\be_2+\al)=k(\be_1,\be_2)$ for any $\al \in \lQ$.
The difference $k(\be_1+\al,\be_2)k(\be_1,\be_2)^{-1}$
is equal to $\exp(\pi i (\lla \be_2,\al \rra+\iota(\be_2)\lla \la_n,\al \rra))$.
Thus, if $\be_2 \in \lQ$ i.e., $\iota(\be_2)=0$, then 
$k(\be_1+\al,\be_2)k(\be_1,\be_2)^{-1}$ is equal to $1$ by Lemma \ref{lem_compatible}.
Similarly, if $\iota(\be_2)=1$, then 
$k(\be_1+\al,\be_2)k(\be_1,\be_2)^{-1}
=\exp(\pi i (\lla \be_2,\al \rra+\lla \la_n,\al \rra))
=\exp(\pi i (\lla \la_n,\al \rra+\lla \la_n,\al \rra))=1,$
thus the claim is proved.

Since $k(0,0)=k(\la_n,0)=k(0,\la_n)=1$ and $k(\la_n,\la_n)=i^n$,
the assertion follows from Lemma \ref{lem_cocycle_explicit}.
\end{proof}
Hence, we have:
\begin{thm}
\label{thm_B}
The composition of $F$ and $\phi^*$ gives a braided monoidal equivalence
between $(U_q(\so),R(\rho))\modu^{Q_\so}$
and $(U_{-q}(\so),R(\rho))\modu$ for any $n \geq 1$.
\end{thm}

\end{document}